\documentclass[oneside,english]{amsart}
\usepackage[T1]{fontenc}
\usepackage[utf8]{inputenc}
\usepackage{amsthm}
\usepackage{amstext}
\usepackage{amssymb}
\usepackage{bbm}
\usepackage{mathabx}
\usepackage[a4paper]{geometry} 
\usepackage{enumitem}

\usepackage{color}

\makeatletter

\newtheorem{thm}{Theorem}
\newtheorem{lem}{Lemma}
\newtheorem{prop}{Proposition}

\newcommand*{\ensembledenombres}{\mathbb}
\newcommand*{\R}{\ensembledenombres{R}}
\newcommand*{\C}{\ensembledenombres{C}}
\newcommand*{\N}{\ensembledenombres{N}}

\newcommand*{\esp}{\mathbb{E}}
\DeclareMathOperator{\Var}{Var}

\DeclareMathOperator{\Tr}{Tr}

\newcommand*{\CV}[1]{\underset{#1 \to +\infty}{\longrightarrow}}

\renewcommand{\Im}{\mathcal{I}}

\begin{document}

\title[Fluctuations of LSS of deformed Wigner matrices]{Fluctuations of linear spectral statistics of deformed Wigner matrices}

\author{Sandrine Dallaporta}
\address{Centre de mathématiques et de leurs applications, CNRS, ENS Paris-Saclay, Université Paris-Saclay, 94235, Cachan cedex, France.}
\address{Laboratoire de Mathématiques et Applications, UMR-CNRS 7348, Université de Poitiers, Téléport 2-BP30179, Boulevard Marie et Pierre Curie, 86962 Chasseneuil, France.}
\email{sandrine.dallaporta@math.univ-poitiers.fr}

\author{Maxime Fevrier}
\address{Universit\'e Paris-Saclay, CNRS, Laboratoire de mathématiques d’Orsay, 91405, Orsay, France.}
\email{maxime.fevrier@universite-paris-saclay.fr}

\thanks{}

\begin{abstract}
We investigate the fluctuations of linear spectral statistics of a Wigner matrix $W_N$ deformed by a deterministic diagonal perturbation $D_N$, around a deterministic equivalent which can be expressed in terms of the free convolution between a semicircular distribution and the empirical spectral measure of $D_N$. We obtain Gaussian fluctuations for test functions in $\mathcal{C}_c^7(\R)$ ($\mathcal{C}_c^2(\R)$ for fluctuations around the mean). Furthermore, we provide as a tool a general method inspired from Shcherbina and Johansson to extend the convergence of the bias if there is a bound on the bias of the trace of the resolvent of a random matrix. Finally, we state and prove an asymptotic infinitesimal freeness result for independent GUE matrices together with a family of deterministic matrices, generalizing the main result from \cite{Shlyakhtenko18}.
\end{abstract}

\maketitle

\section{Introduction}

The celebrated Wigner's Theorem \cite{Wigner58} states that the empirical spectral measure 
(i.e. the uniform distribution on the eigenvalues) $\mu_{W_N}$  of a suitably rescaled 
Hermitian matrix $W_N$ with independent entries having mean $0$ and variance $\sigma^2$, 
now known as a Wigner matrix, weakly converges in probability to the semicircular distribution $\mu_{\sigma^2}$ with density $(2\pi\sigma^2)^{-1}\sqrt{4\sigma^2-x^2}\mathbf{1}_{[-2\sigma;2\sigma]}(x)$.
It is remarkable that the limit distribution is non random and universal, in the sense that it depends 
on the distribution of the entries only through their variance $\sigma^2$. 
Hence, for every test function $\varphi:\mathbb{R} \to \mathbb{C}$ in a set 
including bounded continuous functions and indicator functions of intervals, 
the linear spectral statistic $\int_{\mathbb{R}}\varphi(x)\mu_{W_N}(dx)$
converges in probability to $\int_{\mathbb{R}}\varphi(x)\mu_{\sigma^2}(dx)$.

It is then natural to investigate fluctuations of $\int_{\mathbb{R}}\varphi(x)\mu_{W_N}(dx)$ around its limit 
$\int_{\mathbb{R}}\varphi(x)\mu_{\sigma^2}(dx)$. This question has attracted a lot of attention in the past decades. 
For Wigner matrices with Gaussian entries, results were obtained by a careful analysis of the explicit density 
of the eigenvalue point process \cite{Johansson98}, exploiting its determinantal structure \cite{CosLeb95},
or by a dynamical approach using stochastic calculus \cite{Cabanal01}. For more general entries, 
the question was first attacked in \cite{KKP96} for particular test functions 
$\varphi_z:x\mapsto (z-x)^{-1},\, z\in \C\setminus \R$, for which the linear spectral statistic 
is the Stieltjes transform of $\mu_{W_N}$. Recall here the definition of the Stieltjes transform 
$G_{\mu}$ of a Borel probability measure $\mu$ on $\mathbb{R}$:
$$G_{\mu}(z)=\int_{\mathbb{R}}\varphi_z(x)\mu(dx), \,  z\in\mathbb{C}\setminus\mathbb{R}.$$
See also \cite{BaiYao05,BaoXie16,BGMal16} for similar investigations.
Other approaches, for polynomial test functions using combinatorial arguments in \cite{SiSo98b} 
or for test functions with sufficiently fast decaying Fourier transform using Fourier analysis 
in \cite{LytPas09a,LytPas09b}, have been developed.

It appeared that these fluctuations depend on more moments (up to the fourth) of the entries 
of the Wigner matrix and on the regularity of the test function. 
When entries have the same finite fourth moment and the test function has enough regularity, 
fluctuations are of scale $N^{-1}$ and of Gaussian nature. More precisely, 
the following result is a reformulation of the main result in \cite{BaoXie16}:

\begin{thm}\label{BaoXieThm}
Let $W_N$ be a Wigner matrix satisfying the assumptions \ref{hyp:indep}, \ref{hyp:offdiagonal}, \ref{hyp:diagonal} in Section \ref{Model} below.
Then, for analytic $\varphi : \mathbb{R} \to \mathbb{R}$, the linear spectral statistic
$\int_{\mathbb{R}}\varphi(x)\mu_{W_N}(dx)$ of $W_N$ satisfies the following:
\[ N\Big(\int_{\mathbb{R}}\varphi(x)\mu_{W_N}(dx)-\int_{\mathbb{R}}\varphi(x)\mu_{\sigma^2}(dx)\Big) \Rightarrow \mathcal{N}\big(b_0(\varphi),V_0(\varphi)=C_0(\varphi,\varphi)\big),\]
where $b_0$ and $C_0$ are uniquely determined by 
$$b_0(\varphi_z)=-G_{\mu_{\sigma^2}}'(z)G_{\mu_{\sigma^2}}(z)\big(s^2-\sigma^2+\tau^2\frac{G_{\mu_{\sigma^2}}(z)^2}{1-\tau G_{\mu_{\sigma^2}}(z)^2}+\kappa G_{\mu_{\sigma^2}}(z)^2\big),\quad z\in \C\setminus \R;$$
\begin{align*}
C_0(\varphi_{z_1},\varphi_{z_2}) & =G_{\mu_{\sigma^2}}'(z_1)G_{\mu_{\sigma^2}}'(z_2)\Big[s^2-\sigma^2-\tau+2\kappa G_{\mu_{\sigma^2}}(z_1)G_{\mu_{\sigma^2}}(z_2)\\
& + \frac{\sigma^2}{(1-\sigma^2G_{\mu_{\sigma^2}}(z_1)G_{\mu_{\sigma^2}}(z_2))^2}+\frac{\tau}{(1-\tau G_{\mu_{\sigma^2}}(z_1)G_{\mu_{\sigma^2}}(z_2))^2}\Big],\quad z_1,z_2\in \C\setminus \R.
\end{align*}
\end{thm}

For discontinuous test functions such as indicator functions of intervals, fluctuations 
are still Gaussian but with different scale, mean and variance \cite{CosLeb95,DalVu11,LanSos}. 
Note that the optimal regularity assumption on test functions for Theorem \ref{BaoXieThm} to hold is still an active field of research \cite{BaiWangZhou09,Shcherbina11,SosWon13,Kopel}. 
Gaussian fluctuations with different scale, mean and variance also hold for linear spectral statistics of the form
$\int_{\mathbb{R}}\varphi_z(x)\mu_{W_N}(dx)$ 
when the entries of the Wigner matrix have an infinite fourth moment 
(\cite{BGMal16}; see also \cite{BGGuiMal14} for the case of non square integrable entries, 
in which case Wigner's Theorem fails to hold \cite{BAGui08}). 
When entries of the Wigner matrix are not identically distributed in such a way that 
their variances differ (these matrices are called band matrices or sometimes Wigner matrices with variance profile), 
fluctuations of linear spectral statistics have also been described (see \cite{AdhJanSah} and references therein). 

In this paper, we focus on deformed Wigner matrices, i.e. sums $X_N$ of a Wigner matrix $W_N$ and 
a deterministic Hermitian matrix $D_N$ whose empirical spectral measure $\nu_N$ 
weakly converges to a Borel probability measure $\nu_{\infty}$. 
They were introduced by \cite{PorRoz60} as a generalization of Wigner's model for energy levels of nuclei, 
and had several other applications afterwards. The weak convergence of 
the empirical spectral measure $\mu_N$ of $X_N$ was first proved by Pastur in \cite{Pastur72} (see also \cite{PasShchbook}): 
$\mu_N$ weakly converges (in probability) towards the unique Borel probability measure $\rho $ on $\mathbb{R}$
satisfying the so-called {\em Pastur equation}:
\begin{equation}\label{Pastur}
G_{\rho}(z)=G_{\nu_{\infty}}(z-\sigma^2G_{\rho}(z)), \,  z\in\mathbb{C}\setminus\mathbb{R}.
\end{equation}

The empirical spectral measure $\mu_N$ of $W_N+D_N$ is related via its moments to the noncommutative distribution of $(W_N,D_N)$ in the noncommutative probability space of random matrices with entries having finite moments of any order $(\mathcal{M}_N(L^{\infty-}),\mathbb{E}N^{-1}\Tr )$. Under some additional assumptions, $W_N$ and $D_N$ have been proved in \cite{Dykema93, AGZbook, MinSpebook} to be asymptotically free, in the sense of free probability theory (see \cite{VDN92} for an introduction to free probability theory). As a consequence, the weak limit $\rho$ of $\mu_N$ is the distribution of the sum of free selfadjoint noncommutative random variables respectively distributed according to the limiting empirical spectral measures of $W_N$ and $D_N$. In other words, $\rho$ is the free additive convolution of the semicircular distribution $\mu_{\sigma^2}$ and $\nu_{\infty} $. Voiculescu noticed in \cite{Voiculescu93} that the Stieltjes transform of the free additive convolution $\mu \boxplus \nu$ of two Borel probability measures $\mu, \nu$ on $\mathbb{R}$ is generically subordinated (in the sense of Littlewood) to the Stieltjes transform of $\nu$: there exists an analytic self-map $\omega : \mathbb{C}^+ \to \mathbb{C}^+$ of the upper half-plane $\mathbb{C}^+:=\{z\in \mathbb{C},\, \Im z>0\}$ such that $G_{\mu \boxplus \nu}(z)=G_{\nu}(\omega (z)),\, z\in \mathbb{C}^+$. When $\mu =\mu_{\sigma^2}$ is the semicircle distribution with variance $\sigma^2$ and $\nu=\nu_{\infty}$, then $\omega(z)=z-\sigma^2G_{\rho}(z)$ and the subordination equation coincides with Pastur equation \eqref{Pastur}. 

The fluctuations of linear spectral statistics of deformed Wigner matrices were studied in \cite{Khorunzhy94} in the case of deformed GOE (real Gaussian entries) 
via the resolvent approach and discussed in \cite{Guionnet02} in the case of deformed GUE (complex Gaussian entries) 
via a dynamical approach. See also \cite{Su13} and \cite{BGEnrMic} for related works 
where $W_N$ and $D_N$ have different scales. More recently, deformed Wigner matrices 
with more general entries but for rank one deformations $D_N$ were considered, 
in relation with the free energy of the spherical Sherrington-Kirkpatrick model \cite{BaikLee17} or with 
statistical applications \cite{ChuLee}. During the preparation of this paper, the paper \cite{JiLee} appeared, 
dealing with full rank deformations of Wigner matrices; we discuss the differences between \cite{JiLee} and our work 
at the end of the Introduction.

In this paper, we will consider deformed Wigner matrices with general entries and
deterministic diagonal deformations $D_N$. We study fluctuations of 
$\int_{\mathbb{R}}\varphi(x)\mu_N(dx)$ around a deterministic equivalent $\int_{\mathbb{R}}\varphi(x)\rho_N(dx)$, 
where the Borel probability measure $\rho_N$ is defined, as $\rho$, by Pastur equation \eqref{Pastur}, 
but with $(N\sigma_N^2,\nu_N)$ instead of $(\sigma^2,\nu_{\infty})$, 
and for test functions $\varphi$ in the space $\mathcal{H}_s$ of all $f \in L^2(\R)$ such that 
\[\|f\|_{\mathcal{H}_s}:=\Big(\int_{\R}(1+2|t|)^{2s}|\hat{f}(t)|^2dt\Big)^{1/2}<+\infty, \]
for some $s>0$.
Our main result in this paper is the following:

\begin{thm} \label{main}
	Let $X_N$ be the deformed Wigner matrix satisfying assumptions \ref{hyp:indep}, \ref{hyp:offdiagonal}, \ref{hyp:diagonal}, \ref{hyp:deformation} in Section \ref{Model} below.
	\begin{enumerate}[label=(\arabic*)]
		\item\label{main1} For real-valued $\varphi \in \mathcal{H}_s,\, s>3/2$, the linear spectral statistics $\int_{\mathbb{R}}\varphi(x)\mu_N(dx)$ of $X_N$ satisfies:
		\[ N\Big(\int_{\mathbb{R}}\varphi(x)\mu_N(dx)-\mathbb{E}[\int_{\mathbb{R}}\varphi(x)\mu_N(dx)]\Big) \Rightarrow \mathcal{N}\big(0,V[\varphi]=C(\varphi,\varphi)\big),\]
		where the continuous Hermitian bilinear form $C:\mathcal{H}_s\times \mathcal{H}_s\to \mathbb{C}$ is determined by $C(\varphi_{z_1},\varphi_{z_2})=\Gamma(z_1,z_2),\, z_1,z_2\in \mathbb{C}\setminus \mathbb{R}$ (defined in Proposition \ref{hook}).
		\item\label{main2} For real-valued $\varphi \in \mathcal{H}_s,\, s>13/2$, $N\big(\mathbb{E}[\int_{\mathbb{R}}\varphi(x)\mu_N(dx)]-\int_{\mathbb{R}}\varphi(x)\rho_N(dx)\big) \underset{N\to +\infty}\longrightarrow b(\varphi)$,
		where the continuous linear form $b:\mathcal{H}_s \to \mathbb{C}$ is determined by its restriction to $\{\varphi_z:x\mapsto (z-x)^{-1}; z\in \C\setminus \R\}$ (given in Proposition \ref{convb}).
	\end{enumerate}
\end{thm}

The Borel probability measure $\rho_N$ defined above has a nice interpretation in terms of free probability:
 it is the free additive convolution of the semicircular distribution $\mu_{N\sigma_N^2}$ and $\nu_N$.
Theorem \ref{main} also has interpretations in terms of (extensions of) free probability theory, 
namely higher order freeness (see \cite{MinSpe06, MinSniSpe07, CMSS07}) and infinitesimal freeness (see \cite{BelShl12}). 
Indeed, roughly speaking, when one considers normalized traces of mixed words in certain $N\times N$  random matrices, 
whereas freeness is related to the large $N$ limit of their expectations (or first cumulants), 
infinitesimal freeness is related to the $1/N$ correction in this asymptotic, 
and higher order freeness is related to the large $N$ asymptotic of their higher cumulants. 
It is known \cite{CMSS07} that a GUE random matrix $W_N$ is asymptotically free of all orders from 
any deterministic Hermitian matrix $D_N$ whose empirical spectral measure converges in moments. 
It follows that, under these assumptions, the fluctuations of $\int_{\mathbb{R}}\varphi(x)\mu_N(dx)$ 
for polynomial $\varphi$ are Gaussian and their variance may be computed using the $R$-transform machinery from \cite{CMSS07}
(note that $\Gamma(z_1,z_2)$ from Proposition \ref{hook} with $\tau=\kappa=0$ and $s^2=\sigma^2$ 
indeed coincides with the second-order Cauchy transform that may be computed from the $R$-transform machinery).
Theorem \ref{main} suggests that asymptotic second order freeness still holds between $W_N$ and $D_N$ 
when $W_N$ is a Wigner matrix satisfying \ref{hyp:indep}, \ref{hyp:offdiagonal} and \ref{hyp:diagonal} with $\tau=\kappa=0$ and $s^2=\sigma^2$. 
Similarly, a random matrix $W_N$ whose distribution is invariant under the action (by conjugation) of the orthogonal group is asymptotically real second order free from 
any deterministic real symmetric matrix $D_N$ whose empirical spectral measure converges in moments \cite{MinPop13, Redelmeier14}. 
Theorem \ref{main} suggests that asymptotic real second order freeness still holds between $W_N$ and $D_N$ when $W_N$ is a Wigner matrix satisfying \ref{hyp:indep}, \ref{hyp:offdiagonal} and \ref{hyp:diagonal} with $s^2=\sigma^2$ and $\kappa=0$.
Another corollary of Theorem \ref{main} is that the bias $b(\varphi)$ vanishes 
whenever $s^2=\sigma^2, \tau=0$ and $\kappa=0$ (this is for instance the case for the deformed GUE). 
In that case, the mean empirical spectral measure $\mathbb{E}[\mu_N]$ is particularly well approximated 
by the free additive convolution $\rho_N$ of the semicircular distribution $\mu_{N\sigma_N^2}$ and $\nu_N$. 
This suggests that a Wigner matrix satisfying \ref{hyp:indep}, \ref{hyp:offdiagonal} and \ref{hyp:diagonal} with 
$s^2=\sigma^2$, $\tau=0$ and $\kappa=0$ 
is asymptotically infinitesimally free from bounded sequences of deterministic real diagonal matrices $D_N$
converging in noncommutative distribution. This was known for the GUE and a family of finite rank deterministic matrices 
(see \cite{Shlyakhtenko18}); we generalize this result in appendix.

Due to the simple characterization of $\rho$ and $\rho_N$ in terms of their Stieltjes transforms, 
it is natural to prove Theorem \ref{main} first for test functions 
$\varphi_z,\, z\in \mathbb{C}\setminus \mathbb{R}$, following the strategy of \cite{BaiYao05}. 
The linear spectral statistic associated to the test function 
$\varphi_z$ is related to the trace of the resolvent of $X_N$.
To deal with resolvents, we will use elementary linear algebra tools, 
including various consequences of Schur inversion formula. 
We will then use martingale arguments such as a classical CLT for martingale differences. 
Then we extend the result to test functions in some $\mathcal{H}_s$ by a density argument inspired by \cite{Shcherbina11}, 
justifying the statement made in Theorem \ref{main} that $b$ and $V$ are determined by 
their restrictions to the linear span of $\varphi_z,\, z\in \mathbb{C}\setminus \mathbb{R}$. 

The recent paper \cite{JiLee} addresses questions similar to the one studied in this work. 
Note that \cite{JiLee} also considers the case of a random diagonal deformation $D_N$ 
with independent identically distributed diagonal entries, independent from $W_N$; 
this case is not strictly in the scope of our work. In the case of a deterministic diagonal deformation, 
this paper extends results from \cite{JiLee} by relaxing some of their assumptions. 
First, the proofs are written for complex Hermitian as well as for real symmetric matrices. 
We ask the entries of $W_N$ to have a finite $4(1+\varepsilon)$ moment 
and, in the complex case, remove the assumption $\esp[W_{ij}^2]=0$ for $i\neq j$.
We assume instead that the real and imaginary parts of $W_{ij}$ are uncorrelated. 
Second, the result is proved under no assumption on the support of the empirical spectral measure of $D_N$ (for instance, existence of outlying eigenvalues or disconnected limiting support are allowed): this assumption does not seem necessary because the almost sure (exact) separation of the spectrum of deformed Wigner matrices (see \cite{CDFF11}) prevents it from the usual objection to Gaussian fluctuations 
for matrix models whose limiting empirical spectral measure has disconnected support.
Third, it is valid for functions in $\mathcal{C}_c^7(\R)$ ($\mathcal{C}_c^2(\R)$ for fluctuations around the mean). 
Technically, these improvements are obtained by bypassing the use of the precise local law from \cite{LSSY16}.
Fluctuations are extended to non analytic test functions by a density argument, proved in \cite{Shcherbina11} and very recently used by \cite{JiLee} for fluctuations around the mean. We adapt it to study the convergence of bias (Lemma \ref{lem:extension_bias} and Proposition \ref{prop_bound_bias_Hs}).

Besides the Introduction, the paper is organized in several other sections. Section \ref{Model} introduces the 
deformed Wigner matrices considered in this work. Section \ref{sec:preliminary_results} 
gathers tools from elementary linear algebra, martingale theory, complex analysis and functional analysis used in the proofs. 
Section \ref{concentration} is devoted to concentration bounds that are central in our approach. 
The proof of item \ref{main2} of Theorem \ref{main} can be found in Section \ref{bias}; 
the proof of item \ref{main1} of Theorem \ref{main} is the content of Sections \ref{CLT} and \ref{extension}. 
Two appendices conclude the paper: the first one details the truncation argument allowing to assume 
that entries of $W_N$ are almost surely bounded by a sequence slowly converging to $0$; 
the second one states and proves the free probabilistic interpretation of item $(2)$ of Theorem \ref{main}.

\section{Presentation of the model} \label{Model}

We consider, on a probability space, a sequence of deformed Wigner matrices $$X_N:=W_N+D_N,\quad N\geq 1,$$ where : 
\begin{enumerate}[label=(H\arabic*)]
	\item\label{hyp:indep} entries $\{W_{ij}\}_{1\leq i\leq j\leq N}$ of the $N\times N$ Hermitian matrix $W_N$ are independent random variables;
	\item\label{hyp:offdiagonal} off-diagonal entries $\{W_{ij}\}_{1\leq i< j\leq N}$ of $W_N$ are identically distributed complex random variables such that, for some $\varepsilon >0$, $C_4:=\sup_{N\geq 1}\mathbb{E}[|\sqrt{N}W_{ij}|^{4(1+\varepsilon)}]<+\infty$. We assume that $\mathbb{E}[W_{ij}]=0$ and that
	$$\sigma_N^2:=\mathbb{E}[|W_{ij}|^2]\geq 0,\quad \tau_N:=\mathbb{E}[W_{ij}^2]\in \mathbb{R},\quad \kappa_N:=\mathbb{E}[|W_{ij}|^4]-2\sigma_N^4-\tau_N^2\in \mathbb{R}$$
	satisfy 
	$$\lim_{N\to +\infty}N\sigma_N^2=\sigma^2> 0,\quad \lim_{N\to +\infty}N\tau_N=\tau\in \mathbb{R},\quad \lim_{N\to +\infty}N^2\kappa_N=\kappa\in \mathbb{R};$$
	\item\label{hyp:diagonal} diagonal entries $\{W_{ii}\}_{1\leq i\leq N}$ of $W_N$ are identically distributed real random variables such that, for some $\varepsilon >0$, $C_2:=\sup_{N\geq 1}\mathbb{E}[|\sqrt{N}W_{ii}|^{2(1+\varepsilon)}]<+\infty$. We assume that $\mathbb{E}[W_{ii}]=0$ and that
	$s_N^2:=\mathbb{E}[W_{ii}^2]\geq 0$ satisfies $\lim_{N\to +\infty}Ns_N^2=s^2\geq 0$;
	\item\label{hyp:deformation} $D_N$ is a $N\times N$ deterministic real diagonal matrix 
	such that, for some Borel probability measure $\nu_{\infty}$ on $\mathbb{R}$,
	$$\nu_N:=\frac{1}{N}\sum_{\lambda \in \text{sp}(D_N)}\delta_{\lambda} \Rightarrow \nu_{\infty} .$$
\end{enumerate}
In subsections \ref{subsectionconvb}, \ref{subsectionreduction} and \ref{subsectionhook}, 
we will also assume that all entries of $W_N$ are almost surely bounded by $\delta_N$, 
where $(\delta_N)_{N\geq 1}$ is a sequence of positive numbers slowly converging to $0$ 
(at rate less than $N^{-\eta}$ for any $\eta>0$);
this may be assumed without loss of generality, as shown in Appendix \ref{Truncation}. 
We will use the notation $m_N:=\esp[|W_{ij}|^4]=\kappa_N+2\sigma_N^4+\tau_N^2$.
In assumptions \ref{hyp:offdiagonal} and \ref{hyp:diagonal}, we ask the entries to be identically distributed. 
This assumption does not seem to be necessary for our main result to hold, but leads to a simplification of 
our truncation-centering argument (see Appendix \ref{Truncation}). Therefore, for the readability of the paper, 
we will not pursue the task to relax this assumption. Our assumption $\tau_N\in \mathbb{R}$ 
means that the real and imaginary parts of off-diagonal entries of $W_N$ are uncorrelated 
(but with possibly different variances).

We are interested in the empirical spectral measure $\mu_N$ of $X_N$, defined by: 
$$\mu_N:=\frac{1}{N}\sum_{\lambda \in \text{sp}(X_N)}\delta_{\lambda}.$$
More precisely, we study the fluctuations of the linear spectral statistic $\int_{\mathbb{R}}\varphi(x)\mu_N(dx)$ around its deterministic equivalent  $\int_{\mathbb{R}}\varphi(x)\rho_N(dx)$, for functions $\varphi: \R \to \R$ with enough regularity.

The linear spectral statistic associated to the test function 
$\varphi_z$ is the normalized trace of $$R_N(z):=(zI_N-X_N)^{-1}.$$
We will use Schur inversion formula, relating $R_N(z)$ to 
its one-dimensional Schur complements $$z-W_{kk}-D_{kk}-C_k^{(k)*}R^{(k)}(z)C_k^{(k)},$$
where $R^{(k)}$ is the resolvent of the $(N-1)\times(N-1)$ obtained from $X_N$ 
by deleting the $k$-th row/column and $C_k^{(k)}$  is the $(N-1)$-dimensional vector 
obtained from the $k$-th column of $W_N$ by deleting its $k$-th component. 
Martingales appearing in this paper will be with respect to the filtration 
$(\mathcal{F}_k:=\sigma(W_{ij},1\leq i\leq j\leq k))_{k\geq 1}$; 
$\mathbb{E}_{\leq k}$ denotes the conditional expectation on the sigma-field $\mathcal{F}_k$ 
and $\mathbb{E}_{k}$ the expectation with respect to the $k$-th column $\{W_{ik},1\leq i\leq N\}$ of $W_N$. 

\section{Preliminary results}\label{sec:preliminary_results}
In this section, we gather properties which will be used several times in the sequel.
\subsection{Linear algebra}
We start this subsection by recalling well-known properties of Schur complement. 

\begin{prop}[Schur complement]\label{Schur} Let $A\in \mathcal{M}_n({\mathbb{C}})$ 
and its submatrix $B$ obtained by removing its $k$-th diagonal entry $A_{kk}$, 
what remains of its $k$-th column $c$ and of its $k$-th row $r$. 
Then, if $A$ and $B$ are invertible, $A_{kk}-rB^{-1}c\neq 0$ and the following formulas hold:
$$(A^{-1})_{kk}=\frac{1}{A_{kk}-rB^{-1}c};\quad \Tr(A^{-1})-\Tr(B^{-1})=\frac{1+rB^{-2}c}{A_{kk}-rB^{-1}c}.$$
\end{prop}

We will apply Proposition \ref{Schur} to express diagonal entries and trace 
of the resolvent of the Hermitian matrix $X_N$. 
The resolvent of a $n\times n$ matrix with complex entries 
$A\in \mathcal{M}_n(\mathbb{C})$ is the map $z\mapsto (zI_n-A)^{-1}$ 
defined on $\mathbb{C}\setminus \text{sp}(A)$ and satisfying 
$\big((zI_n-A)^{-1}\big)^*=(\overline{z}I_n-A^*)^{-1},\, z\in \mathbb{C}\setminus \text{sp}(A)$. 
It follows that the resolvent of a normal $n\times n$ matrix takes its values 
in the set of normal $n\times n$ matrices. In particular, the resolvent $R$ 
of a $n\times n$ Hermitian matrix $M$ is defined on $\mathbb{C}\setminus \mathbb{R}$ and satisfies:
\begin{equation}\label{resolventbound}
\| R(z)\| \leqslant \frac{1}{|\mathcal{I}z|},\quad z\in \mathbb{C}\setminus \mathbb{R}.
\end{equation}
The following Lemma is elementary but useful. 
\begin{lem}[Resolvent identity]\label{lem_resolvent_identity}
Let $M_1$ and $M_2$ be $n\times n$ Hermitian matrices and denote by $R_1$ and $R_2$ their respective resolvents. Then, for all $z_1,z_2 \in \C\setminus \R$, 
\[R_1(z_1)-R_2(z_2)=R_1(z_1)\Big((z_2-z_1)I_n+M_1-M_2\Big)R_2(z_2).\]
\end{lem}
As a consequence, the resolvent $R$ of a $n\times n$ Hermitian matrix $M$ satisfies:
\begin{equation}\label{im}
\mathcal{I}\psi(R(z))=-\mathcal{I}z\psi(R(z)^*R(z)),\, z\in \mathbb{C}\setminus \mathbb{R},
\end{equation}
where $\psi:\mathcal{M}_n(\mathbb{C})\to \mathbb{C}$ is any selfadjoint linear functional. 
Another immediate consequence of Lemma \ref{lem_resolvent_identity} 
is the following relation between the resolvents $R$ and $R^{(ab)}$ of a $n\times n$ Hermitian matrix $M$ 
and of the Hermitian matrix $M^{(ab)}$ obtained from $M$ by replacing its $(a,b)$ and $(b,a)$ entries by $0$:
\begin{equation}\label{remove}
R^{(ab)}(z)-R(z)=R^{(ab)}(z)(1-\frac{1}{2}\delta_{ab})(M_{ab}E_{ab}+\overline{M_{ab}}E_{ba})R(z).
\end{equation}

\begin{lem}\label{technicalbound}
Let $1 \leq k\leq n$ and $A,B\in \mathcal{M}_n(\mathbb{C})$. Define $C\in \mathcal{M}_n(\mathbb{C})$ by $C_{ij}=\sum_{m<k}A_{im}B_{mj}$, for $1 \leq i,j\leq n$.
Then $\|C\|\leq \| A\| \| B\|$. In particular, 
\[|C_{ij}|=\big|\sum_{m<k}A_{im}B_{mj}\big| \leq \| A\| \| B\|,\quad 1 \leq i,j\leq n;\]
\[\big(\sum_{i=1}^n|C_{ij}|^2\big)^{1/2}=\Big(\sum_{i=1}^n\big|\sum_{m<k}A_{im}B_{mj}\big|^2\Big)^{1/2} \leq \| A\| \| B\|,\quad 1 \leq j\leq n.\]
\end{lem}
\begin{proof}
Remark that $C=APB$, where $P$ is the orthogonal projection onto the subspace generated by the first $k-1$ vectors of the canonical basis. As $\|P\| \leq 1$, $\|C\|=\|APB\|\leq \|A\|\|P\|\|B\|\leq \|A\|\|B\|$.
\end{proof}

\begin{lem}\label{lipschitz}
If $\varphi:\R \to \C$ is Lipschitz continuous, then, 
for $n\times n$ Hermitian matrices $M_1, M_2$ and $p\geq 1$,
\[|\Tr(\varphi(M_1))-\Tr(\varphi(M_2))|^p\leq \| \varphi \|_{\text{Lip}}^pn^{p-1}\| M_1-M_2\|_{S^p}^p,\]
where $\| A\|_{S^p}=\big(\sum_{\lambda \in \text{sp}(A)}|\lambda|^p\big)^{1/p}$ is the Schatten $p$-norm of the normal matrix $A$.
\end{lem}
\begin{proof}
Denote by $\lambda_1\geq \cdots \geq \lambda_n$ the eigenvalues of $M_1$ 
and $\mu_1\geq \cdots \geq\mu_n$ the eigenvalues of $M_2$. 
Then, \[|\Tr(\varphi(M_1))-\Tr(\varphi(M_2))|\leq \sum_{i=1}^n|\varphi(\lambda_i)-\varphi(\mu_i)|\leq \| \varphi \|_{\text{Lip}}\sum_{i=1}^n|\lambda_i-\mu_i|.\]
Using Hölder and Hoffman-Wielandt inequalities, 
\[|\Tr(\varphi(M_1))-\Tr(\varphi(M_2))|^p\leq \| \varphi \|_{\text{Lip}}^pn^{p-1}\sum_{i=1}^n|\lambda_i-\mu_i|^p\leq \| \varphi \|_{\text{Lip}}^pn^{p-1}\| M_1-M_2\|_{S^p}^p.\]
\end{proof}

\subsection{Martingales}
The proofs of our variance bounds and of our CLT rely on martingale theory.
\begin{lem}\label{martingalevariance}
Let $(M_k)_{k \in \N}$ be a square integrable complex martingale. 
Then
\[\mathbb{E}\Big[\big|\sum_{k=1}^N (M_k-M_{k-1})\big|^2\Big]=\sum_{k=1}^N \mathbb{E}\Big[\big|(M_k-M_{k-1})\big|^2\Big].\]
\end{lem}
Note that, if $M_0$ is deterministic (this will be the case of our martingales in this paper), the left-hand side of the equality stated in the preceding Lemma is $\Var M_N$. The following result may be deduced from its version for real-valued martingales (Theorem 35.12 in \cite{Billingsleybook}).

\begin{thm}\label{thm_CLT_martingale}
Suppose that, for all $N\geq 1$, $(M_k^{(N)})_{k \in \N}$ is a square integrable complex martingale and define, for $k \geq 1$, $\Delta_k^{(N)}:=M_k^{(N)}-M_{k-1}^{(N)}$. If
\begin{equation}\label{L}
\forall \varepsilon > 0,\, L(\varepsilon,N):=\sum_{k\geq 1}\mathbb{E}[|\Delta_k^{(N)}|^2\mathbf{1}_{|\Delta_k^{(N)}|\geq \varepsilon}]\underset{N\to +\infty}{\longrightarrow} 0,
\end{equation}
\begin{equation}\label{v}
V_N:=\sum_{k\geq 1}\mathbb{E}_{\leq k-1}[|\Delta_k^{(N)}|^2]\underset{N\to +\infty}{\longrightarrow} v\geq 0,
\end{equation}
and
\begin{equation}\label{w}
W_N:=\sum_{k\geq 1}\mathbb{E}_{\leq k-1}[(\Delta_k^{(N)})^2]\underset{N\to +\infty}{\longrightarrow} w \in \mathbb{C}
\end{equation}
(convergences in \eqref{v} and \eqref{w} have to be understood in probability), then 
\[\sum_{k\geq 1}\Delta_k^{(N)}\Rightarrow_{N\to +\infty}\mathcal{N}_{\mathbb{C}}(0,v,w).\]
\end{thm}

\subsection{Analytic subordination for free additive convolution}
Voiculescu noticed in \cite{Voiculescu93} that the Stieltjes transform of the free additive convolution $\mu \boxplus \nu$ of two Borel  probability measures $\mu, \nu$ on $\mathbb{R}$ is generically subordinated (in the sense of Littlewood) to the Stieltjes transform of $\nu$: there exists an analytic self-map of the upper half-plane $\omega : \mathbb{C}^+ \to \mathbb{C}^+$ satisfying $(iy)^{-1}\omega(iy) \longrightarrow 1$ when $y\longrightarrow +\infty$ such that $G_{\mu \boxplus \nu}(z)=G_{\nu}(\omega (z)),\, z\in \mathbb{C}^+$. In fact, the subordination relation holds at the level of operators: if $x,y$ are free selfadjoint noncommutative random variables in a tracial $W^*$-probability space, then $E((z1-x-y)^{-1})=(\omega(z)1-y)^{-1},\, z\in \mathbb{C}^+$, where $E$ is the conditional expectation with respect to the von Neumann subalgebra generated by $y$ \cite{Biane98}. Note that  these subordination relations hold for $z\in \mathbb{C}\setminus \mathbb{R}$ if $\omega$ is analytically continued by Schwarz reflection principle. 

When $\mu=\mu_v$ is a semicircular distribution, then $\omega(z)=z-vG_{\mu\boxplus \nu}(z)$. The subordination map $\omega $ is then a conformal map from $\mathbb{C}^+$ onto a simply connected domain $\Omega \subset \mathbb{C}^+$, its inverse being the restriction to $\Omega$ of $H : \mathbb{C}^+ \to \mathbb{C}$ defined by $H(z)=z+vG_{\nu}(z)$; moreover $\Omega=H^{-1}(\mathbb{C}^+)$ (see \cite{Biane97b} for these results). We will, by a slight abuse, use the notation $\omega$, $\Omega$ and $H$ when $v=\sigma^2$ and $\nu=\nu_{\infty}$ and $\omega_N$, $\Omega_N$ and $H_N$ when $v=N\sigma_N^2$ and $\nu=\nu_N$. It follows from the continuity of free additive convolution that the sequence of analytic maps $(\omega_N)_{N\geq 1}$ converges uniformly on compact sets of $\mathbb{C}^+$ to $\omega$.

\subsection{Definition and properties of $\mathcal{H}_s$}\label{sec:Hs}
For $s>0$, define the normed linear space $\mathcal{H}_s$ of all (real or complex-valued) $f \in L^2(\R)$ such that 
\[\|f\|_{\mathcal{H}_s}:=\Big(\int_{\R}(1+2|t|)^{2s}|\hat{f}(t)|^2dt\Big)^{1/2}<+\infty. \]

It is proved in \cite{BGEnrMic} (Lemma 13) that $\text{span}\{\varphi_z:x\mapsto (z-x)^{-1}; z\in \C\setminus \R\}$ 
is dense in $(\mathcal{H}_s,\| \|_{\mathcal{H}_s})$, for any $s>0$.

\begin{prop}[Sobolev embedding]
For $s>\frac{1}{2}$, there is a continuous injection $j$ of $(\mathcal{H}_s,\|\cdot\|_{\mathcal{H}_s})$ into $(\mathcal{C}_b,\|\cdot\|_{\infty})$, i.e. there exists a constant $C_s>0$ such that, for all $\varphi \in \mathcal{H}_s$,
\[ \|j(\varphi)\|_{\infty} \leq C_s \|\varphi\|_{\mathcal{H}_s}. \]
\end{prop}
\begin{proof}
This is a particular case of Sobolev embedding. We give a simple argument to prove this fact in our situation. 
As $s>\frac{1}{2}$, the Fourier transform induces a bounded linear map $(\mathcal{H}_s,\|\cdot \|_{\mathcal{H}_s}) \to (L^1(\mathbb{R})\cap L^2(\mathbb{R}),\|\cdot \|_{1})$. Indeed, for $\varphi \in \mathcal{H}_s$, by Cauchy-Schwarz inequality,
\begin{align*}
 \|\hat{\varphi}\|_1 & = \int_{\R}(1+2|t|)^s|\hat{\varphi}(t)|(1+2|t|)^{-s}dt \leq \|\varphi\|_{\mathcal{H}_s}\Big(\int_{\R}(1+2|t|)^{-2s}dt\Big)^{1/2}.
\end{align*}
Therefore, by Fourier inversion formula, one may define a bounded continuous version of $\varphi$ by $j(\varphi) : t\mapsto (2\pi)^{-1}\int_{\R}\hat{\varphi}(-\xi)e^{it\xi}d\xi$. Moreover, this defines a bounded linear map $j : (\mathcal{H}_s,\|\cdot \|_{\mathcal{H}_s}) \to (\mathcal{C}_b(\mathbb{R}),\|\cdot \|_{\infty})$:
\[\|j(\varphi)\|_{\infty} \leq (2\pi)^{-1}\|\hat{\varphi}\|_1\leq C_s\|\varphi\|_{\mathcal{H}_s}.\]
\end{proof}

Therefore, given a continuous linear functional $T$ on $(\mathcal{C}_b(\mathbb{R}),\|\cdot \|_{\infty})$, one may properly define the continuous linear functional (still denoted by) $T$ on $(\mathcal{H}_s,\|\cdot \|_{\mathcal{H}_s})$ by $T(\varphi)=T(j(\varphi))$. This is how the linear spectral statistic $\int_{\mathbb{R}}\varphi(x)\mu_N(dx)$ is defined for $\varphi \in \mathcal{H}_s$.

Another remarkable and well-known property of $\mathcal{H}_s$ is that, for $s \leq k$, $\mathcal{C}^k_c(\R) \subset \mathcal{H}_s$, due to the link between regularity of a function and decreasing at infinity of its Fourier transform. Therefore, item \ref{main1} of Theorem \ref{main} is true in particular for functions in $\mathcal{C}^2_c(\R)$ and item \ref{main2} for functions in $\mathcal{C}^7_c(\R)$.

\section{Concentration bounds}\label{concentration}

In this Section, one derives concentration bounds that hold, unless explicitly stated, 
without assuming the entries of $W_N$ to be bounded. 

\subsection{Quadratic forms}

Applying Proposition \ref{Schur} to $zI_N-X_N$ leads to expressions involving 
random quadratic forms $C_k^{(k)*}R^{(k)}(z)C_k^{(k)}$. 
It is easy to compute the expectation of such quadratic forms. 
(Recall that $\esp_k$ denotes the expectation with respect to $\{W_{ik},1 \leq i \leq N\}$.)
\[\mathbb{E}\Big[C_k^{(k)*}R^{(k)}(z)C_k^{(k)}\Big]=\mathbb{E}\Big[\mathbb{E}_k[C_k^{(k)*}R^{(k)}(z)C_k^{(k)}]\Big]=\sigma_N^2\mathbb{E}[\Tr(R^{(k)}(z))].\]
Their variance may be deduced from the following Lemma:

\begin{lem}\label{quadratic forms}
For $(N-1)\times (N-1)$ random matrices $A,B$ independent of $\{W_{ik},1\leq i\leq N\}$ (for convenience, rows and columns of $A, B$ are indexed by $\{1,\ldots ,N\}\setminus\{k\}$), and $h\in \{1,\ldots ,N\}$,
$$\mathbb{E}_k\Big[\mathbb{E}_{\leq h}\big[C_k^{(k)*}AC_k^{(k)}-\sigma_N^2\Tr(A)\big]\mathbb{E}_{\leq h}\big[C_k^{(k)*}BC_k^{(k)}-\sigma_N^2\Tr(B)\big]\Big]$$
$$=\sigma_N^4\sum_{i,j\neq k\leq h}\mathbb{E}_{\leq h}[A_{ij}]\mathbb{E}_{\leq h}[B_{ji}]+\tau_N^2\sum_{i,j\neq k\leq h}\mathbb{E}_{\leq h}[A_{ij}]\mathbb{E}_{\leq h}[B_{ij}]+\kappa_N\sum_{i\neq k\leq h}\mathbb{E}_{\leq h}[A_{ii}]\mathbb{E}_{\leq h}[B_{ii}].$$
\end{lem}
In particular, with $A=R^{(k)}(z)^{q*}, B=R^{(k)}(z)^q, q\in \{1,2\}$, $h=N$, 
$$\mathbb{E}_k\Big[\big|C_k^{(k)*}R^{(k)}(z)^qC_k^{(k)}-\sigma_N^2\Tr(R^{(k)}(z)^q)\big|^2\Big]$$
$$=\sigma_N^4\Tr(R^{(k)}(z)^{q*}R^{(k)}(z)^q)+\tau_N^2\Tr(\overline{R^{(k)}(z)^{q}}R^{(k)}(z)^q)+\kappa_N\sum_{i\neq k}|R^{(k)}(z)^q_{ii}|^2$$
$$\leq m_N\Tr(R^{(k)}(z)^{q*}R^{(k)}(z)^q).$$

\begin{proof}[Proof of Lemma \ref{quadratic forms}]
Since
\begin{align*}
\mathbb{E}_{\leq h}[C_k^{(k)*}AC_k^{(k)}-\sigma_N^2\Tr(A)] & = \mathbb{E}_{\leq h}\Big[\sum_{i\neq j\neq k}\overline{W_{ki}}W_{kj}A_{ij}+\sum_{i\neq k}(|W_{ki}|^2-\sigma_N^2)A_{ii}\Big]\\
										      & = \sum_{i\neq j\neq k\leq h}\overline{W_{ki}}W_{kj}\mathbb{E}_{\leq h}[A_{ij}]+\sum_{i\neq k\leq h}(|W_{ki}|^2-\sigma_N^2)\mathbb{E}_{\leq h}[A_{ii}]
\end{align*}
and
\begin{align*}
\mathbb{E}_{\leq h}[C_k^{(k)*}BC_k^{(k)}-\sigma_N^2\Tr(B)] & = \mathbb{E}_{\leq h}\Big[\sum_{i'\neq j'\neq k}\overline{W_{ki'}}W_{kj'}B_{i'j'}+\sum_{i'\neq k}(|W_{ki'}|^2-\sigma_N^2)B_{i'i'}\Big]\\
										     & = \sum_{i'\neq j'\neq k\leq h}\overline{W_{ki'}}W_{kj'}\mathbb{E}_{\leq h}[B_{i'j'}]+\sum_{i'\neq k\leq h}(|W_{ki'}|^2-\sigma_N^2)\mathbb{E}_{\leq h}[B_{i'i'}],
\end{align*}
it follows that 
$$\mathbb{E}_k\Big[\mathbb{E}_{\leq h}\big[C_k^{(k)*}AC_k^{(k)}-\sigma_N^2\Tr(A)\big]\mathbb{E}_{\leq h}\big[C_k^{(k)*}BC_k^{(k)}-\sigma_N^2\Tr(B)\big]\Big]$$
$$=\sum_{i\neq j\neq k\leq h}\mathbb{E}[|W_{ki}|^2]\mathbb{E}[|W_{kj}|^2]\mathbb{E}_{\leq h}[A_{ij}]\mathbb{E}_{\leq h}[B_{ji}]$$
$$+\sum_{i\neq j\neq k\leq h}\mathbb{E}[\overline{W_{ki}}^2]\mathbb{E}[W_{kj}^2]\mathbb{E}_{\leq h}[A_{ij}]\mathbb{E}_{\leq h}[B_{ij}]$$
$$+\sum_{i\neq k\leq h}\mathbb{E}\big[(|W_{ki}|^2-\sigma_N^2)^2\big]\mathbb{E}_{\leq h}[A_{ii}]\mathbb{E}_{\leq h}[B_{ii}]$$
$$=\sigma_N^4\sum_{i,j\neq k\leq h}\mathbb{E}_{\leq h}[A_{ij}]\mathbb{E}_{\leq h}[B_{ji}]+\tau_N^2\sum_{i,j\neq k\leq h}\mathbb{E}_{\leq h}[A_{ij}]\mathbb{E}_{\leq h}[B_{ij}]+\kappa_N\sum_{i\neq k\leq h}\mathbb{E}_{\leq h}[A_{ii}]\mathbb{E}_{\leq h}[B_{ii}].$$
\end{proof}

To obtain more general $L^p$ concentration bounds for quadratic forms, we will use Lemma 2.7 in \cite{BaiSil98}:

\begin{lem}\label{lem_moment_quadratic_forms}
	Let $p \geqslant 2$ and $Y=(Y_1, \ldots , Y_n)$ be a $n$-tuple of independent identically distributed standard complex entries 
			with finite $2p$-th moment. Then, for all $A\in \mathcal{M}_n(\mathbb{C})$,
	\[
	\esp\big[|Y^*AY-\Tr (A)|^p\big] \leqslant K_p \Big(\big(\esp\big[|Y_1|^4\big]\Tr (AA^*)\big)^{p/2}+\esp\big[|Y_1|^{2p}\big]\Tr((AA^*)^{p/2})\Big).
	\]
\end{lem}
By independence of the entries of $X_N$, 
Lemma \ref{lem_moment_quadratic_forms} may be applied to $Y=\sigma_N^{-1}C_k^{(k)}$, $A=\sigma_N^2R^{(k)}(z)^q$, $q\in\{1,2\}$ and $p\geq 2$. Under the additional assumption that the entries of $W_N$ are bounded by $\delta_N$, 
one gets, using \eqref{resolventbound}:
$$\esp\big[|C_k^{(k)*}R^{(k)}(z)^qC_k^{(k)}-\sigma_N^2\Tr(R^{(k)}(z)^q)|^{p}\big] \leq K_{p}|\Im z|^{-qp}\big((Nm_N)^{p/2}+\delta_N^{2p-4}Nm_N\big).$$
If $p=2(1+\varepsilon)$, one has the following improved bound:
\begin{equation}\label{eq_improved_bound}
\esp\big[|C_k^{(k)*}R^{(k)}(z)^qC_k^{(k)}-\sigma_N^2\Tr(R^{(k)}(z)^q)|^{2(1+\varepsilon)}\big] \leq K_{2(1+\varepsilon)}|\Im z|^{-2q(1+\varepsilon)} \big((Nm_N)^{1+\varepsilon}+C_4N^{-1-2\varepsilon}\big).
\end{equation}

\subsection{Variance bounds}\label{sec_variance_bound}

A key point is to obtain bounds on the variance of the trace $\Tr(R_N(z))$ of the resolvent $R_N(z)=(zI_N-X_N)^{-1}$ of $X_N$.
Recall that $\mathbb{E}_{\leq k}$ denotes the conditional expectation on the sigma-field $\sigma(W_{ij},1\leq i\leq j\leq k)$. 
The random variable $\Tr(R_N(z))$ being bounded, $(\mathbb{E}_{\leq k}[\Tr(R_N(z))])_{k\in \mathbb{N}}$ is a square integrable complex martingale, hence, by Lemma \ref{martingalevariance},
$$\Var[\Tr(R_N(z))]=\mathbb{E}\Big[\Big|\sum_{k=1}^N (\mathbb{E}_{\leq k}-\mathbb{E}_{\leq k-1})[\Tr(R_N(z))]\Big|^2\Big]=\sum_{k=1}^N \mathbb{E}\Big[\Big|(\mathbb{E}_{\leq k}-\mathbb{E}_{\leq k-1})[\Tr(R_N(z))]\Big|^2\Big].$$
By Proposition \ref{Schur},
\begin{equation}\label{eq_trR_trRk}
\Tr(R_N(z))-\Tr(R^{(k)}(z))=\frac{1+C_k^{(k)*}R^{(k)}(z)^2C_k^{(k)}}{z-W_{kk}-D_{kk}-C_k^{(k)*}R^{(k)}(z)C_k^{(k)}}.
\end{equation}
Hence, applying Lemma \ref{usefulbound} below, 
\begin{equation}\label{bounddifftrace}
|\Tr(R_N(z))-\Tr(R^{(k)}(z))|\leq \frac{1}{|\mathcal{I}z|}.
\end{equation}
\begin{lem}\label{usefulbound}
$$\bigg|\frac{1+C_k^{(k)*}R^{(k)}(z)^2C_k^{(k)}}{z-W_{kk}-D_{kk}-C_k^{(k)*}R^{(k)}(z)C_k^{(k)}}\bigg|\leq \frac{1}{|\mathcal{I}z|}.$$
\end{lem}
\begin{proof}
Using \eqref{im} with $\psi(A)=C_k^{(k)*}AC_k^{(k)}$, the proof goes as follows:
\begin{align*}
\bigg|\frac{1+C_k^{(k)*}R^{(k)}(z)^2C_k^{(k)}}{z-W_{kk}-D_{kk}-C_k^{(k)*}R^{(k)}(z)C_k^{(k)}}\bigg|&\leq \frac{|1+C_k^{(k)*}R^{(k)}(z)^2C_k^{(k)}|}{\big|\mathcal{I}(z-W_{kk}-D_{kk}-C_k^{(k)*}R^{(k)}(z)C_k^{(k)})\big|}\\
														         &\leq \frac{1+|C_k^{(k)*}R^{(k)}(z)^2C_k^{(k)}|}{\big|\mathcal{I}z-\mathcal{I}C_k^{(k)*}R^{(k)}(z)C_k^{(k)}\big|}\\
														         &\leq \frac{1+C_k^{(k)*}R^{(k)}(z)^{*}R^{(k)}(z)C_k^{(k)}}{\big|\mathcal{I}z+\mathcal{I}zC_k^{(k)*}R^{(k)}(z)^{*}R^{(k)}(z)C_k^{(k)}\big|}\\
														         &\leq \frac{1}{|\mathcal{I}z|}.
\end{align*}
\end{proof}
Since $R^{(k)}(z)$ does not involve the $k$-th row/column of $W_N$, $(\mathbb{E}_{\leq k}-\mathbb{E}_{\leq k-1})[\Tr(R^{(k)}(z))]=0$ and then,
$$\Var[\Tr(R_N(z))]=\sum_{k=1}^N \mathbb{E}\Big[\big|(\mathbb{E}_{\leq k}-\mathbb{E}_{\leq k-1})[\Tr(R_N(z))-\Tr(R^{(k)}(z))]\big|^2\Big].$$
Plugging then \eqref{bounddifftrace} in the right-hand side gives a first bound on the variance:
\begin{equation}\label{variancebound1}
\Var[\Tr(R_N(z))]\leq \frac{4N}{|\mathcal{I}z|^2}.
\end{equation}
The same bound holds for $\Var[\Tr(R^{(k)}(z))]$.
The rest of this subsection is devoted to the following more elaborate bound:

\begin{prop}\label{variancebound2}
For $z\in \mathbb{C}\setminus \mathbb{R}$ and $0\leq \delta \leqslant1$, 
$$\Var[\Tr(R_N(z))]\leq 2\bigg(\frac{s_N^2}{|\mathcal{I}z|^{3-\delta}}+2\frac{m_N(N\sigma_N^2)^{\delta}}{\sigma_N^2|\mathcal{I}z|^{3+\delta}}\bigg)\sum_{k=1}^N\mathbb{E}[|R_N(z)_{kk}|^{1+\delta}].$$
\end{prop}

In particular, for $\delta=0$, the bound (also valid for $\Var[\Tr(R^{(k)}(z))]$) becomes
$$\Var[\Tr(R_N(z))]\leq 2|\mathcal{I}z|^{-4}N(s_N^2+2\sigma_N^{-2}m_N)=O(1).$$

\begin{proof}[Proof of Proposition \ref{variancebound2}]
In \eqref{eq_trR_trRk}, decompose 
\begin{align*}
\frac{1}{z-W_{kk}-D_{kk}-C_k^{(k)*}R^{(k)}(z)C_k^{(k)}} & = \frac{1}{z-D_{kk}-\sigma_N^2\Tr(R^{(k)}(z))}\\
									& + \frac{W_{kk}+C_k^{(k)*}R^{(k)}(z)C_k^{(k)}-\sigma_N^2\Tr(R^{(k)}(z))}{(z-W_{kk}-D_{kk}-C_k^{(k)*}R^{(k)}(z)C_k^{(k)})(z-D_{kk}-\sigma_N^2\Tr(R^{(k)}(z)))}
\end{align*}
and 
\[1+C_k^{(k)*}R^{(k)}(z)^2C_k^{(k)}=1+\sigma_N^2\Tr(R^{(k)}(z)^2)+\Big(C_k^{(k)*}R^{(k)}(z)^2C_k^{(k)}-\sigma_N^2\Tr(R^{(k)}(z)^2)\Big)\]
to get :
\begin{equation}\label{psi}
\Tr(R_N(z))=\Tr(R^{(k)}(z))+\frac{1+\sigma_N^2\Tr(R^{(k)}(z)^2)}{z-D_{kk}-\sigma_N^2\Tr(R^{(k)}(z))}+\Psi_k,
\end{equation}
where $$\Psi_k=\frac{\big(1+C_k^{(k)*}R^{(k)}(z)^2C_k^{(k)}\big)\big(W_{kk}+C_k^{(k)*}R^{(k)}(z)C_k^{(k)}-\sigma_N^2\Tr(R^{(k)}(z))\big)}{\big(z-W_{kk}-D_{kk}-C_k^{(k)*}R^{(k)}(z)C_k^{(k)}\big)\big(z-D_{kk}-\sigma_N^2\Tr(R^{(k)}(z))\big)}$$
$$+\frac{C_k^{(k)*}R^{(k)}(z)^2C_k^{(k)}-\sigma_N^2\Tr(R^{(k)}(z)^2)}{z-D_{kk}-\sigma_N^2\Tr(R^{(k)}(z))}.$$
Note that now the first two terms of the right-hand side of \eqref{psi} are not involving the $k$-th row/column of $W_N$, hence
\begin{equation}\label{vanish}
(\mathbb{E}_{\leq k}-\mathbb{E}_{\leq k-1})\Big[\Tr(R^{(k)}(z))+\frac{1+\sigma_N^2\Tr(R^{(k)}(z)^2)}{z-D_{kk}-\sigma_N^2\Tr(R^{(k)}(z))}\Big]=0.
\end{equation}
Using Jensen inequality (with respect to $\mathbb{E}_{\leq k}$) after writing $\mathbb{E}_{\leq k-1}=\mathbb{E}_{\leq k}\mathbb{E}_{k}$, 
$$\mathbb{E}[|(\mathbb{E}_{\leq k}-\mathbb{E}_{\leq k-1})[\Tr(R_N(z))]|^2]\leq \mathbb{E}[|\Psi_k-\mathbb{E}_k[\Psi_k]|^2]=\mathbb{E}[\mathbb{E}_k[|\Psi_k-\mathbb{E}_k[\Psi_k]|^2]]\leq \mathbb{E}[\mathbb{E}_k[|\Psi_k|^2]].$$
We will need the following bound:
\begin{align*}
\Big|\frac{1}{z-D_{kk}-\sigma_N^2\Tr(R^{(k)}(z))}\Big|&\leq \frac{1}{\big|\mathcal{I}(z-D_{kk}-\sigma_N^2\Tr(R^{(k)}(z)))\big|}\\
								  &\leq \frac{1}{\big|\mathcal{I}z-\sigma_N^2\mathcal{I}\Tr(R^{(k)}(z)))\big|}\\
								 &\leq \frac{1}{|\mathcal{I}z|\big(1+\sigma_N^2\Tr(R^{(k)}(z)^{*}R^{(k)}(z))\big)}\\
								  &\leq \min\Big(\frac{1}{|\mathcal{I}z|},\frac{1}{\sigma_N^2|\mathcal{I}z|\Tr(R^{(k)}(z)^{*}R^{(k)}(z))}\Big),
\end{align*}
where we have used \eqref{im} with $n=N-1$ and $\psi=\Tr$.
Deduce from Lemma \ref{quadratic forms} that
\begin{align*}
\mathbb{E}_k\big[\big|W_{kk}+C_k^{(k)*}R^{(k)}(z)C_k^{(k)}-\sigma_N^2\Tr(R^{(k)}(z))\big|^2\big] & = \mathbb{E}_k[W_{kk}^2]\\
			& \quad + \mathbb{E}_k\big[|C_k^{(k)*}R^{(k)}(z)C_k^{(k)}-\sigma_N^2\Tr(R^{(k)}(z))|^2\big]\\
			& \leq s_N^2+m_N\Tr(R^{(k)}(z)^*R^{(k)}(z));
\end{align*}
\begin{align*}
\mathbb{E}_k\Big[\big|C_k^{(k)*}R^{(k)}(z)^2C_k^{(k)}-\sigma_N^2\Tr(R^{(k)}(z)^2)\big|^2\Big]&\leq m_N\Tr(R^{(k)}(z)^{2*}R^{(k)}(z)^{2})\\
													          &\leq m_N|\mathcal{I}z|^{-2}\Tr(R^{(k)}(z)^*R^{(k)}(z)).
\end{align*}
Hence, for arbitrary $0\leq \delta \leqslant1$,
\begin{align*}
 \mathbb{E}_k\Big[\Big|&\frac{W_{kk}+C_k^{(k)*}R^{(k)}(z)C_k^{(k)}-\sigma_N^2\Tr(R^{(k)}(z))}{z-D_{kk}-\sigma_N^2\Tr(R^{(k)}(z))}\Big|^2\Big]\\
&\leq \frac{s_N^2+m_N\Tr(R^{(k)}(z)^*R^{(k)}(z))}{|z-D_{kk}-\sigma_N^2\Tr(R^{(k)}(z))|^{1-\delta}\big|\mathbb{E}_k[z-W_{kk}-D_{kk}-C_k^{(k)*}R^{(k)}(z)C_k^{(k)}]\big|^{1+\delta}}\\
		       &\leq \big(\frac{s_N^2}{|\mathcal{I}z|^{1-\delta}}+\frac{m_N(N|\mathcal{I}z|^{-2})^{\delta}}{(\sigma_N^2|\mathcal{I}z|)^{1-\delta}}\big)\mathbb{E}_k[|R_N(z)_{kk}|^{1+\delta}],
\end{align*}
by Jensen inequality applied to function $x \mapsto |x|^{-1-\delta}$. Similarly, 
$$\mathbb{E}_k\Big[\Big|\frac{C_k^{(k)*}R^{(k)}(z)^2C_k^{(k)}-\sigma_N^2\Tr(R^{(k)}(z)^2)}{z-D_{kk}-\sigma_N^2\Tr(R^{(k)}(z))}\Big|^2\Big]\leq \frac{m_N|\mathcal{I}z|^{-2}(N|\mathcal{I}z|^{-2})^{\delta}}{(\sigma_N^2|\mathcal{I}z|)^{1-\delta}}\mathbb{E}_k[|R_N(z)_{kk}|^{1+\delta}].$$
It follows that 
$$\mathbb{E}_k[|\Psi_k|^2]\leq 2\bigg(\frac{s_N^2}{|\mathcal{I}z|^{3-\delta}}+2\frac{m_N(N\sigma_N^2)^{\delta}}{\sigma_N^2|\mathcal{I}z|^{3+\delta}}\bigg)\mathbb{E}_k[|R_N(z)_{kk}|^{1+\delta}],$$
and we are done.
\end{proof}

\subsection{Concentration of diagonal entries of the resolvent}

For $z\in \mathbb{C}\setminus \mathbb{R}$, 
\begin{align*}
R_N(z)_{kk}& = \tilde{R}(z)_{kk}+\tilde{R}(z)_{kk}R_N(z)_{kk}(W_{kk}+C_k^{(k)*}R^{(k)}(z)C_k^{(k)}-\sigma_N^2\mathbb{E}[\Tr(R_N(z))]),
\end{align*}
where \[\tilde{R}(z):=\big((z-\sigma_N^2\mathbb{E}[\Tr(R_N(z))])I_N-D_N\big)^{-1}.\]
This equality is obtained by applying, with $n=1$, the following Lemma, which is easily proved by induction from Proposition \ref{Schur}.

\begin{lem}\label{taylor}
\begin{eqnarray*}
R_N(z)_{kk}=&\sum_{i=1}^n\tilde{R}(z)_{kk}^i\Big(W_{kk}+C_k^{(k)*}R^{(k)}(z)C_k^{(k)}-\sigma_N^2\mathbb{E}[\Tr(R_N(z))]\Big)^{i-1}\\
		 &+\tilde{R}(z)_{kk}^nR(z)_{kk}\Big(W_{kk}+C_k^{(k)*}R^{(k)}(z)C_k^{(k)}-\sigma_N^2\mathbb{E}[\Tr(R_N(z))]\Big)^{n}.
\end{eqnarray*}
\end{lem}

Observe that
$$\mathbb{E}\big[\big|W_{kk}+C_k^{(k)*}R^{(k)}(z)C_k^{(k)}-\sigma_N^2\mathbb{E}[\Tr(R_N(z))]\big|^2\big]=$$
$$\mathbb{E}[W_{kk}^2]+\mathbb{E}\big[\big|C_k^{(k)*}R^{(k)}(z)C_k^{(k)}-\sigma_N^2\Tr(R^{(k)}(z))\big|^2\big]+\sigma_N^4\big(\Var[\Tr(R^{(k)}(z))]+\big|\mathbb{E}[\Tr(R^{(k)}(z))-\Tr(R_N(z))]\big|^2\big).$$
By Lemma \ref{quadratic forms}, \eqref{variancebound1} and \eqref{bounddifftrace},
\begin{equation}\label{boundnum}
\mathbb{E}[|W_{kk}+C_k^{(k)*}R^{(k)}(z)C_k^{(k)}-\sigma_N^2\mathbb{E}[\Tr(R_N(z))]|^2]\leq s_N^2+Nm_N|\mathcal{I}z|^{-2}+ (3N+1)\sigma_N^4|\mathcal{I}z|^{-2}=O(N^{-1}),
\end{equation}
and then,
\[\mathbb{E}[|R_N(z)_{kk}-\tilde{R}(z)_{kk}|^2]\leq |\mathcal{I}z|^{-4}(s_N^2+Nm_N|\mathcal{I}z|^{-2}+ (3N+1)\sigma_N^4|\mathcal{I}z|^{-2})=O(N^{-1}),\]
uniformly in $k$. Similarly, uniformly in $1\leq i\neq k\leq N$,
$$\mathbb{E}[|R^{(k)}(z)_{ii}-\widetilde{R^{(k)}}(z)_{ii}|^2]=O(N^{-1}),$$
where \[\widetilde{R^{(k)}}(z):=\Big(\big(z-\sigma_N^2\esp[\Tr R^{(k)}(z)]\big)I_{N-1}-D_N^{(k)}\Big)^{-1},\]
and $D_N^{(k)}$ is the $(N-1)\times(N-1)$ obtained from $D_N$ by deleting the $k$-th row/column (for convenience, rows and columns of $D_N^{(k)}$ are indexed by $\{1,\ldots ,N\}\setminus\{k\}$).
Notice that, by \eqref{bounddifftrace},
$$|\widetilde{R^{(k)}}(z)_{ii}-\tilde{R}(z)_{ii}|\leq \sigma_N^2|\mathcal{I}z|^{-2}\mathbb{E}[|\Tr R^{(k)}(z)-\Tr R_N(z)|]\leq \sigma_N^2|\mathcal{I}z|^{-3}.$$
Finally,
\begin{equation}\label{eq_approximation_Rk_tildeR}
\mathbb{E}[|R^{(k)}(z)_{ii}-\tilde{R}(z)_{ii}|^2]=O(N^{-1}),
\end{equation}
uniformly in $1\leq i\neq k\leq N$.

\subsection{Approximate subordination relations}

From the equation
\begin{eqnarray*}
\mathbb{E}[G_{\mu_N}(z)]&=&\mathbb{E}\Big[\frac{1}{N}\Tr(R_N(z))\Big]=\frac{1}{N}\sum_{k=1}^N\mathbb{E}[R_N(z)_{kk}]\\
				 &=&\frac{1}{N}\sum_{k=1}^N\tilde{R}(z)_{kk}+O(N^{-1/2})=G_{\nu_N}\big(z-N\sigma_N^2\mathbb{E}[G_{\mu_N}(z)]\big)+O(N^{-1/2}),
\end{eqnarray*}
one deduces that the normal sequence $(\mathbb{E}[G_{\mu_N}])_{N\geq 1}$ of analytic maps on $\mathbb{C}\setminus\mathbb{R}$ has for unique accumulation point the Stieltjes transform $G_{\rho}$ of the unique Borel probability measure $\rho $ on $\mathbb{R}$ satisfying Pastur equation \eqref{Pastur}.
Hence $(\mathbb{E}[G_{\mu_N}])_{N\geq 1}$ converges uniformly on compact sets of $\mathbb{C}\setminus\mathbb{R}$ to $G_{\rho}$. Using the variance bound \eqref{variancebound1}, one deduces weak convergence (in probability) of $\mu_N$ towards $\rho$, or equivalently convergence (in probability) of linear spectral statistics $\int_{\mathbb{R}}\varphi(x)\mu_N(dx)$ towards $\int_{\mathbb{R}}\varphi(x)\rho(dx)$ for bounded continuous $\varphi$. This completes a proof of Pastur's Theorem. 

The relations 
$$\mathbb{E}[G_{\mu_N}(z)]=G_{\nu_N}\big(z-N\sigma_N^2\mathbb{E}[G_{\mu_N}(z)]\big)+O(N^{-1/2}),\quad z\in \mathbb{C}\setminus\mathbb{R},$$
$$\mathbb{E}[R_N(z)_{kk}]=((z-N\sigma_N^2\mathbb{E}[G_{\mu_N}(z)])I_N-D_N)^{-1}_{kk}+O(N^{-1/2}),\quad z\in \mathbb{C}\setminus\mathbb{R},$$
may be interpreted as approximate subordination relations for the sum of $W_N$ and $D_N$ at the level of scalars and of operators, respectively. The formula
$$\tilde{\omega}_N(z):=z-N\sigma_N^2\mathbb{E}[G_{\mu_N}(z)]=z-\sigma_N^2\mathbb{E}[\Tr(R_N(z))],\quad z\in \mathbb{C}\setminus\mathbb{R},$$
defines an approximate subordination function $\tilde{\omega}_N : \mathbb{C}\setminus \mathbb{R} \to \mathbb{C}$ 
inducing an analytic self-map of $\mathbb{C}^+$ such that $(iy)^{-1}\tilde{\omega}_N(iy) \longrightarrow 1$ when $y\longrightarrow +\infty$. It follows from Pastur's Theorem that the sequence of analytic maps $(\tilde{\omega}_N)_{N\geq 1}$ converges uniformly on compact sets of $\mathbb{C}\setminus \mathbb{R}$ to $\omega$.

\section{Analysis of bias} \label{bias}

This Section is devoted to the analysis of the bias of the linear spectral statistic $\int_{\mathbb{R}} \varphi(x)\mu_N(dx)$ with respect to its deterministic equivalent $\int_{\mathbb{R}} \varphi(x)\rho_N(dx)$. More precisely, to prove that 
$$b_N(\varphi ):=N\Big(\mathbb{E}[\int_{\mathbb{R}} \varphi(x)\mu_N(dx)]-\int_{\mathbb{R}}\varphi(x)\rho_N(dx)\Big)$$
converges to some limit $b(\varphi)$, we first establish this convergence for 
$$\varphi \in \mathcal{L}_1:=\text{span}\{\varphi_z:x\mapsto (z-x)^{-1}; z\in \C\setminus \R\}$$
and then extend this convergence to more general test functions $\varphi$. To this last purpose, one will need a bound on $b_N(\varphi )$, that we deduce from a bound on $\beta_N(z):=b_N(\varphi _z)$.

\subsection{Bound on the bias of Stieltjes transform}

In this subsection, we derive the following bound on $\beta_N(z)$, which will be true without any boundedness assumption on the entries of $W_N$:

\begin{prop}\label{boundb}
There is a polynomial $P_N$ of degree $5$ with bounded coefficients such that, for $z \in \C\setminus \R$, 
$$|\beta_N(z)| \leqslant N^{-1}P_N(|\mathcal{I}z|^{-1})\sum_{k=1}^N|\tilde{R}(z)_{kk}|^2.$$
\end{prop}

This will follow from a bound on 
$$\tilde{\beta}_N(z):=\esp[\Tr R_N(z)]-\Tr \tilde{R}(z).$$

\begin{lem}\label{boundbtilde}
There is a polynomial $\tilde{P}_N$ of degree $3$ with bounded coefficients such that, for $z \in \C\setminus \R$, 
$$|\tilde{\beta}_N(z)| \leqslant N^{-1}\tilde{P}_N(|\mathcal{I}z|^{-1})\sum_{k=1}^N|\tilde{R}(z)_{kk}|^2.$$
\end{lem}

\begin{proof}[Proof of Lemma \ref{boundbtilde}]
The strategy is to control for each $k$ the difference between $\esp[R_N(z)_{kk}]$ and $\tilde{R}(z)_{kk}$, uniformly in $k$. First of all, $R_N(z)_{kk}-\tilde{R}(z)_{kk}$ is expanded thanks to Lemma \ref{taylor} up to order $n=2$. For $z\in \mathbb{C}\setminus \mathbb{R}$, 
\begin{align*}
R_N(z)_{kk}-\tilde{R}(z)_{kk} & =   \tilde{R}(z)_{kk}^2\big(W_{kk}+C_k^{(k)*}R^{(k)}(z)C_k^{(k)}-\sigma_N^2\mathbb{E}[\Tr(R_N(z))]\big)\\
	  & \quad + \tilde{R}(z)_{kk}^2R(z)_{kk}\big(W_{kk}+C_k^{(k)*}R^{(k)}(z)C_k^{(k)}-\sigma_N^2\mathbb{E}[\Tr(R_N(z))]\big)^2.
\end{align*}
Then, using \eqref{bounddifftrace} and \eqref{boundnum},
\begin{align*}
|\esp[R_N(z)_{kk}] - \tilde{R}(z)_{kk}| & \leq \sigma_N^2|\tilde{R}(z)_{kk}|^2\esp[|\Tr (R^{(k)}(z))-\Tr(R_N(z))|]\\
	  & \quad  + |\tilde{R}(z)_{kk}|^2\esp\Big[\big|R(z)_{kk}\big|\big|W_{kk}+C_k^{(k)*}R^{(k)}(z)C_k^{(k)}-\sigma_N^2\mathbb{E}[\Tr(R_N(z))]\big|^2\Big]\\
					         & \leq \frac{\sigma_N^2}{|\Im z|}|\tilde{R}(z)_{kk}|^2 + \Big(\frac{s_N^2}{|\Im z|}+\frac{Nm_N+ (3N+1)\sigma_N^4}{|\Im z|^3}\Big)|\tilde{R}(z)_{kk}|^2.
\end{align*}
Summing on $k$ yields
$$|\tilde{\beta}_N(z)| \leqslant N^{-1}\Big(\frac{N\sigma_N^2+Ns_N^2}{|\Im z|}+\frac{N^2m_N+ N(3N+1)\sigma_N^4}{|\Im z|^3}\Big)\sum_{k=1}^N|\tilde{R}(z)_{kk}|^2.$$
The assumptions on $\sigma_N$, $s_N$ and $m_N$ conclude the proof.
\end{proof}

We are ready for the proof of Proposition \ref{boundb}.

\begin{proof}[Proof of Proposition \ref{boundb}]
We borrow a trick from \cite{HaaTho05}. Define $z'=H_N(\tilde{\omega}_N(z))$. Observe then that 
$$z'-z=-\sigma_N^2\tilde{\beta}_N(z).$$
On the one hand, if 
$$\sigma_N^2N^{-1}\tilde{P}_N(|\mathcal{I}z|^{-1})\sum_{k=1}^N|\tilde{R}(z)_{kk}|^2\geq \frac{|\mathcal{I}z|}{2},$$
or equivalently 
$$1\leq 2\sigma_N^2N^{-1}|\Im z|^{-1}\tilde{P}_N(|\mathcal{I}z|^{-1})\sum_{k=1}^N|\tilde{R}(z)_{kk}|^2,$$
then it is straightforward to get
$$|\beta_N(z)| \leqslant 2N|\mathcal{I}z|^{-1} \leqslant N^{-1}4N\sigma_N^2|\mathcal{I}z|^{-2}\tilde{P}_N(|\mathcal{I}z|^{-1})\sum_{k=1}^N|\tilde{R}(z)_{kk}|^2.$$
On the other hand, if 
$$\sigma_N^2N^{-1}\tilde{P}_N(|\mathcal{I}z|^{-1})\sum_{k=1}^N|\tilde{R}(z)_{kk}|^2\leq \frac{|\mathcal{I}z|}{2},$$
one has, using the bound we obtained on $\tilde{\beta}_N(z)$: 
$$|\mathcal{I}z'-\mathcal{I}z|\leq |z'-z|=\sigma_N^2|\tilde{\beta}_N(z)|\leq \frac{|\mathcal{I}z|}{2},$$
therefore $z'\in \mathbb{C}^+$, or equivalently $\tilde{\omega}_N(z)\in \Omega_N$, and consequently $\omega_N(z')=\tilde{\omega}_N(z)$. It follows that 
\begin{align*}
\beta_N(z)-\tilde{\beta}_N(z) & = N\big(G_{\nu_N}(\tilde{\omega}_N(z))-G_{\nu_N}(\omega_N(z))\big)\\
			  & = N\big(G_{\nu_N}(\omega_N(z'))-G_{\nu_N}(\omega_N(z))\big)\\
			  & = N\big(G_{\rho_N}(z')-G_{\rho_N}(z)\big)\\
			  & = N(z-z')\int_{\mathbb{R}}\frac{\rho_N(dx)}{(z'-x)(z-x)}\\
			  & = N\sigma_N^2\tilde{\beta}_N(z)\int_{\mathbb{R}}\frac{\rho_N(dx)}{(z'-x)(z-x)}.
\end{align*}
Hence, 
\[|\beta_N(z)|\leq (1+2N\sigma_N^2|\Im z|^{-2})|\tilde{\beta}_N(z)|\leq N^{-1}(1+2N\sigma_N^2|\Im z|^{-2})\tilde{P}_N(|\mathcal{I}z|^{-1})\sum_{k=1}^N|\tilde{R}(z)_{kk}|^2,\]
which concludes the proof.
\end{proof}

\subsection{Equivalent of the bias of the Stieltjes transform} \label{subsectionconvb}

The aim of this section is to show that $\beta_N(z)$ converges and to compute its limit $\beta(z)$. 
Note that, in this subsection, the entries of $W_N$ are supposed to be bounded by $\delta_N$.

\begin{prop} \label{convb}
For $z \in \C\setminus \R$,
$$\beta_N(z)\underset{N\to +\infty}{\longrightarrow} \beta(z):=-\frac{\omega''(z)}{2\omega'(z)^2}\Big[\frac{s^2-\sigma^2-\tau}{\sigma^2}+\frac{\tau}{\sigma^2}\frac{\omega'(z)}{\frac{\tau}{\sigma^2}+(1-\frac{\tau}{\sigma^2})\omega'(z)}+\frac{\kappa}{\sigma^4}(1-\frac{1}{\omega'(z)}) \Big].$$
\end{prop}

Remark that, in the non deformed case, $\nu_{\infty}=\delta_0$, $\rho=\mu_{\sigma^2}$ and $\omega(z)=z-\sigma^2G_{\mu_{\sigma^2}}(z)=G_{\mu_{\sigma^2}}(z)^{-1}$. In this case, 
$\beta(z)$ coincides with $b_0(\varphi_z)$ in Theorem \ref{BaoXieThm}.

We deduce Proposition \ref{convb} from the following Lemma stating the convergence of $\tilde{\beta}_N(z)$:

\begin{lem}
For $z \in \C\setminus \R$,
$$\tilde{\beta}_N(z)\underset{N\to +\infty}{\longrightarrow} \tilde{\beta}(z):=-\frac{\omega''(z)}{2\omega'(z)^3}\Big[\frac{s^2-\sigma^2-\tau}{\sigma^2}+\frac{\tau}{\sigma^2}\frac{\omega'(z)}{\frac{\tau}{\sigma^2}+(1-\frac{\tau}{\sigma^2})\omega'(z)}+\frac{\kappa}{\sigma^4}(1-\frac{1}{\omega'(z)}) \Big].$$
\end{lem}

\begin{proof} 
The conclusion is obtained by controlling for each $k$ the difference between $\esp[R_N(z)_{kk}]$ and $\tilde{R}(z)_{kk}$, uniformly in $k$. This time, $R_N(z)_{kk}-\tilde{R}(z)_{kk}$ is expanded using Lemma \ref{taylor} up to order $n=3$: for $z\in \mathbb{C}\setminus \mathbb{R}$, 
\begin{align*}
R_N(z)_{kk}-\tilde{R}(z)_{kk} & = \tilde{R}(z)_{kk}^2\big(W_{kk}+C_k^{(k)*}R^{(k)}(z)C_k^{(k)}-\sigma_N^2\mathbb{E}[\Tr(R_N(z))]\big)\\
			      & \quad + \tilde{R}(z)_{kk}^3\big(W_{kk}+C_k^{(k)*}R^{(k)}(z)C_k^{(k)}-\sigma_N^2\mathbb{E}[\Tr(R_N(z))]\big)^2\\
			      & \quad +\tilde{R}(z)_{kk}^3R(z)_{kk}\big(W_{kk}+C_k^{(k)*}R^{(k)}(z)C_k^{(k)}-\sigma_N^2\mathbb{E}[\Tr(R_N(z))]\big)^3.
\end{align*}
By \eqref{psi}, the expectation of the first parenthesis writes:
\begin{align*}
\mathbb{E}\big[W_{kk}+C_k^{(k)*}R^{(k)}(z)C_k^{(k)}-\sigma_N^2\mathbb{E}[\Tr(R_N(z))]\big] & = \sigma_N^2\mathbb{E}[\Tr(R^{(k)}(z))-\Tr(R_N(z))]\\
& = -\sigma_N^2\mathbb{E}\big[\frac{1+\sigma_N^2\Tr(R^{(k)}(z)^2)}{z-D_{kk}-\sigma_N^2\Tr(R^{(k)}(z))}+\Psi_k\big]\\
& = -\sigma_N^2\tilde{R}(z)_{kk}\big(1+\sigma_N^2\mathbb{E}[\Tr(R^{(k)}(z)^2)]\big)+O(N^{-3/2}),
\end{align*}
uniformly in $k$. 

Applying Lemma \ref{quadratic forms}, the expectation of the second parenthesis becomes:
\begin{align*}
  \mathbb{E}\big[\big(W_{kk}+ & C_k^{(k)*}R^{(k)}(z)C_k^{(k)}-  \sigma_N^2\mathbb{E}[\Tr(R_N(z))]\big)^2\big]\\
 & = \mathbb{E}[W_{kk}^2]+\mathbb{E}\big[\big(C_k^{(k)*}R^{(k)}(z)C_k^{(k)}-\sigma_N^2\Tr(R^{(k)}(z))\big)^2\big]\\
 & \quad +\sigma_N^4\big(\mathbb{E}\big[(\Tr(R^{(k)}(z))-\mathbb{E}[\Tr(R^{(k)}(z))])^2\big]+\mathbb{E}\big[\Tr(R^{(k)}(z))-\Tr(R_N(z))\big]^2\big)\\
 & = s_N^2+\sigma_N^4\mathbb{E}[\Tr(R^{(k)}(z)^2)]+\tau_N^2\mathbb{E}[\Tr(^tR^{(k)}(z)R^{(k)}(z))]+\kappa_N\sum_{i\neq k}\mathbb{E}[R^{(k)}(z)_{ii}^2]\\
 & \quad +\sigma_N^4\mathbb{E}\big[(\Tr(R^{(k)}(z))-\mathbb{E}[\Tr(R^{(k)}(z))])^2\big]+\sigma_N^4\mathbb{E}[\Tr(R^{(k)}(z))-\Tr(R_N(z))]^2.
\end{align*}
Note that the last two terms are $O(N^{-2})$, uniformly in $k$.

To bound the third parenthesis, we will make use of Lemma \ref{lem_moment_quadratic_forms} and \eqref{bounddifftrace}: 
\begin{align*}
\mathbb{E}\big[\big|W_{kk}+C_k^{(k)*}R^{(k)}(z) & C_k^{(k)} -\sigma_N^2\mathbb{E}[\Tr(R_N(z))]\big|^3\big]\\
& \leq 16\big(\mathbb{E}[|W_{kk}|^3]+\mathbb{E}\big[\big|C_k^{(k)*}R^{(k)}(z)C_k^{(k)}-\sigma_N^2\Tr(R^{(k)}(z))\big|^3\big]\\
& \quad + \sigma_N^6\big(\mathbb{E}[|\Tr(R^{(k)}(z))-\mathbb{E}[\Tr(R^{(k)}(z))]|^3]+\mathbb{E}[|\Tr(R^{(k)}(z))-\Tr(R_N(z))|]^3\big)\big)\\
& \leq 16\delta_Ns_N^2+16K_{3}|\Im z|^{-3}((Nm_N)^{3/2}+\delta_N^{2}Nm_N)\\
& \quad + 32N\sigma_N^6|\mathcal{I}z|^{-1}\Var[\Tr(R^{(k)}(z))] + 16\sigma_N^6|\mathcal{I}z|^{-3}=o(N^{-1}),
\end{align*}
uniformly in $k$. Hence, using the bound from Proposition \ref{variancebound2},
\begin{equation}\label{eq_bias}
\mathbb{E}[R_N(z)_{kk}]-\tilde{R}(z)_{kk} - \tilde{R}(z)_{kk}^3\Big(s_N^2-\sigma_N^2 + \tau_N^2\mathbb{E}[\Tr(^tR^{(k)}(z)R^{(k)}(z))]+\kappa_N\sum_{i\neq k}\mathbb{E}\big[R^{(k)}(z)_{ii}^2\big]\Big)=o(N^{-1}),
\end{equation}
uniformly in $k$.

We will now deal with the term $\mathbb{E}[\Tr(^tR^{(k)}(z)R^{(k)}(z))]$:

\begin{lem}\label{BaoXie}
For $z\in \C\setminus \R$, for large enough $N$,
\[\mathbb{E}[\Tr(^tR^{(k)}(z)R^{(k)}(z))]=\frac{\sum_{i\neq k}\tilde{R}(z)_{ii}^2}{1-\tau_N\sum_{i\neq k}\tilde{R}(z)_{ii}^2} +O(N\delta_N),\]
uniformly in $k$.
\end{lem}

We first check that the right-hand side is well-defined:

\begin{lem}\label{lem_well_defined}
For any compact subset $K$ of $\C\setminus \R$,
\[\limsup_{N\to +\infty} \sup_{z_1, z_2\in K} \sup_{|w_{i,N}|\leq \sigma_N^2} \sum_{i=1}^N|w_{i,N}||\tilde{R}(z_1)_{ii}\tilde{R}(z_2)_{ii}|<1.\]
\end{lem}

\begin{proof}
Note that (we follow the proof of Corollary 3.35 in \cite{JiLee})
\begin{align*}
\sum_{i=1}^N|w_{i,N}||\tilde{R}(z_1)_{ii}\tilde{R}(z_2)_{ii}|
& \leqslant \frac{N\sigma_N^2}{2N}\sum_{i=1}^N\Big(|\tilde{R}(z_1)_{ii}|^2+|\tilde{R}(z_2)_{ii}|^2\Big)\\
& \leqslant \frac{N\sigma_N^2}{2} \Big(\int_\mathbb{R} \frac{d\nu_N(x)}{|\tilde{\omega}_N(z_1)-x|^2}+ \int_\mathbb{R}\frac{d\nu_N(x)}{|\tilde{\omega}_N(z_2)-x|^2}\Big).
\end{align*}
Since
\begin{align*}
\sup_{z\in K}N\sigma_N^2\int_\mathbb{R} \frac{d\nu_N(x)}{|\tilde{\omega}_N(z)-x|^2}
&=\sup_{z\in K}N\sigma_N^2\frac{-\Im G_{\nu_N}(\tilde{\omega}_N(z))}{\Im\tilde{\omega}_N(z)}\\
&=\sup_{z\in K}\frac{-N\sigma_N^2\Im \mathbb{E}[G_{\mu_N}(z)]}{\Im z-N\sigma_N^2\Im \mathbb{E}[G_{\mu_N}(z)]}+O(N^{-1/2}),
\end{align*}
Thus, by Pastur's Theorem,
\[ \limsup_{N\to +\infty} \sup_{z_1, z_2\in K} \sup_{|w_{i,N}|\leq \sigma_N^2} \sum_{i=1}^N|w_{i,N}||\tilde{R}(z_1)_{ii}\tilde{R}(z_2)_{ii}| \leqslant \sup_{z\in K}\frac{-\sigma^2\Im \mathbb{E}[G_{\rho}(z)]}{\Im z-\sigma^2\Im \mathbb{E}[G_{\rho}(z)]}<1.\]
\end{proof}

\begin{proof}[Proof of Lemma \ref{BaoXie}]
 Using the fact that $R^{(k)}(z)$ is the inverse of $zI_{N-1}-M^{(k)}$ yields the following:
\[(z-D_{ii})\mathbb{E}\big[(R^{(k)}(z)^tR^{(k)}(z))_{ii}\big]=\mathbb{E}\big[R^{(k)}(z)_{ii}\big]+\sum_{j\neq k}\mathbb{E}\big[W_{ij}(R^{(k)}(z)^tR^{(k)}(z))_{ji}\big].\]
Decompose $R^{(k)}(z)$ (apply \eqref{remove} twice) in order to remove the dependence with $W_{ij}$:
\begin{align}\label{eq_decomposition_Rk}
 R^{(k)}(z) & = R^{(kij)}(z)+R^{(kij)}(z)(1-\frac{1}{2}\delta_{ij})(W_{ij}E_{ij}+\overline{W_{ij}}E_{ji})R^{(kij)}(z) \notag \\
 & \quad + \big(R^{(kij)}(z)(1-\frac{1}{2}\delta_{ij})(W_{ij}E_{ij}+\overline{W_{ij}}E_{ji})\big)^2R^{(k)}(z),
\end{align}
so that, for each $j\neq k$, $\mathbb{E}\big[W_{ij}(R^{(k)}(z)^tR^{(k)}(z))_{ji}\big]$ is the sum of nine terms. 
The only contributing terms are 
\[\mathbb{E}\big[W_{ij}\big(R^{(kij)}(z)^t\big(R^{(kij)}(z)(1-\frac{1}{2}\delta_{ij})(W_{ij}E_{ij}+\overline{W_{ij}}E_{ji})R^{(kij)}(z)\big)\big)_{ji}\big]\]
and
\[\mathbb{E}\Big[W_{ij}\Big(R^{(kij)}(z)(1-\frac{1}{2}\delta_{ij})(W_{ij}E_{ij}+\overline{W_{ij}}E_{ji})R^{(kij)}(z)^tR^{(kij)}(z)\Big)_{ji}\Big].\]
Indeed, by independence of $W_{ij}$ and $R^{(kij)}(z)$, the term $\mathbb{E}\big[W_{ij}(R^{(kij)}(z)^tR^{(kij)}(z))_{ji}\big]$ vanishes. The six other terms may be bounded by $C|\Im z|^{-l-1}\esp[|W_{ij}|^l]$, for $l \in \{3,4,5\}$ and $C>0$ a numerical constant which does not depend on $k$, and are therefore $O(\delta_NN^{-1})$, uniformly in $i,j,k$. When $i\neq j$, the first contributing term is equal to
\begin{align}\label{eq_bias_contributing_term1}
\mathbb{E} & \big[W_{ij}\big(R^{(kij)}(z)^tR^{(kij)}(z)(W_{ij}E_{ji}+\overline{W_{ij}}E_{ij})^tR^{(kij)}(z)\big)_{ji}\big] \notag \\
& = \tau_N\mathbb{E}\big[\big(R^{(kij)}(z)^tR^{(kij)}(z)\big)_{jj}R^{(kij)}(z)_{ii}\big]+\sigma_N^2\mathbb{E}\big[\big(R^{(kij)}(z)^tR^{(kij)}(z)\big)_{ji}R^{(kij)}(z)_{ij}\big].
\end{align}
Similarly, the second one is
\begin{equation}\label{eq_bias_contributing_term2}
 \tau_N\mathbb{E}\big[R^{(kij)}(z)_{ji}\big(R^{(kij)}(z)^tR^{(kij)}(z)\big)_{ji}\big]+\sigma_N^2\mathbb{E}\big[R^{(kij)}(z)_{jj}\big(R^{(kij)}(z)^tR^{(kij)}(z)\big)_{ii}\big].
\end{equation}
When $i=j$, the first and second contributing terms become
\begin{equation*}\label{eq_bias_contributing_term3}
 s_N^2\mathbb{E}\big[\big(R^{(kii)}(z)^tR^{(kii)}(z)\big)_{ii}R^{(kii)}(z)_{ii}\big].
\end{equation*}

Note that, by \eqref{remove}, 
\begin{equation}\label{removebound}
\| R^{(kil)}(z)-R^{(k)}(z)\|=\|  R^{(kil)}(z)(1-\frac{1}{2}\delta_{il})(W_{il}E_{il}+\overline{W_{il}}E_{li})R^{(k)}(z)\| \leq 2\delta_N|\mathcal{I}z|^{-2}.
\end{equation}
Therefore, for any $p$ and $q$, 
\[\big|R^{(kij)}(z)_{pq}-R^{(k)}(z)_{pq}\big| \leq 2\delta_N|\Im z|^{-2}, \]
and
\[\big|(R^{(kij)}(z)^tR^{(kij)}(z))_{pq}-(R^{(k)}(z)^tR^{(k)}(z))_{pq}\big| \leq 4\delta_N|\Im z|^{-3}. \]
As a consequence, uniformly in $i,k$,
\begin{align*}
 \sum_{j\neq k}\mathbb{E}\big[W_{ij} & (R^{(k)}(z)^tR^{(k)}(z))_{ji}\big]\\
 & = \tau_N\mathbb{E}[\Tr\big(R^{(k)}(z)^tR^{(k)}(z)\big)R^{(k)}(z)_{ii}]+\sigma_N^2\mathbb{E}\big[\big(R^{(k)}(z)R^{(k)}(z)^tR^{(k)}(z)\big)_{ii}\big]\\
 & \quad + \tau_N\mathbb{E}\big[\big(^tR^{(k)}(z)R^{(k)}(z)^tR^{(k)}(z)\big)_{ii}\big]+\sigma_N^2\mathbb{E}[\Tr(R^{(k)}(z))(R^{(k)}(z)^tR^{(k)}(z))_{ii}] +O(\delta_N).
\end{align*}
Hence, uniformly in $i,k$,
\begin{align*}
 (z-\sigma_N^2 & \mathbb{E}[\Tr(R_N(z))]-D_{ii})\mathbb{E}[(R^{(k)}(z)^tR^{(k)}(z))_{ii}\big]\\
 & = \mathbb{E}[R^{(k)}(z)_{ii}]+\tau_N\mathbb{E}\big[\Tr\big(R^{(k)}(z)^tR^{(k)}(z)\big)R^{(k)}(z)_{ii}\big]+\sigma_N^2\mathbb{E}\big[\big(R^{(k)}(z)R^{(k)}(z)^tR^{(k)}(z)\big)_{ii}\big]\\
 & \quad +\tau_N\mathbb{E}\big[\big(^tR^{(k)}(z)R^{(k)}(z)^tR^{(k)}(z)\big)_{ii}\big]+O(\delta_N).
\end{align*}
Note that the last two terms are $O(N^{-1})$, uniformly in $i,k$. Moreover, it is known from \eqref{eq_approximation_Rk_tildeR} that $R^{(k)}(z)_{ii}$ may be replaced by $\tilde{R}(z)_{ii}$ at cost no more than $O(N^{-1/2})$. Therefore it follows 
\[\mathbb{E}[(R^{(k)}(z)^tR^{(k)}(z))_{ii}]-\tau_N\tilde{R}(z)_{ii}^2\mathbb{E}[\Tr(R^{(k)}(z)^tR^{(k)}(z))]=\tilde{R}(z)_{ii}^2+O(\delta_N),\]
uniformly in $i,k$. Then, uniformly in $k$, 
\[\big(1-\tau_N\sum_{i\neq k}\tilde{R}(z)_{ii}^2\big)\mathbb{E}[\Tr(R^{(k)}(z)^tR^{(k)}(z))]=\sum_{i\neq k}\tilde{R}(z)_{ii}^2+O(N\delta_N).\]
By Lemma \ref{lem_well_defined}, for $N$ large enough, one may invert $1-\tau_N\sum_{i\neq k}\tilde{R}(z)_{ii}^2$ which is bounded away from $0$, uniformly in $k$, which concludes the proof.
\end{proof}

Now going back to the formula \eqref{eq_bias}, we need to control the term $\sum_{i\neq k}\esp[R^{(k)}(z)_{ii}^2]$. We use again the fact (see \eqref{eq_approximation_Rk_tildeR}) that $R^{(k)}(z)_{ii}$ may be replaced by $\tilde{R}(z)_{ii}$ at cost no more than $O(N^{-1/2})$, uniformly in $k$. Therefore, uniformly in $k$, 
\[ \sum_{i\neq k}\esp[R^{(k)}(z)_{ii}^2]=\sum_{i\neq k}\tilde{R}(z)_{ii}^2+O(N^{1/2}). \]

The bias can now be computed entirely, from \eqref{eq_bias}.
\begin{align*}
 \esp [R_N(z)_{kk}] & = \tilde{R}(z)_{kk} + \tilde{R}(z)_{kk}^3\Big(s_N^2-\sigma_N^2 + \tau_N^2\frac{\sum_{i\neq k}\tilde{R}(z)_{ii}^2}{1-\tau_N\sum_{i\neq k}\tilde{R}(z)_{ii}^2} +\kappa_N\sum_{i\neq k}\tilde{R}(z)_{ii}^2\Big) + o(N^{-1}),
\end{align*}
where the $o(N^{-1})$ term is uniform in $k$.
Then
\begin{align*}
			\tilde{\beta}_N(z) & = \sum_{k=1}^N\tilde{R}(z)_{kk}^3\Big [s_N^2-\sigma_N^2+ \tau_N^2\frac{\sum_{i\neq k}\tilde{R}(z)_{ii}^2}{1-\tau_N\sum_{i\neq k}\tilde{R}(z)_{ii}^2} +\kappa_N\sum_{i\neq k}\tilde{R}(z)_{ii}^2 \Big] + o(1)\\
			& \underset{N\to +\infty}{\longrightarrow} \int_{\R}\frac{d\nu_{\infty}(x)}{(\omega(z)-x)^3}\Big[s^2-\sigma^2+\tau^2\frac{\int_{\R}\frac{d\nu_{\infty}(x)}{(\omega(z)-x)^2}}{1-\tau\int_{\R}\frac{d\nu_{\infty}(x)}{(\omega(z)-x)^2}} +\kappa\int_{\R}\frac{d\nu_{\infty}(x)}{(\omega(z)-x)^2} \Big]\\
			& =-\frac{\omega''(z)}{2\omega'(z)^3}\Big[\frac{s^2-\sigma^2-\tau}{\sigma^2}+\frac{\tau}{\sigma^2}\frac{\omega'(z)}{\frac{\tau}{\sigma^2}+(1-\frac{\tau}{\sigma^2})\omega'(z)}+\frac{\kappa}{\sigma^4}(1-\frac{1}{\omega'(z)}) \Big].
\end{align*}
\end{proof}

\begin{proof}[Proof of Proposition \ref{convb}]
The difference between $\beta_N(z)$ and $\tilde{\beta}_N(z)$ is
\begin{align*}
\beta_N(z)-\tilde{\beta}_N(z) & = N\big(G_{\nu_N}(\tilde{\omega}_N(z))-G_{\nu_N}(\omega_N(z))\big)\\
			  & = N\big(\omega_N(z)-\tilde{\omega}_N(z)\big)\int_{\R}\frac{d\nu_N(x)}{(\tilde{\omega}_N(z)-x)(\omega_N(z)-x)}\\
			  & = N\sigma_N^2\beta_N(z)\int_{\R}\frac{d\nu_N(x)}{(\tilde{\omega}_N(z)-x)(\omega_N(z)-x)}.
\end{align*}
Hence,
\begin{align*}
\beta_N(z) & = \big(1-N\sigma_N^2\int_{\R}\frac{d\nu_N(x)}{(\tilde{\omega}_N(z)-x)(\omega_N(z)-x)}\big)^{-1}\tilde{\beta}_N(z)\\
	& \underset{N \to +\infty}{\longrightarrow} \big(1-\sigma^2\int_{\R}\frac{d\nu_{\infty}(x)}{(\omega(z)-x)^2}\big)^{-1}\tilde{\beta}(z)=\frac{\tilde{\beta}(z)}{H'(\omega(z))}=\omega'(z)\tilde{\beta}(z).
\end{align*}

The limiting bias is therefore:

$$\beta(z)=-\frac{\omega''(z)}{2\omega'(z)^2}\Big[\frac{s^2-\sigma^2-\tau}{\sigma^2}+\frac{\tau}{\sigma^2}\frac{\omega'(z)}{\frac{\tau}{\sigma^2}+(1-\frac{\tau}{\sigma^2})\omega'(z)}+\frac{\kappa}{\sigma^4}(1-\frac{1}{\omega'(z)}) \Big].$$
\end{proof}

\subsection{Bias of linear spectral statistics of deformed generalized Wigner matrices}
The preceding subsection proves that $b_N(\varphi_z)=\beta_N(z)$ converges to $b(\varphi_z):=\beta(z)$. 
The purpose of this subsection is to extend this convergence result to $\varphi$ in a larger class of functions. 

We extend the convergence of $b_N(\varphi)$ to $\varphi \in \mathcal{H}_s$ for some $s>0$ by a strategy similar to Shcherbina's one for the fluctuations.

\begin{lem}\label{lem:extension_bias}
Let $(\mathcal{L},\| \, \|)$ be a normed vector space and 
$(b_N)_{N\geq 1}$ be a sequence of linear forms on $\mathcal{L}$. Assume that:
\begin{itemize}
\item there exists $C>0$ such that for any $\varphi \in \mathcal{L}$ and large enough $N\geq 1$, one has 
\[ |b_N(\varphi)| \leqslant C\|\varphi\|,\]
\item there exists a dense linear subspace $\mathcal{L}_1 \subset \mathcal{L}$ and a linear form $b : \mathcal{L}_1 \to \C$ such that $b_N(\varphi) \CV{N} b(\varphi)$, for all $\varphi \in \mathcal{L}_1$.
\end{itemize}
Then $b$ admits a unique continuous (linear) extension to $\mathcal{L}$ and 
\[ b_N(\varphi) \CV{N} b(\varphi)\]
holds for all $\varphi \in \mathcal{L}$.
\end{lem}
\begin{proof}
For $\varphi \in \mathcal{L}_1$, by letting $N$ tend to $+\infty$ in $|b_N(\varphi)| \leqslant C\|\varphi\|$, one obtains that $|b(\varphi)| \leqslant C\|\varphi\|$: $b$ is a continuous linear form on $\mathcal{L}_1$. Since $\mathcal{L}_1$ is dense in $\mathcal{L}$, $b$ can be uniquely extended into a continuous linear form on $\mathcal{L}$. 
Note that $|b(\varphi)|\leqslant C\|\varphi\|$ is true for all $\varphi \in \mathcal{L}$.
Let $\varphi \in \mathcal{L}$ and $\varepsilon >0$. By density, there exists $\varphi_1 \in \mathcal{L}_1$ such that $\|\varphi-\varphi_1\| \leqslant \varepsilon$. By linearity, for large enough $N \geq 1$, 
 \[|b_N(\varphi)-b(\varphi)| \leqslant 2C\varepsilon+|b_N(\varphi_1)-b(\varphi_1)|.\]
 But $b_N(\varphi_1) \CV{N} b(\varphi_1)$ and therefore $b_N(\varphi) \CV{N} b(\varphi)$.
\end{proof}
Here
\[\mathcal{L}_1=\text{span}\{\varphi_z:x\mapsto (z-x)^{-1}; z\in \C\setminus \R\}.\]
Recall from Section \ref{sec:Hs} that it is dense in the normed vector space $\mathcal{L}=\mathcal{H}_s,\, s>0.$ 

\begin{prop}\label{prop_bound_bias_Hs}
 Suppose $s>\frac{1}{2}$. Then, for all $\varphi \in \mathcal{H}_s$,
 \[|b_N(\varphi)| \leqslant  \frac{\|\varphi\|_{\mathcal{H}_s}}{2\pi^{3/2}\Gamma(2s)^{1/2}}\bigg(\int_0^{+\infty} e^{-y}y^{2s-1}\int_{\R}|\beta_N(x+iy)|^2dxdy\bigg)^{1/2}. \]
\end{prop}

This proposition provides a bound on $|b_N(\varphi)|$ for $\varphi \in \mathcal{H}_s$ of the expected form $C\|\varphi\|_{\mathcal{H}_s}$, where $C$ can be written as an integral of $\beta_N(x+iy)$. 

\begin{proof}
Set $\mu'_N:=N(\esp[\mu_N]-\rho_N)$. It is a signed measure, which writes as the difference of two finite positive measures. For bounded measurable $\varphi$, $b_N(\varphi)=\int_{\R}\varphi d\mu'_N$.
The restriction $\mathcal{D}_N$ of $b_N$ to the Schwartz space $\mathcal{S}$ is a tempered distribution. 
Its Fourier transform $\hat{\mathcal{D}}_N$ is also a tempered distribution defined by 
$$\langle \hat{\mathcal{D}}_N,f\rangle=\int_{\R}f(t)\int_{\R}e^{-itx}d\mu'_N(x)dt,\quad f\in \mathcal{S},$$
and satisfying 
\begin{equation}\label{inverseFourier}
\langle \mathcal{D}_N,f\rangle=\frac{1}{2\pi}\langle \hat{\mathcal{D}}_N,\hat{f}(-\cdot)\rangle,\quad f \in \mathcal{S}.
\end{equation}
In what follows, set $\alpha_N(t)=\int_{\R}e^{-itx}d\mu'_N(x)$. 
Remark that $\alpha_N:\mathbb{R}\to \mathbb{C}$ is a bounded continuous function. Then, for $\varphi \in \mathcal{S}$,
\begin{align*}
 |b_N(\varphi)| & = \frac{1}{2\pi}\bigg|\int_{\R}\hat{\varphi}(-t)\alpha_N(t)dt\bigg|\\
 & = \frac{1}{2\pi}\bigg|\int_{\R}(1+2|t|)^s\hat{\varphi}(-t)(1+2|t|)^{-s}\alpha_N(t)dt\bigg|\\
 & \leqslant \frac{1}{2\pi}\bigg(\int_{\R}(1+2|t|)^{2s}|\hat{\varphi}(t)|^2dt \int_{\R}(1+2|t|)^{-2s}|\alpha_N(t)|^2dt\bigg)^{1/2}\\
 & \leqslant \frac{1}{2\pi}\|\varphi\|_{\mathcal{H}_s}\bigg(\int_{\R}(1+2|t|)^{-2s}|\alpha_N(t)|^2dt\bigg)^{1/2}.
\end{align*}
As $\alpha_N$ is bounded and $s>\frac{1}{2}$, $\int_{\R}(1+2|t|)^{-2s}|\alpha_N(t)|^2dt<+\infty$. We just observed that the norm of the restriction of $b_N$ to $\mathcal{S}$ is bounded by $\int_{\R}(1+2|t|)^{-2s}|\alpha_N(t)|^2dt$. As $\mathcal{S}$ is dense in $\mathcal{H}_s$, the norm of $b_N: \mathcal{H}_s \to \mathbb{C}$ is the same.
Remark that 
\[ \Gamma(2s)(1+2|t|)^{-2s}=\int_0^{+\infty}e^{-(1+2|t|)y}y^{2s-1}dy.\]
As a consequence, by Fubini-Tonelli theorem, 
\begin{align*}
 \int_{\R}(1+2|t|)^{-2s}|\alpha_N(t)|^2dt & = \frac{1}{\Gamma(2s)}\int_{\R}\int_0^{+\infty}e^{-(1+2|t|)y}y^{2s-1}dy|\alpha_N(t)|^2dt\\
 & = \frac{1}{\Gamma(2s)}\int_0^{+\infty}e^{-y}y^{2s-1}\int_{\R}e^{-2|t|y}|\alpha_N(t)|^2dtdy\\
 & = \frac{2}{\Gamma(2s)}\int_0^{+\infty}e^{-y}y^{2s-1}\int_{0}^{+\infty}e^{-2ty}|\alpha_N(t)|^2dtdy,
\end{align*}
as $\alpha_N(-t)=\bar{\alpha}_N(t)$. Use now that $\beta_N(\cdot+iy) \in L^2$ with Fourier transform 
$$\widehat{\beta_N(\cdot+iy)}(t)=-2i\pi\alpha_N(t)e^{-ty}\mathbbm{1}_{(0,+\infty)}(t)$$
(this will be proved at the end). Then, applying Parseval identity, 
\begin{align*}
 \int_{\R}(1+2|t|)^{-2s}|\alpha_N(t)|^2dt & =  \frac{1}{\Gamma(2s)2\pi^2}\int_0^{+\infty}e^{-y}y^{2s-1}\int_{\R}|\widehat{\beta_N(\cdot+iy)}(t)|^2dtdy\\
 & = \frac{1}{\Gamma(2s)\pi}\int_0^{+\infty}e^{-y}y^{2s-1}\int_{\R}|\beta_N(x+iy)|^2dxdy.
\end{align*}
Consequently, for all $\varphi \in \mathcal{H}_s$,
\[|b_N(\varphi)| \leqslant \frac{\|\varphi\|_{\mathcal{H}_s}}{2\pi^{3/2}\Gamma(2s)^{1/2}}\Big(\int_0^{+\infty}e^{-y}y^{2s-1}\int_{\R}|\beta_N(x+iy)|^2dxdy\Big)^{1/2}. \]

It remains to prove that $\beta_N(\cdot+iy) \in L^2$ with Fourier transform 
$$\widehat{\beta_N(\cdot+iy)}(t)=-2i\pi\alpha_N(t)e^{-ty}\mathbbm{1}_{(0,+\infty)}(t).$$
\begin{align*}
 \int_{\R}|\beta_N(x+iy)|^2dx & = \int_{\R}\Big|\int_{\R}\frac{d\mu'_N(t)}{x+iy-t}\Big|^2dx\\
 & \leqslant \int_{\R}\Big(\int_{\R}\frac{d|\mu'_N|(t)}{|x+iy-t|}\Big)^2dx \quad \text{where} \ |\mu'_N|=N(\esp[\mu_N]+\rho_N)\\
 & \leqslant 4N^2\int_{\R}\Big(\int_{\R}\frac{1}{|x+iy-t|}\frac{d|\mu'_N|(t)}{2N}\Big)^2dx\\
 & \leqslant 4N^2\int_{\R}\int_{\R}\frac{1}{|x+iy-t|^2}\frac{d|\mu'_N|(t)}{2N}dx \quad \text{by Jensen inequality}\\
 & \leqslant 2N\int_{\R}\int_{\R}\frac{dx}{(x-t)^2+y^2}d|\mu'_N|(t) \quad \text{by Fubini-Tonelli theorem}\\
 & \leqslant 4N^2\int_{\R}\frac{dx}{x^2+y^2} < +\infty.
\end{align*}
Therefore $\widehat{\beta_N(\cdot+iy)}$ exists in $L^2$. Set 
$f(\xi)=-2i\pi e^{-y\xi}\alpha_N(\xi)\mathbbm{1}_{(0,+\infty)}(\xi)$ for $\xi \in \R$. 
As $\alpha_N$ is bounded on $\R$, $f \in L^1\cap L^2$. We compute its Fourier transform. 
For $x \in \R$,
\begin{align*}
 \hat{f}(x) & = -2i\pi\int_{\R}e^{-y\xi}\alpha_N(\xi)\mathbbm{1}_{(0,+\infty)}(\xi)e^{-ix\xi}d\xi\\
 & = -2i\pi\int_0^{+\infty}e^{-y\xi}\int_{\R}e^{-it\xi}d\mu'_N(t)e^{-ix\xi}d\xi\\
 & = -2i\pi\int_{\R}\int_0^{+\infty}e^{-(y+i(t+x))\xi}d\xi d\mu'_N(t)\\
 & = -2i\pi\int_{\R}\frac{d\mu'_N(t)}{y+i(t+x)}\\
 & = 2\pi\int_{\R}\frac{d\mu'_N(t)}{-x+iy-t}\\
 & = 2\pi \beta_N(-x+iy).
\end{align*}
Hence $\beta_N(\cdot+iy)=\frac{1}{2\pi}\hat{f}(-\cdot)$ and finally $\widehat{\beta_N(\cdot+iy)}=f$ in $L^2$, which concludes.
\end{proof}
Notice that $\tilde{R}_{kk}$ is the Stieltjes transform of a probability measure $\tilde{\mu}_k$. 
Hence, for $y>0$ and $p\geq 1$, by Jensen inequality and Fubini theorem,
\begin{align*}
\int_{\R}|\tilde{R}(x+iy)_{kk}|^pdx & = \int_{\R}\Big|\int_{\R}\frac{d\tilde{\mu}_{k}(t)}{x+iy-t}\Big|^pdx\\
				    & \leq \int_{\R}\int_{\R}\frac{d\tilde{\mu}_{k}(t)}{|x+iy-t|^p}dx\\
				    & \leq \int_{\R}\int_{\R}\frac{dx}{\big((x-t)^2+y^2\big)^{p/2}}d\tilde{\mu}_{k}(t)\\
				    & \leq y^{1-p}\int_{\R}\int_{\R}\frac{du}{(1+u^2)^{p/2}}d\tilde{\mu}_{k}(t)\\
				    & \leq y^{1-p}\int_{\R}\frac{du}{(1+u^2)^{p/2}}.
\end{align*}
It follows, by Proposition \ref{boundb} and Cauchy-Schwarz inequality, that
\begin{align*}
 \int_{\R}|\beta_N(x+iy)|^2dx & \leqslant N^{-2}|P_N(y^{-1})|^2\int_{\R}\Big(\sum_{k=1}^N|\tilde{R}(x+iy)_{kk}|^2\Big)^2dx\\
				 & \leqslant N^{-1}|P_N(y^{-1})|^2\sum_{k=1}^N\int_{\R}|\tilde{R}(x+iy)_{kk}|^4dx\\
				 & \leqslant y^{-3}|P_N(y^{-1})|^2\int_{\R}\frac{du}{(1+u^2)^{2}}.
\end{align*}
Then,
\begin{align*}
 |b_N(\varphi)| &\leqslant \frac{\|\varphi\|_{\mathcal{H}_s}}{2\pi^{3/2}\Gamma(2s)^{1/2}}\bigg(\int_0^{+\infty} e^{-y}y^{2s-1}\int_{\R}|\beta_N(x+iy)|^2dxdy\bigg)^{1/2}\\
		   &\leqslant \frac{\|\varphi\|_{\mathcal{H}_s}}{2\pi^{3/2}\Gamma(2s)^{1/2}}\bigg(\int_0^{+\infty} e^{-y}y^{2s-4}|P_N(y^{-1})|^2\int_{\R}\frac{du}{(1+u^2)^2}dy\bigg)^{1/2}\\
  & \leqslant \Big(\int_{\R}\frac{du}{(1+u^2)^2}\Big)^{1/2}\frac{\|\varphi\|_{\mathcal{H}_s}}{2\pi^{3/2}\Gamma(2s)^{1/2}}\Big(\int_0^{+\infty}e^{-y}y^{2s-4}|P_N(y^{-1})|^2dy\Big)^{1/2}\\
  & \leqslant C\|\varphi\|_{\mathcal{H}_s},
\end{align*}
where \[C:=\sup_{N\geq 1}\frac{1}{2\pi^{3/2}\Gamma(2s)^{1/2}}\Big(\int_{\R}\frac{du}{(1+u^2)^2}\Big)^{1/2}\Big(\int_0^{+\infty}e^{-y}y^{2s-4}|P_N(y^{-1})|^2dy\Big)^{1/2}<+\infty\]
as soon as $s>13/2$ (recall that $P_N$ is of degree $5$ with bounded coefficients). 
Therefore, for $s>13/2$, $b$ may be extended to $\mathcal{H}_s$ and 
$b_N(\varphi)$ converges to $b(\varphi)$ for every $\varphi \in \mathcal{H}_s$. 

\section{Fluctuations of the Stieltjes transform of the spectral measure of deformed Wigner matrices} \label{CLT}
Let \[\mathcal{N}_N(\varphi):=\sum_{\lambda \in \text{sp}(X_N)}\varphi(\lambda)=N\int_{\mathbb{R}}\varphi(x)\mu_N(dx).\]
We have seen previously in Section \ref{sec_variance_bound} that $(\Var \mathcal{N}_N(\varphi))_{N\geq 1}$ is bounded for 
$\varphi\in \mathcal{L}_1=\text{span}\{\varphi_z:x\mapsto (z-x)^{-1}; z\in \C\setminus \R\}$.
In this section, we will give a proof of the convergence in finite-dimensional distributions of the complex process $(\mathcal{N}_N(\varphi)-\mathbb{E}[\mathcal{N}_N(\varphi)], \, \varphi\in \mathcal{L}_1)$ to a centred complex Gaussian process, based on Theorem \ref{thm_CLT_martingale}.

\subsection{Reduction of the problem} \label{subsectionreduction}
 
One has to prove that, for a fixed $\varphi\in \mathcal{L}_1$, $\mathcal{N}_N(\varphi)-\mathbb{E}[\mathcal{N}_N(\varphi)]$ converges in distribution to a complex Gaussian variable $Z\sim \mathcal{N}_{\mathbb{C}}(0,V[\varphi],W[\varphi])$. 
Notice that 
$$\mathcal{N}_N(\varphi)-\mathbb{E}[\mathcal{N}_N(\varphi)]=\sum_{k=1}^N (\mathbb{E}_{\leq k}-\mathbb{E}_{\leq k-1})[\mathcal{N}_N(\varphi)].$$
The random variable $\mathcal{N}_N(\varphi)$ being bounded, $(\mathbb{E}_{\leq k}[\mathcal{N}_N(\varphi)])_{k\geq 1}$ is a square integrable complex martingale, hence $(\mathbb{E}_{\leq k}[\mathcal{N}_N(\varphi)]-\mathbb{E}_{\leq k-1}[\mathcal{N}_N(\varphi)])_{k\geq 1}$ is a martingale difference. Our strategy is to apply the central limit theorem for sums of martingale differences. More precisely, we will decompose $(\mathbb{E}_{\leq k}-\mathbb{E}_{\leq k-1})[\mathcal{N}_N(\varphi)]$ in two parts and apply Theorem \ref{thm_CLT_martingale} to the first part.

\begin{prop}
For $1 \leq k \leq N$,
\[(\mathbb{E}_{\leq k}-\mathbb{E}_{\leq k-1})[\mathcal{N}_N(\varphi)]=\Delta_k^{(N)}+\varepsilon_k^{(N)},\]
where $\Delta_k^{(N)}$ is a linear combination of 
$\esp_{\leq k}[-\frac{\partial}{\partial z}\phi_k^{(N)}(z)],\, z\in \mathbb{C}\setminus\mathbb{R}$, with
\[ \phi_k^{(N)}(z):=\frac{W_{kk}+C_k^{(k)*}R^{(k)}(z)C_k^{(k)}-\sigma_N^2\Tr(R^{(k)}(z))}{z-D_{kk}-\sigma_N^2\Tr(R^{(k)}(z))},\quad z\in \mathbb{C}\setminus\mathbb{R},\]
and $\sum_{k=1}^N\varepsilon_k^{(N)} \underset{N \to +\infty}{\longrightarrow} 0$ in probability.
\end{prop}

\begin{proof}
For $z\in \mathbb{C}\setminus\mathbb{R}$, decomposing further the first term of $\Psi_k$ in \eqref{psi} yields:
\begin{align*}
& \frac{\big(1+C_k^{(k)*}R^{(k)}(z)^2C_k^{(k)}\big)\big(W_{kk}+C_k^{(k)*}R^{(k)}(z)C_k^{(k)}-\sigma_N^2\Tr(R^{(k)}(z))\big)}{\big(z-W_{kk}-D_{kk}-C_k^{(k)*}R^{(k)}(z)C_k^{(k)}\big)\big(z-D_{kk}-\sigma_N^2\Tr(R^{(k)}(z))\big)}\\
& = \frac{\big(1+\sigma_N^2\Tr(R^{(k)}(z)^2)\big)\big(W_{kk}+C_k^{(k)*}R^{(k)}(z)C_k^{(k)}-\sigma_N^2\Tr(R^{(k)}(z))\big)}{\big(z-D_{kk}-\sigma_N^2\Tr(R^{(k)}(z))\big)^2}+\varepsilon_{k,1}^{(N)}(z)+\varepsilon_{k,2}^{(N)}(z)+\varepsilon_{k,3}^{(N)}(z),
\end{align*}
where
\begin{align*}
\varepsilon_{k,1}^{(N)}(z) & :=\frac{\big(1+\sigma_N^2\Tr(R^{(k)}(z)^2)\big)\big(W_{kk}+C_k^{(k)*}R^{(k)}(z)C_k^{(k)}-\sigma_N^2\Tr(R^{(k)}(z))\big)^2}{\big(z-W_{kk}-D_{kk}-C_k^{(k)*}R^{(k)}(z)C_k^{(k)}\big)\big(z-D_{kk}-\sigma_N^2\Tr(R^{(k)}(z))\big)^2},\\
\varepsilon_{k,2}^{(N)}(z) & :=\frac{\big(C_k^{(k)*}R^{(k)}(z)^2C_k^{(k)}-\sigma_N^2\Tr(R^{(k)}(z)^2)\big)\big(W_{kk}+C_k^{(k)*}R^{(k)}(z)C_k^{(k)}-\sigma_N^2\Tr(R^{(k)}(z))\big)^2}{\big(z-W_{kk}-D_{kk}-C_k^{(k)*}R^{(k)}(z)C_k^{(k)}\big)\big(z-D_{kk}-\sigma_N^2\Tr(R^{(k)}(z))\big)^2},\\
\varepsilon_{k,3}^{(N)}(z) & :=\frac{\big(C_k^{(k)*}R^{(k)}(z)^2C_k^{(k)}-\sigma_N^2\Tr(R^{(k)}(z)^2)\big)\big(W_{kk}+C_k^{(k)*}R^{(k)}(z)C_k^{(k)}-\sigma_N^2\Tr(R^{(k)}(z))\big)}{\big(z-D_{kk}-\sigma_N^2\Tr(R^{(k)}(z))\big)^2}.
\end{align*}
It follows that
\begin{align*}
\Tr(R_N(z)) & = \Tr(R^{(k)}(z))+\frac{1+\sigma_N^2\Tr(R^{(k)}(z)^2)}{z-D_{kk}-\sigma_N^2\Tr(R^{(k)}(z))}\\
& + \frac{\big(1+\sigma_N^2\Tr(R^{(k)}(z)^2)\big)\big(W_{kk}+C_k^{(k)*}R^{(k)}(z)C_k^{(k)}-\sigma_N^2\Tr(R^{(k)}(z))\big)}{\big(z-D_{kk}-\sigma_N^2\Tr(R^{(k)}(z))\big)^2}\\
& + \frac{C_k^{(k)*}R^{(k)}(z)^2C_k^{(k)}-\sigma_N^2\Tr(R^{(k)}(z)^2)}{z-D_{kk}-\sigma_N^2\Tr(R^{(k)}(z))}+\varepsilon_{k,1}^{(N)}(z)+\varepsilon_{k,2}^{(N)}(z)+\varepsilon_{k,3}^{(N)}(z).
\end{align*}
Recall (from \eqref{vanish}) that 
$$(\mathbb{E}_{\leq k}-\mathbb{E}_{\leq k-1})\Big[\Tr(R^{(k)}(z))+\frac{1+\sigma_N^2\Tr(R^{(k)}(z)^2)}{z-D_{kk}-\sigma_N^2\Tr(R^{(k)}(z))}\Big]=0.$$
By integrating first with respect to the $k$-th column of $W_N$, one has that 
\begin{align*}
\mathbb{E}_{\leq k-1} & \Big[\frac{\big(1+\sigma_N^2\Tr(R^{(k)}(z)^2)\big)\big(W_{kk}+C_k^{(k)*}R^{(k)}(z)C_k^{(k)}-\sigma_N^2\Tr(R^{(k)}(z))\big)}{\big(z-D_{kk}-\sigma_N^2\Tr(R^{(k)}(z))\big)^2}\Big]=0 ;\\
\mathbb{E}_{\leq k-1} & \Big[\frac{C_k^{(k)*}R^{(k)}(z)^2C_k^{(k)}-\sigma_N^2\Tr(R^{(k)}(z)^2)}{z-D_{kk}-\sigma_N^2\Tr(R^{(k)}(z))}\Big]=0.
\end{align*}
Observe now that
\begin{align*}
& \frac{\big(1+\sigma_N^2\Tr(R^{(k)}(z)^2)\big)\big(W_{kk}+C_k^{(k)*}R^{(k)}(z)C_k^{(k)}-\sigma_N^2\Tr(R^{(k)}(z))\big)}{\big(z-D_{kk}-\sigma_N^2\Tr(R^{(k)}(z))\big)^2} \\
& \hspace{8cm} +  \frac{C_k^{(k)*}R^{(k)}(z)^2C_k^{(k)}-\sigma_N^2\Tr(R^{(k)}(z)^2)}{z-D_{kk}-\sigma_N^2\Tr(R^{(k)}(z))}\\
= & -\frac{\partial}{\partial z}\frac{W_{kk}+C_k^{(k)*}R^{(k)}(z)C_k^{(k)}-\sigma_N^2\Tr(R^{(k)}(z))}{z-D_{kk}-\sigma_N^2\Tr(R^{(k)}(z))}=-\frac{\partial}{\partial z}\phi_k^{(N)}(z).
\end{align*}
Finally, 
\[(\mathbb{E}_{\leq k}-\mathbb{E}_{\leq k-1})[\mathcal{N}_N(\varphi)]=\Delta_k^{(N)}+\varepsilon_k^{(N)},\]
where $\Delta_k^{(N)}$ and $\varepsilon_k^{(N)}$ may be written as linear combinations respectively of $\mathbb{E}_{\leq k}[-\frac{\partial}{\partial z}\phi_k^{(N)}(z)],\, z \in \C\setminus \R$, and $(\mathbb{E}_{\leq k}-\mathbb{E}_{\leq k-1})[\varepsilon_{k,1}^{(N)}(z)+\varepsilon_{k,2}^{(N)}(z)+\varepsilon_{k,3}^{(N)}(z)]$, $z\in \C\setminus \R$. It remains to prove that $\sum_{k=1}^N\varepsilon_k^{(N)} \underset{N \to +\infty}{\longrightarrow} 0$ in probability. This is the object of next lemma (assuming that the entries of $W_N$ are bounded by $\delta_N$).
\end{proof}

\begin{lem}
 \[\Big\|\sum_{k\geq 1}\varepsilon_{k}^{(N)}\Big\|_{L^2} \underset{N \to +\infty}{\longrightarrow} 0.\]
\end{lem}
\begin{proof}
It follows from Lemma \ref{usefulbound} and Lemma \ref{lem_moment_quadratic_forms} that: 
\begin{align*}
\mathbb{E}\big[|\varepsilon_{k,1}^{(N)}(z)|^2\big] & \leq |\Im z|^{-6}\mathbb{E}\big[\big|W_{kk}+C_k^{(k)*}R^{(k)}(z)C_k^{(k)}-\sigma_N^2\Tr(R^{(k)}(z))\big|^4\big]\\
						   & \leq 2^3|\Im z|^{-6}\big(\mathbb{E}[W_{kk}^4]+\mathbb{E}[|C_k^{(k)*}R^{(k)}(z)C_k^{(k)}-\sigma_N^2\Tr(R^{(k)}(z))|^4]\big)\\
						   & \leq 2^3|\Im z|^{-6}\big(\delta_N^2s_N^2+K_{4}|\Im z|^{-4}\big((Nm_N)^2+\delta_N^{4}Nm_N\big)\big).
\end{align*}
Therefore, $\esp[|\varepsilon_{k,1}^{(N)}(z)|^2]=o(N^{-1})$, uniformly in $k$.

Since by Lemma \ref{lem_moment_quadratic_forms}
\[\mathbb{E}\big[\big|C_k^{(k)*}R^{(k)}(z)^2C_k^{(k)}-\sigma_N^2\Tr(R^{(k)}(z)^2)\big|^4\big]\leq K_{4}|\Im z|^{-8}\big((Nm_N)^{2}+\delta_N^{4}Nm_N\big),\]
and 
\begin{align*}
\mathbb{E}\big[\big|W_{kk}+C_k^{(k)*}R^{(k)}(z)C_k^{(k)}-\sigma_N^2 & \Tr(R^{(k)}(z))\big|^8\big]\\
& \leq 2^7\big(\mathbb{E}\big[|W_{kk}|^8\big]+\mathbb{E}\big[\big|C_k^{(k)*}R^{(k)}(z)C_k^{(k)}-\sigma_N^2\Tr(R^{(k)}(z))\big|^8\big]\big)\\
										    & \leq 2^7\big(\delta_N^{6}s_N^2+K_{8}|\Im z|^{-8}\big((Nm_N)^{4}+\delta_N^{12}Nm_N\big)\big),
\end{align*}
one obtains by Cauchy-Schwarz inequality that $\mathbb{E}[|\varepsilon_{k,2}^{(N)}(z)|^2]=o(N^{-1})$, uniformly in $k$.

Similarly,
\[\mathbb{E}\big[\big|C_k^{(k)*}R^{(k)}(z)^2C_k^{(k)}-\sigma_N^2\Tr(R^{(k)}(z)^2)\big|^4\big]\leq K_{4}|\Im z|^{-8}\big((Nm_N)^{2}+\delta_N^{4}Nm_N\big)\]
and
\begin{align*}
\mathbb{E}\big[\big|W_{kk}+C_k^{(k)*}R^{(k)}(z)C_k^{(k)}-\sigma_N^2 & \Tr(R^{(k)}(z))\big|^4\big]\\
& \leq 2^3\big(\mathbb{E}\big[|W_{kk}|^4\big]+\mathbb{E}\big[\big|C_k^{(k)*}R^{(k)}(z)C_k^{(k)}-\sigma_N^2\Tr(R^{(k)}(z))\big|^4\big]\big)\\
										    & \leq 2^3\Big(\delta_N^{2}s_N^2+K_{4}|\Im z|^{-4}\big((Nm_N)^{2}+\delta_N^{4}Nm_N\big)\Big)
\end{align*}
imply that $\mathbb{E}[|\varepsilon_{k,3}^{(N)}(z)|^2]=o(N^{-1})$, uniformly in $k$.

\

\noindent By Lemma \ref{martingalevariance} and Jensen inequality, 
\begin{align*} 
\mathbb{E}\big[\big|\sum_{k=1}^N(\mathbb{E}_{\leq k}-\mathbb{E}_{\leq k-1})\big[\varepsilon_{k,1}^{(N)}(z)+\varepsilon_{k,2}^{(N)}(z) & +\varepsilon_{k,3}^{(N)}(z)\big]\big|^2\big]\\
& = \sum_{k=1}^N\mathbb{E}\big[\big|(\mathbb{E}_{\leq k}-\mathbb{E}_{\leq k-1})\big[\varepsilon_{k,1}^{(N)}(z)+\varepsilon_{k,2}^{(N)}(z)+\varepsilon_{k,3}^{(N)}(z)\big]\big|^2\big]\\
					      &\leq \sum_{k=1}^N6\sum_{j=1}^3\big(\mathbb{E}\big[\big|(\mathbb{E}_{\leq k-1}\big[\varepsilon_{k,j}^{(N)}(z)\big]\big|^2\big]+\mathbb{E}\big[\big|(\mathbb{E}_{\leq k}\big[\varepsilon_{k,j}^{(N)}(z)\big]\big|^2\big]\big)\\
					      &\leq 12\sum_{k=1}^N\mathbb{E}\big[\big|\varepsilon_{k,1}^{(N)}(z)\big|^2+\big|\varepsilon_{k,2}^{(N)}(z)\big|^2+\big|\varepsilon_{k,3}^{(N)}(z)\big|^2\big]=o(1).
\end{align*}
Finally, by triangle inequality, 
$$\Big\|\sum_{k\geq 1}\varepsilon_{k}^{(N)}\Big\|_{L^2}=o(1).$$
\end{proof}

As announced at the beginning of the Section, the strategy is now to apply Theorem \ref{thm_CLT_martingale} to $\Delta_k^{(N)}$.

\subsection{Verification of Lyapounov condition} To check condition \eqref{L}, one first uses Markov inequality to get, for $p=2(1+\varepsilon)>2$:
$$ L(\varepsilon,N) \leq \varepsilon^{2-p}\sum_{k=1}^N\|\Delta_k^{(N)}\|_{L^p}^p.$$
It is therefore sufficient to prove that 
\begin{equation}\label{eq_cond_norm_p}
\|\Delta_k^{(N)}\|_{L^p}=O(N^{-\frac{1}{2}}),
\end{equation}
uniformly in $k$. By triangle inequality, it is sufficient to have 
\begin{equation*}
\sup_{1\leq k\leq N}\Big\|\mathbb{E}_{\leq k}\big[-\frac{\partial}{\partial z}\phi_k^{(N)}(z)\big]\Big\|_{L^p}=O(N^{-\frac{1}{2}}),
\end{equation*}
uniformly in $k$. By Jensen and triangular inequalities and Lemma \ref{usefulbound}, for $1\leq k\leq N$,
\begin{align*}
\Big\|\mathbb{E}_{\leq k}\big[-\frac{\partial}{\partial z}\phi_k^{(N)}(z)\big]\Big\|_{L^p} & \leq \bigg\| \frac{(1+\sigma_N^2\Tr(R^{(k)}(z)^2))(W_{kk}+C_k^{(k)*}R^{(k)}(z)C_k^{(k)}-\sigma_N^2\Tr(R^{(k)}(z)))}{(z-D_{kk}-\sigma_N^2\Tr(R^{(k)}(z)))^2}\bigg\|_{L^p}\\
& \quad + \ \bigg\| \frac{C_k^{(k)*}R^{(k)}(z)^2C_k^{(k)}-\sigma_N^2\Tr(R^{(k)}(z)^2)}{z-D_{kk}-\sigma_N^2\Tr(R^{(k)}(z))}\bigg\|_{L^p}\\
& \leq |\mathcal{I}z|^{-2}\big(\| W_{kk}\|_{L^p}+\| C_k^{(k)*}R^{(k)}(z)C_k^{(k)}-\sigma_N^2\Tr(R^{(k)}(z))\|_{L^p}\big)\\
& \quad + |\mathcal{I}z|^{-1}\big\| C_k^{(k)*}R^{(k)}(z)^2C_k^{(k)}-\sigma_N^2\Tr(R^{(k)}(z)^2)\big\|_{L^p}.
\end{align*}
One deduces \eqref{eq_cond_norm_p} because of the assumption $\| W_{kk}\|_{L^p}\leq C_2^{1/p}N^{-1/2}$ and of the improved bound \eqref{eq_improved_bound} (recall that $p=2(1+\varepsilon)$):
\[\big\| C_k^{(k)*}R^{(k)}(z)^qC_k^{(k)}-\sigma_N^2\Tr(R^{(k)}(z)^q)\big\|_{L^p}\leq K_{p}^{1/p}|\Im z|^{-q}\Big((Nm_N)^{p/2}+C_{4}N^{-1-2\varepsilon}\Big)^{1/p}.\]

\subsection{Convergence of the hook process} \label{subsectionhook}

\noindent By bilinearity, the verification of conditions \eqref{v} and \eqref{w} is equivalent to the convergence in probability of the {\em hook process}:
\[\Gamma_N(z_1,z_2):=\sum_{k=1}^N\mathbb{E}_{\leq k-1}\Big[\mathbb{E}_{\leq k}\big[\frac{\partial}{\partial z}\phi_k^{(N)}(z_1)\big]\mathbb{E}_{\leq k}\big[\frac{\partial}{\partial z}\phi_k^{(N)}(z_2)\big]\Big],\, z_1,z_2\in\mathbb{C}\setminus\mathbb{R}.\]

\begin{prop} \label{hook}
 For all $z_1,z_2\in\mathbb{C}\setminus\mathbb{R}$
 \[\Gamma_N(z_1,z_2) \underset{N\to +\infty}{\longrightarrow} \Gamma(z_1,z_2)\]
in probability, where 
 \begin{align*}
 \Gamma(z_1,z_2) & := \frac{\partial^2}{\partial z_1\partial z_2}\bigg[(s^2-\sigma^2-\tau)\int_{\mathbb{R}}\frac{\nu_{\infty}(dx)}{(\omega(z_1)-x)(\omega(z_2)-x)}+\frac{\kappa}{2}\Big(\int_{\mathbb{R}}\frac{\nu_{\infty}(dx)}{(\omega(z_1)-x)(\omega(z_2)-x)}\Big)^2\\
 		         & -\log\Big(1-\sigma^2\int_{\mathbb{R}}\frac{\nu_{\infty}(dx)}{(\omega(z_1)-x)(\omega(z_2)-x)}\Big)-\log\Big(1-\tau\int_{\mathbb{R}}\frac{\nu_{\infty}(dx)}{(\omega(z_1)-x)(\omega(z_2)-x)}\Big)\bigg]\\
 		         & = \frac{\partial^2}{\partial z_1\partial z_2}\bigg[\frac{s^2-\sigma^2-\tau}{\sigma^2}\Big(1-\frac{z_1-z_2}{\omega(z_1)-\omega(z_2)}\Big)+\frac{\kappa}{2\sigma^4}\Big(1-\frac{z_1-z_2}{\omega(z_1)-\omega(z_2)}\Big)^2\\
 		         & +\log\Big(\frac{\omega(z_1)-\omega(z_2)}{z_1-z_2}\Big)-\log\Big(1-\frac{\tau}{\sigma^2}\Big(1-\frac{z_1-z_2}{\omega(z_1)-\omega(z_2)}\Big)\Big)\bigg].
 \end{align*}
\end{prop}

Remark that, in the non deformed case, $\nu_{\infty}=\delta_0$, $\rho=\mu_{\sigma^2}$ and $\omega(z)=z-\sigma^2G_{\mu_{\sigma^2}}(z)=G_{\mu_{\sigma^2}}(z)^{-1}$. In this case, 
$\Gamma(z_1,z_2)$ coincides with $C_0(\varphi_{z_1},\varphi_{z_2})$ in Theorem \ref{BaoXieThm}.
 
In what follows, we focus on the convergence of 
\begin{equation*}\label{eq:hook_primitive}
\gamma_N(z_1,z_2):=\sum_{k=1}^N\mathbb{E}_{\leq k-1}\Big[\mathbb{E}_{\leq k}\big[\phi_k^{(N)}(z_1)\big]\mathbb{E}_{\leq k}\big[\phi_k^{(N)}(z_2)\big]\Big].
\end{equation*}
This will be enough to establish the convergence of the hook process, as explained in Subsection \ref{sec:CV_hook_ccl}.

\begin{lem}
\begin{align*}
\sum_{k=1}^N & \bigg\{\mathbb{E}_{\leq k-1}[\mathbb{E}_{\leq k}[\phi_k^{(N)}(z_1)] \mathbb{E}_{\leq k}[\phi_k^{(N)}(z_2)]]-\tilde{R}(z_1)_{kk}\tilde{R}(z_2)_{kk}\Big(s_N^2\\
& +\mathbb{E}_{\leq k-1}[\mathbb{E}_{\leq k}  [C_k^{(k)*}R^{(k)}(z_1)C_k^{(k)}-\sigma_N^2\Tr(R^{(k)}(z_1))]\mathbb{E}_{\leq k}[C_k^{(k)*}R^{(k)}(z_2)C_k^{(k)}-\sigma_N^2\Tr(R^{(k)}(z_2))]]\Big)\bigg\}\\
& \underset{N\to +\infty}{\longrightarrow}0
\end{align*}
in probability.
\end{lem}
\begin{proof}
Define 
\begin{align*}
\varepsilon_{k,4}^{(N)}(z) & := \phi_k^{(N)}(z)-\tilde{R}(z)_{kk}\big(W_{kk}+C_k^{(k)*}R^{(k)}(z)C_k^{(k)}-\sigma_N^2\Tr(R^{(k)}(z))\big)\\
			& = \frac{\sigma_N^2(\Tr(R^{(k)}(z))-\mathbb{E}[\Tr(R_N(z))]\big)\big(W_{kk}+C_k^{(k)*}R^{(k)}(z)C_k^{(k)}-\sigma_N^2\Tr(R^{(k)}(z)))}{\big(z-D_{kk}-\sigma_N^2\Tr(R^{(k)}(z))\big)\big(z-D_{kk}-\sigma_N^2\mathbb{E}[\Tr(R_N(z))]\big)}.
\end{align*}
Remark from \eqref{variancebound1} that it satisfies
\begin{align*}
\mathbb{E}\big[|\mathbb{E}_{\leq k} & [\varepsilon_{k,4}^{(N)} (z)]|^2\big]\\
&\leq \frac{\sigma_N^4}{|\mathcal{I}z|^{4}}\Big(s_N^2+\frac{N}{|\Im z|^{2}}(\sigma_N^4+\tau_N^2+\kappa_N)\Big) \mathbb{E}[|\Tr(R^{(k)}(z))-\mathbb{E}[\Tr(R_N(z))]|^2\big]\\
						        &\leq \frac{\sigma_N^4}{|\mathcal{I}z|^{4}}\big(s_N^2+\frac{N}{|\Im z|^{2}}(\sigma_N^4+\tau_N^2+\kappa_N)\big)\big(\Var[\Tr(R^{(k)}(z))]+\big|\mathbb{E}[\Tr(R^{(k)}(z))-\Tr(R_N(z))]\big|^2\big)\\
						        &\leq \frac{\sigma_N^4}{|\mathcal{I}z|^{4}}\big(s_N^2+\frac{N}{|\Im z|^{2}}(\sigma_N^4+\tau_N^2+\kappa_N)\big)\big(\frac{4N}{|\mathcal{I}z|^2}+\frac{1}{|\mathcal{I}z|^2}\big)=O(N^{-2}),
\end{align*}
uniformly in $k$ (see the comment following Proposition \ref{variancebound2}). Furthermore, from Lemma \ref{quadratic forms}, 
\[\esp[|\phi_k^{(N)}(z)|^2]\leq |\Im z|^{-2}(s_N^2+Nm_N|\Im z|^{-2}),\]
and
\[\esp[|\tilde{R}(z)_{kk}\big(W_{kk}+C_k^{(k)*}R^{(k)}(z)C_k^{(k)}-\sigma_N^2\Tr(R^{(k)}(z))|^2] \leq |\Im z|^{-2}(s_N^2+Nm_N|\Im z|^{-2}).\]
Remark now that
\begin{align*}
& \sum_{k=1}^N  \bigg\{\mathbb{E}_{\leq k-1}[\mathbb{E}_{\leq k}[\phi_k^{(N)}(z_1)] \mathbb{E}_{\leq k}[\phi_k^{(N)}(z_2)]]-\tilde{R}(z_1)_{kk}\tilde{R}(z_2)_{kk}\Big(s_N^2\\
& \hspace{0.5cm} +\mathbb{E}_{\leq k-1}[\mathbb{E}_{\leq k}  [C_k^{(k)*}R^{(k)}(z_1)C_k^{(k)}-\sigma_N^2\Tr(R^{(k)}(z_1))]\mathbb{E}_{\leq k}[C_k^{(k)*}R^{(k)}(z_2)C_k^{(k)}-\sigma_N^2\Tr(R^{(k)}(z_2))]]\Big)\bigg\}\\
& = \sum_{k=1}^N \bigg\{\mathbb{E}_{\leq k-1}[\mathbb{E}_{\leq k}[\varepsilon_{k,4}^{(N)}(z_1)] \mathbb{E}_{\leq k}[\phi_k^{(N)}(z_2)]]\\
& \hspace{0.5 cm} +  \tilde{R}(z_1)_{kk}\mathbb{E}_{\leq k-1}[\mathbb{E}_{\leq k}\big[W_{kk}+C_k^{(k)*}R^{(k)}(z_1)C_k^{(k)}-\sigma_N^2\Tr(R^{(k)}(z_1))\big]\mathbb{E}_{\leq k}[\varepsilon_{k,4}^{(N)}(z_2)] \big]\bigg\}.
\end{align*}
By triangle inequality, the $L^1$ norm of this expression is bounded by
\begin{align*}
 & = \sum_{k=1}^N \bigg\{\mathbb{E}[\big|\mathbb{E}_{\leq k}[\varepsilon_{k,4}^{(N)}(z_1)]\phi_k^{(N)}(z_2)]\big|\\
& \hspace{1 cm} +  \mathbb{E}[|\tilde{R}(z_1)_{kk}\big(W_{kk}+C_k^{(k)*}R^{(k)}(z_1)C_k^{(k)}-\sigma_N^2\Tr(R^{(k)}(z_1))\big)\mathbb{E}_{\leq k}[\varepsilon_{k,4}^{(N)}(z_2)]\big|]\bigg\}.
\end{align*}
Using Cauchy-Schwarz inequality yields that this $L^1$ norm is $O(N^{-1/2})$. Therefore, this term goes to $0$ in probability.
\end{proof}

Recall that, by Lemma \ref{quadratic forms}:
\begin{align*}
&s_N^2+\mathbb{E}_{\leq k-1} \Big[\mathbb{E}_{\leq k}\big[C_k^{(k)*}R^{(k)}(z_1)C_k^{(k)}-\sigma_N^2\Tr(R^{(k)}(z_1))\big]\mathbb{E}_{\leq k}\big[C_k^{(k)*}R^{(k)}(z_2)C_k^{(k)}-\sigma_N^2\Tr(R^{(k)}(z_2))\big]\Big]\\
 &=s_N^2+\sigma_N^4\sum_{i,j<k}\mathbb{E}_{\leq k-1}[R^{(k)}(z_1)_{ij}]\mathbb{E}_{\leq k-1}[R^{(k)}(z_2)_{ji}]\\
 & \hspace{0.9cm} +\tau_N^2\sum_{i,j<k}\mathbb{E}_{\leq k-1}[R^{(k)}(z_1)_{ij}]\mathbb{E}_{\leq k-1}[R^{(k)}(z_2)_{ij}]\\
 & \hspace{0.9cm} +\kappa_N\sum_{i<k}\mathbb{E}_{\leq k-1}[R^{(k)}(z_1)_{ii}]\mathbb{E}_{\leq k-1}[R^{(k)}(z_2)_{ii}].
\end{align*}
Therefore, what remains to study is the sum of four terms. They will be studied separately in the following paragraphs. The first and the fourth terms are studied very easily, whereas the second and third ones need quite long computations, making repeated use of linear algebra properties which were collected in Section \ref{sec:preliminary_results}. These second and third terms are very similar.

\subsubsection{Contribution of the first term}
Notice that 
\begin{eqnarray*}
\frac{1}{N}\sum_{k=1}^N\tilde{R}(z_1)_{kk}\tilde{R}(z_2)_{kk}&=&\int_{\mathbb{R}}\frac{\nu_N(dx)}{(\omega_N(z_1)-x)(\omega_N(z_2)-x)}\\
									        &\underset{N\to +\infty}{\longrightarrow} &\int_{\mathbb{R}}\frac{\nu_{\infty}(dx)}{(\omega(z_1)-x)(\omega(z_2)-x)}.
\end{eqnarray*}
Hence, $$s_N^2\sum_{k=1}^N\tilde{R}(z_1)_{kk}\tilde{R}(z_2)_{kk}\underset{N\to +\infty}{\longrightarrow} s^2\int_{\mathbb{R}}\frac{\nu_{\infty}(dx)}{(\omega(z_1)-x)(\omega(z_2)-x)}.$$

\subsubsection{Contribution of the second term} \label{sec_2nd_term}
The second term writes as follows.
\[\sigma_N^4\sum_{k=1}^N\tilde{R}(z_1)_{kk}\tilde{R}(z_2)_{kk}\sum_{i,j<k}\esp_{\leq k-1}[R^{(k)}(z_1)_{ij}]\esp_{\leq k-1}[R^{(k)}(z_2)_{ji}]. \]
To handle it, use first the definition of the resolvent $R^{(k)}(z_1)$.
\[(z_1-D_{ii})R^{(k)}(z_1)_{ij} = \delta_{ij}+\sum_{l\neq k}W_{il}R^{(k)}(z_1)_{lj}.\]
We want to remove the dependence between $W_{il}$ and $R^{(k)}(z_1)$, using \eqref{remove}. 
Plug now the following expression
\begin{align*}
(z_1-D_{ii})R^{(k)}(z_1)_{ij} = \delta_{ij} & + \sum_{l\neq k}\Big\{W_{il}R^{(kil)}(z_1)_{lj}\\
                                            &  + (1-\frac{1}{2}\delta_{il})\Big(W_{il}^2R^{(kil)}(z_1)_{li}R^{(k)}(z_1)_{lj}+|W_{il}|^2R^{(kil)}(z_1)_{ll}R^{(k)}(z_1)_{ij}\Big)\Big\},
\end{align*}
to get
\begin{align*}
(z_1-D_{ii})\mathbb{E}_{\leq k-1}[R^{(k)}(z_1)_{ij}]& \mathbb{E}_{\leq k-1}[R^{(k)}(z_2)_{ji}]\\
& = \delta_{ij}\mathbb{E}_{\leq k-1}[R^{(k)}(z_2)_{ji}]\\
& \quad +\sum_{l<k}W_{il}\mathbb{E}_{\leq k-1}[R^{(kil)}(z_1)_{lj}]\mathbb{E}_{\leq k-1}[R^{(k)}(z_2)_{ji}]\\
& \quad +\sum_{l\neq k}(1-\frac{1}{2}\delta_{il})\mathbb{E}_{\leq k-1}[W_{il}^2R^{(kil)}(z_1)_{li}R^{(k)}(z_1)_{lj}]\mathbb{E}_{\leq k-1}[R^{(k)}(z_2)_{ji}]\\
& \quad +\sum_{l\neq k}(1-\frac{1}{2}\delta_{il})\mathbb{E}_{\leq k-1}[|W_{il}|^2R^{(kil)}(z_1)_{ll}R^{(k)}(z_1)_{ij}]\mathbb{E}_{\leq k-1}[R^{(k)}(z_2)_{ji}].
\end{align*}
Remark that $R^{(k)}(z_1)_{ij} \esp_{\leq k-1}[R^{(k)}(z_2)_{ji}]$ appears in the last term of the right-hand side. Heuristically, imagine that in this last term $|W_{il}|^2,\, l\neq i,$ is close to its expectation $\sigma_N^2$ and $R^{(kil)}(z_1)_{ll}\approx \mathbb{E}[R_N(z_1)_{ll}]$. The whole term is therefore close to $\sigma_N^2\mathbb{E}[\Tr(R_N(z_1))]\mathbb{E}_{\leq k-1}[R^{(k)}(z_1)_{ij}]\mathbb{E}_{\leq k-1}[R^{(k)}(z_2)_{ji}]$. 

To make this argument rigorous, subtract $\sigma_N^2\mathbb{E}[\Tr(R_N(z_1))]\mathbb{E}_{\leq k-1}[R^{(k)}(z_1)_{ij}]\mathbb{E}_{\leq k-1}[R^{(k)}(z_2)_{ji}]$ to get:
\begin{align*}
(\tilde{\omega}_N(z_1)-D_{ii})&\mathbb{E}_{\leq k-1}[R^{(k)}(z_1)_{ij}] \mathbb{E}_{\leq k-1}[R^{(k)}(z_2)_{ji}]\\
& = \delta_{ij}\mathbb{E}_{\leq k-1}[R^{(k)}(z_2)_{ji}] +\sum_{l<k}W_{il}\mathbb{E}_{\leq k-1}[R^{(kil)}(z_1)_{lj}]\mathbb{E}_{\leq k-1}[R^{(k)}(z_2)_{ji}]\\
& \quad +\sum_{l\neq k}(1-\frac{1}{2}\delta_{il})\mathbb{E}_{\leq k-1}[W_{il}^2R^{(kil)}(z_1)_{li}R^{(k)}(z_1)_{lj}]\mathbb{E}_{\leq k-1}[R^{(k)}(z_2)_{ji}]\\
& \quad +\mathbb{E}_{\leq k-1}[\sum_{l\neq k}(1-\frac{1}{2}\delta_{il})(|W_{il}|^2R^{(kil)}(z_1)_{ll}-\sigma_N^2\mathbb{E}[R_N(z_1)_{ll}])R^{(k)}(z_1)_{ij}]\mathbb{E}_{\leq k-1}[R^{(k)}(z_2)_{ji}]\\
& \quad -\frac{\sigma_N^2}{2}\mathbb{E}[R_N(z_1)_{ii}]\mathbb{E}_{\leq k-1}[R^{(k)}(z_1)_{ij}]\mathbb{E}_{\leq k-1}[R^{(k)}(z_2)_{ji}]\\
& \quad -\sigma_N^2\mathbb{E}[R_N(z_1)_{kk}]\mathbb{E}_{\leq k-1}[R^{(k)}(z_1)_{ij}]\mathbb{E}_{\leq k-1}[R^{(k)}(z_2)_{ji}].
\end{align*}
As $\tilde{R}(z_1)_{ii}=(\tilde{\omega}_N(z_1)-D_{ii})^{-1}$,
\begin{align*}
\mathbb{E}_{\leq k-1} [& R^{(k)}(z_1)_{ij}]\mathbb{E}_{\leq k-1}[R^{(k)}(z_2)_{ji}]\\
& =\delta_{ij}\tilde{R}(z_1)_{ii}\mathbb{E}_{\leq k-1}[R^{(k)}(z_2)_{ji}] +\tilde{R}(z_1)_{ii}\sum_{l<k}W_{il}\mathbb{E}_{\leq k-1}[R^{(kil)}(z_1)_{lj}]\mathbb{E}_{\leq k-1}[R^{(k)}(z_2)_{ji}]\\
& + \quad \sum_{l\neq k}(1-\frac{1}{2}\delta_{il})\tilde{R}(z_1)_{ii}\mathbb{E}_{\leq k-1}[W_{il}^2R^{(kil)}(z_1)_{li}R^{(k)}(z_1)_{lj}]\mathbb{E}_{\leq k-1}[R^{(k)}(z_2)_{ji}]\\
& + \quad \mathbb{E}_{\leq k-1}[\sum_{l\neq k}(1-\frac{1}{2}\delta_{il})\tilde{R}(z_1)_{ii}(|W_{il}|^2R^{(kil)}(z_1)_{ll}-\sigma_N^2\mathbb{E}[R_N(z_1)_{ll}])R^{(k)}(z_1)_{ij}]\mathbb{E}_{\leq k-1}[R^{(k)}(z_2)_{ji}]\\
& - \quad \frac{\sigma_N^2}{2}\tilde{R}(z_1)_{ii}\mathbb{E}[R_N(z_1)_{ii}]\mathbb{E}_{\leq k-1}[R^{(k)}(z_1)_{ij}]\mathbb{E}_{\leq k-1}[R^{(k)}(z_2)_{ji}]\\
& - \quad \sigma_N^2\tilde{R}(z_1)_{ii}\mathbb{E}[R_N(z_1)_{kk}]\mathbb{E}_{\leq k-1}[R^{(k)}(z_1)_{ij}]\mathbb{E}_{\leq k-1}[R^{(k)}(z_2)_{ji}].
\end{align*}
It may be seen that the last four terms of the right-hand side will not asymptotically contribute. This is the object of next Lemma whose proof is postponed to the end of the paragraph.
\begin{lem}\label{lem_2nd_term_simplification1}
\begin{align*}
  \sum_{i,j<k}&\esp_{\leq k-1}[R^{(k)}(z_1)_{ij}]\esp_{\leq k-1}[R^{(k)}(z_2)_{ji}]\\
  & = \sum_{i<k}\tilde{R}(z_1)_{ii}\mathbb{E}_{\leq k-1}[R^{(k)}(z_2)_{ii}] + \sum_{i,j<k}\tilde{R}(z_1)_{ii}\sum_{l<k}W_{il}\mathbb{E}_{\leq k-1}[R^{(kil)}(z_1)_{lj}]\mathbb{E}_{\leq k-1}[R^{(k)}(z_2)_{ji}] + \varepsilon_{k,5}^{(N)},
 \end{align*}
 with $\sigma_N^4\sum_{k=1}^N\tilde{R}(z_1)_{kk}\tilde{R}(z_2)_{kk}\varepsilon_{k,5}^{(N)} \underset{N \to +\infty}{\longrightarrow} 0$ in probability.
\end{lem}
It remains to analyze the first two terms of the right-hand side. Let us focus on the second one. Use again \eqref{remove} to remove the dependence between $W_{il}$ and $R^{(k)}(z_2)$.
\begin{align*}
 \sum_{i,j<k}\tilde{R}(z_1)_{ii} \sum_{l<k} & W_{il}\mathbb{E}_{\leq k-1}[R^{(kil)}(z_1)_{lj}]\mathbb{E}_{\leq k-1}[R^{(k)}(z_2)_{ji}]\\
 & = \sum_{i,j<k}\tilde{R}(z_1)_{ii}\sum_{l<k}W_{il}\mathbb{E}_{\leq k-1}[R^{(kil)}(z_1)_{lj}]\mathbb{E}_{\leq k-1}[R^{(kil)}(z_2)_{ji}]\\
 & + \sum_{i,j<k}\tilde{R}(z_1)_{ii}\sum_{l<k}(1-\frac{1}{2}\delta_{il})W_{il}^2\mathbb{E}_{\leq k-1}[R^{(kil)}(z_1)_{lj}]\mathbb{E}_{\leq k-1}[R^{(kil)}(z_2)_{ji}R^{(k)}(z_2)_{li}]\\
 & + \sum_{i,j<k}\tilde{R}(z_1)_{ii}\sum_{l<k}(1-\frac{1}{2}\delta_{il})|W_{il}|^2\mathbb{E}_{\leq k-1}[R^{(kil)}(z_1)_{lj}]\mathbb{E}_{\leq k-1}[R^{(kil)}(z_2)_{jl}R^{(k)}(z_2)_{ii}].
\end{align*}
Heuristically, consider that $|W_{il}|^2,\, i\neq l,$ is close to its expectation $\sigma_N^2$, $R^{(k)}(z_2)_{ii}\approx \tilde{R}(z_2)_{ii}$ and replace $R^{(kil)}$ by $R^{(k)}$. Then the last term of the right-hand side becomes 
\[\sigma_N^2\Big(\sum_{i<k}\tilde{R}(z_1)_{ii}\tilde{R}(z_2)_{ii}\Big)\Big(\sum_{l,j<k}\mathbb{E}_{\leq k-1}[R^{(k)}(z_1)_{lj}]\mathbb{E}_{\leq k-1}[R^{(k)}(z_2)_{jl}]\Big).\]

Subtract this quantity to both sides of equation in Lemma \ref{lem_2nd_term_simplification1}. This yields the following Lemma, whose proof is postponed to the end of the paragraph.
\begin{lem}\label{lem_2nd_term_simplification2}
\begin{align*}
 (1-\sigma_N^2\sum_{i<k}\tilde{R}(z_1)_{ii}\tilde{R}(z_2)_{ii})\sum_{i,j<k}\mathbb{E}_{\leq k-1}[R^{(k)}(z_1)_{ij}] & \mathbb{E}_{\leq k-1}[R^{(k)}(z_2)_{ji}]\\
 & = \sum_{i<k}\tilde{R}(z_1)_{ii}\mathbb{E}_{\leq k-1}[R^{(k)}(z_2)_{ii}]+\varepsilon_{k,6}^{(N)},
\end{align*}
with $\sigma_N^4\sum_{k=1}^N\tilde{R}(z_1)_{kk}\tilde{R}(z_2)_{kk}\varepsilon_{k,6}^{(N)} \underset{N \to +\infty}{\longrightarrow} 0$ in probability.
\end{lem}
We are ready now to derive the contribution of this whole second term.

\begin{prop}\label{prop_2nd_term} The following convergence holds in probability:
\begin{align*}
& \sigma_N^4\sum_{k=1}^N\tilde{R}(z_1)_{kk}\tilde{R}(z_2)_{kk}\sum_{i,j<k}\mathbb{E}_{\leq k-1}[R^{(k)}(z_1)_{ij}]\mathbb{E}_{\leq k-1}[R^{(k)}(z_2)_{ji}]\\
& \underset{N \to +\infty}{\longrightarrow} -\log\Big(1-\sigma^2\int_{\mathbb{R}}\frac{\nu_{\infty}(dx)}{(\omega(z_1)-x)(\omega(z_2)-x)}\Big)-\sigma^2\int_{\mathbb{R}}\frac{\nu_{\infty}(dx)}{(\omega(z_1)-x)(\omega(z_2)-x)} .
\end{align*}
\end{prop}
\begin{proof}
We follow Section B.5.3 of \cite{JiLee}. Set $$f_{j,k,N}(t):=\frac{\sigma_N^2\tilde{R}(z_1)_{jj}\tilde{R}(z_2)_{jj}}{1-\sigma_N^2\sum_{i<k}\tilde{R}(z_1)_{ii}\tilde{R}(z_2)_{ii}-t\sigma_N^2\tilde{R}(z_1)_{kk}\tilde{R}(z_2)_{kk}}.$$
From Lemma \ref{lem_well_defined}, for large enough $N$, $f_{j,k,N}$ is well-defined and bounded on $[0;1]$, uniformly in $j,k$, by $C\sigma_N^2$. Moreover, $\mathbb{E}_{\leq k-1}[R^{(k)}(z)_{ii}]-\tilde{R}(z)_{ii}\underset{N\to +\infty}{\longrightarrow} 0$ in $L^2$, uniformly in $i,k$ (because of Jensen inequality). 
Therefore, Lemma \ref{lem_2nd_term_simplification2} yields the following:
\begin{align*}
 \sigma_N^2\sum_{i,j<k}\mathbb{E}_{\leq k-1}[R^{(k)}(z_1)_{ij}]\mathbb{E}_{\leq k-1}[R^{(k)}(z_2)_{ji}]
 & = \sum_{i<k}f_{i,k,N}(0)+\varepsilon_{k,7}^{(N)},
\end{align*}
with $E_N:=\sigma_N^2\sum_{k=1}^N\tilde{R}(z_1)_{kk}\tilde{R}(z_2)_{kk}\varepsilon_{k,7}^{(N)} \underset{N \to +\infty}{\longrightarrow} 0$ in probability.

As a consequence,
\begin{align*}
 \sigma_N^4\sum_{k=1}^N\tilde{R}(z_1)_{kk}\tilde{R}(z_2)_{kk} & \sum_{i,j<k}\mathbb{E}_{\leq k-1}[R^{(k)}(z_1)_{ij}]\mathbb{E}_{\leq k-1}[R^{(k)}(z_2)_{ji}]\\
 & = \sum_{k=1}^N\sigma_N^2\tilde{R}(z_1)_{kk}\tilde{R}(z_2)_{kk}\frac{\sigma_N^2\sum_{i<k}\tilde{R}(z_1)_{ii}\tilde{R}(z_2)_{ii}}{1-\sigma_{N}^2\sum_{i<k}\tilde{R}(z_1)_{ii}\tilde{R}(z_2)_{ii}} +E_N\\
 & = \sum_{k=1}^Nf_{k,k,N}(0)-\sigma_N^2\sum_{k=1}^N\tilde{R}(z_1)_{kk}\tilde{R}(z_2)_{kk}+E_N.
\end{align*}
For $t\in[0;1]$,
\[|f_{k,k,N}(0) -f_{k,k,N}(t)|=|tf_{k,k,N}(0)f_{k,k,N}(t)|\leq C^2\sigma_N^4.\]
Integrating with respect to $t\in[0;1]$ and summing on $k$, one obtains:
\[\Big|\sum_{k=1}^Nf_{k,k,N}(0)+\log(1-\sigma_N^2\sum_{i=1}^N\tilde{R}(z_1)_{ii}\tilde{R}(z_2)_{ii})\Big|\leq C^2N\sigma_N^4.\]
Finally, 
\begin{align*}
& \sigma_N^4\sum_{k=1}^N\tilde{R}(z_1)_{kk}\tilde{R}(z_2)_{kk}\sum_{i,j<k}\mathbb{E}_{\leq k-1}[R^{(k)}(z_1)_{ij}]\mathbb{E}_{\leq k-1}[R^{(k)}(z_2)_{ji}]\\
& \underset{N\to +\infty}{\longrightarrow} -\log\Big(1-\sigma^2\int_{\mathbb{R}}\frac{\nu_{\infty}(dx)}{(\omega(z_1)-x)(\omega(z_2)-x)}\Big)-\sigma^2\int_{\mathbb{R}}\frac{\nu_{\infty}(dx)}{(\omega(z_1)-x)(\omega(z_2)-x)} .
\end{align*}
\end{proof}

It remains to prove Lemma \ref{lem_2nd_term_simplification1} and Lemma \ref{lem_2nd_term_simplification2}.

\begin{proof}[Proof of Lemma \ref{lem_2nd_term_simplification1}]
 Recall that
 \begin{align*}
  \varepsilon_{k,5}^{(N)} & = \sum_{i,j<k}\sum_{l \neq k}(1-\frac{1}{2}\delta_{il})\tilde{R}(z_1)_{ii}\esp_{\leq k-1}[W_{il}^2R^{(kil)}(z_1)_{li}R^{(k)}(z_1)_{lj}\esp_{\leq k-1}[R^{(k)}(z_2)_{ji}]]\\
   & + \sum_{i,j<k}\esp_{\leq k-1}[\sum_{l\neq k}(1-\frac{1}{2}\delta_{il})\tilde{R}(z_1)_{ii}\Big(|W_{il}|^2R^{(kil)}(z_1)_{ll}-\sigma_N^2\esp[R_N(z_1)_{ll}]\Big)R^{(k)}(z_1)_{ij}\esp_{\leq k-1}[R^{(k)}(z_2)_{ji}]]\\
   & - \frac{\sigma_N^2}{2}\sum_{i,j<k}\tilde{R}(z_1)_{ii}\esp[R_N(z_1)_{ii}]\esp_{\leq k-1}[R^{(k)}(z_1)_{ij}\esp_{\leq k-1}[R^{(k)}(z_2)_{ji}]]\\
   & - \sigma_N^2\sum_{i,j<k}\tilde{R}(z_1)_{ii}\esp[R_N(z_1)_{kk}]\esp_{\leq k-1}[R^{(k)}(z_1)_{ij}\esp_{\leq k-1}[R^{(k)}(z_2)_{ji}]].
 \end{align*}
 It writes as a sum of four terms, which will be denoted by $a_k^{(N)}$, $b_k^{(N)}$, $c_k^{(N)}$ and $d_k^{(N)}$.
 
 From Lemma \ref{technicalbound}, note that, in all terms,
 \[ \big|\sum_{j<k}R^{(k)}(z_1)_{ij}\esp_{\leq k-1}[R^{(k)}(z_2)_{ji}]\big| \leq \|R^{(k)}(z_1) \| \|\esp_{\leq k-1}[R^{(k)}(z_2)]\| \leq \frac{1}{|\Im z_1 \Im z_2|}. \]
Therefore,
\begin{align*}
 \Big|\sigma_N^4\sum_{k=1}^N\tilde{R}(z_1)_{kk}\tilde{R}(z_2)_{kk}d_k^{(N)} \Big| & \leq \frac{\sigma_N^6}{|\Im z_1\Im z_2|}\sum_{k=1}^N|\tilde{R}(z_1)_{kk}||\tilde{R}(z_2)_{kk}|\sum_{i<k}|\tilde{R}(z_1)_{ii}||\esp[R_N(z_1)_{kk}]|\\
 & \leq \frac{N^2\sigma_N^6}{|\Im z_1|^4|\Im z_2|^2}.
\end{align*}
As a consequence, 
\[ \Big|\sigma_N^4\sum_{k=1}^N\tilde{R}(z_1)_{kk}\tilde{R}(z_2)_{kk}d_k^{(N)} \Big| \underset{N\to +\infty}{\longrightarrow} 0\]
in probability. 

Similarly,
\[ \Big|\sigma_N^4\sum_{k=1}^N\tilde{R}(z_1)_{kk}\tilde{R}(z_2)_{kk}c_k^{(N)}\Big| \leq \frac{N^2\sigma_N^6}{2|\Im z_1|^4|\Im z_2|^2}  \underset{N\to +\infty}{\longrightarrow} 0\]
in probability. 

In order to study the convergence in probability of the first term, we consider its $L^2$ norm. Indeed, remark that, by Minkowski inequality,
\begin{align*}
 \|\sigma_N^4\sum_{k=1}^N\tilde{R}(z_1)_{kk}\tilde{R}(z_2)_{kk}a_k^{(N)}\|_{L^2} & \leq \frac{\sigma_N^4}{|\Im z_1 \Im z_2|}\sum_{k=1}^N\|a_k^{(N)}\|_{L^2}.
 \end{align*}
 Then, again by Minkowski inequality (sum on $i$) and by Jensen inequality applied to $\esp_{\leq k-1}$,
 \begin{align*}
  & \|\sigma_N^4\sum_{k=1}^N\tilde{R}(z_1)_{kk}\tilde{R}(z_2)_{kk}a_k^{(N)}\|_{L^2}\\
 & \leq \frac{\sigma_N^2}{|\Im z_1 \Im z_2|}\sum_{k=1}^N\sum_{i<k}\Big\|\sum_{l \neq k}|(1-\frac{1}{2}\delta_{il})\tilde{R}(z_1)_{ii}W_{il}^2R^{(kil)}(z_1)_{li}\sum_{j<k}R^{(k)}(z_1)_{lj}\esp_{\leq k-1}[R^{(k)}(z_2)_{ji}]|\Big\|_{L^2}.
 \end{align*}
 Use now Cauchy-Schwarz inequality on the $l$-sum to get the following.
 \begin{align*}
 \| \sigma_N^4 \sum_{k=1}^N\tilde{R}(z_1)_{kk} & \tilde{R}(z_2)_{kk}a_k^{(N)}\|_{L^2}\\
 & \leq \frac{\sigma_N^4}{|\Im z_1|^3|\Im z_2|}\sum_{k=1}^N\sum_{i<k}\esp\Big[\sum_{l \neq k}|W_{il}|^4\sum_{l\neq k}\Big|\sum_{j<k}R^{(k)}(z_1)_{lj}\esp_{\leq k-1}[R^{(k)}(z_2)_{ji}]\Big|^2\Big]^{1/2}\\
 & \leq \frac{\sigma_N^4}{|\Im z_1|^4|\Im z_2|^2}\sum_{k=1}^N\sum_{i<k}\esp\Big[\sum_{l \neq k}|W_{il}|^4\Big]^{1/2},
\end{align*}
from Lemma \ref{technicalbound}. Then,
\[ \|\sigma_N^4 \sum_{k=1}^N\tilde{R}(z_1)_{kk}\tilde{R}(z_2)_{kk}a_k^{(N)}\|_{L^2} \leq  \frac{N^2\sigma_N^4}{|\Im z_1|^4|\Im z_2|^2}\big(\delta_N^2s_N^2+(N-2)m_N\big)^{1/2}=O(N^{-1/2}).\]
As a consequence, $\sigma_N^4 \sum_{k=1}^N\tilde{R}(z_1)_{kk}\tilde{R}(z_2)_{kk}a_k^{(N)} \underset{N\to +\infty}{\longrightarrow} 0$ in probability.

We turn now to the second term. Note that $\sum_{j<k}R^{(k)}(z_1)_{ij}\esp_{\leq k-1}[R^{(k)}(z_2)_{ji}]$ does not involve $l$. Therefore,
\begin{align*}
 & \sigma_N^4 \Big|\sum_{k=1}^N  \tilde{R}(z_1)_{kk} \tilde{R}(z_2)_{kk}b_k^{(N)}\Big|\\
 & \leq \frac{\sigma_N^4}{|\Im z_1\Im z_2|}\sum_{k=1}^N \sum_{i<k}\Big|\sum_{l \neq k}(1-\frac{1}{2}\delta_{il})\big(|W_{il}|^2R^{(kil)}(z_1)_{ll}-\sigma_N^2\esp[R_N(z_1)_{ll}]\big)\sum_{j<k}R^{(k)}(z_1)_{ij}\esp_{\leq k-1}[R^{(k)}(z_2)_{ji}]\Big|\\
 & \leq \frac{\sigma_N^4}{|\Im z_1|^2|\Im z_2|^2}\sum_{k=1}^N \sum_{i<k}\Big|\sum_{l \neq k}(1-\frac{1}{2}\delta_{il})\big(|W_{il}|^2R^{(kil)}(z_1)_{ll}-\sigma_N^2\esp[R(z_1)_{ll}]\big)\Big|\\
 & \leq \frac{\sigma_N^4}{|\Im z_1|^2|\Im z_2|^2}\sum_{k=1}^N \sum_{i<k}\bigg(\Big|\sum_{l \neq k}(1-\frac{1}{2}\delta_{il})\big(|W_{il}|^2-\sigma_N^2\big)\esp[R(z_1)_{ll}]\Big|\\
 & \hspace{5cm} + \Big|\sum_{l \neq k}(1-\frac{1}{2}\delta_{il})|W_{il}|^2\big(R^{(kil)}(z_1)_{ll}-\esp[R(z_1)_{ll}]\big)\Big|\bigg).
\end{align*}
The $L^2$ norm of the second part goes to zero as $N$ goes to infinity. Indeed, by Minkowski inequality,
\begin{align*}
\|\frac{\sigma_N^4}{|\Im z_1|^2|\Im z_2|^2}\sum_{k=1}^N\sum_{i<k} & \Big|\sum_{l \neq k}(1-\frac{1}{2}\delta_{il})|W_{il}|^2\big(R^{(kil)}(z_1)_{ll}-\esp[R(z_1)_{ll}]\big)\Big|\|_{L^2}\\
  & \leq  \frac{\sigma_N^4}{|\Im z_1|^2|\Im z_2|^2}\sum_{k=1}^N\sum_{i<k}\sum_{l\neq k} \| |W_{il}|^2\big|R^{(kil)}(z_1)_{ll}-\esp[R(z_1)_{ll}]\big| \|_{L^2}\\
  & \leq \frac{\sigma_N^4}{|\Im z_1|^2|\Im z_2|^2}\sum_{k=1}^N\sum_{i<k}\sum_{l\neq k} \||W_{il}|^2\|_{L^2}\|R^{(kil)}(z_1)_{ll}-\esp[R(z_1)_{ll}]\|_{L^2},
\end{align*}
by independence. Recall now that $\esp\big[\big|R^{(kil)}(z_1)_{ll}-\esp[R(z_1)_{ll}]\big|^2\big]=O(N^{-1})$, uniformly in $i,k,l$. As a consequence,
\[\|\frac{\sigma_N^4}{|\Im z_1|^2|\Im z_2|^2}\sum_{k=1}^N\sum_{i<k} \Big|\sum_{l \neq k}(1-\frac{1}{2}\delta_{il})|W_{il}|^2\big(R^{(kil)}(z_1)_{ll}-\esp[R(z_1)_{ll}]\big)\|_{L^2} =O(N^{-1/2}). \]

 It remains to deal with the first part. We consider its $L^2$ norm. 
 \begin{align*}
 \| \frac{\sigma_N^4}{|\Im z_1|^2|\Im z_2|^2}&\sum_{k=1}^N \sum_{i<k}\big|\sum_{l \neq k}(1-\frac{1}{2}\delta_{il})\big(|W_{il}|^2-\sigma_N^2\big)\esp[R(z_1)_{ll}]\big|\|_{L^2}\\
  & \leq \frac{\sigma_N^4}{|\Im z_1|^2|\Im z_2|^2}\sum_{k=1}^N \sum_{i<k}\|\sum_{l \neq k}(1-\frac{1}{2}\delta_{il})\big(|W_{il}|^2-\sigma_N^2\big)\esp[R(z_1)_{ll}]\|_{L^2}\\
  & \leq \frac{\sigma_N^4}{|\Im z_1|^2|\Im z_2|^2}\sum_{k=1}^N \sum_{i<k}\Big(\sum_{l \neq k}\esp\Big[\big(|W_{il}|^2-\sigma_N^2\big)^2|\esp[R(z_1)_{ll}]|^2\Big]\Big)^{1/2},
 \end{align*}
by independence. Then,
\begin{align*}
  \frac{\sigma_N^4}{|\Im z_1|^2|\Im z_2|^2}\esp\Big[\Big|&\sum_{k=1}^N \sum_{i<k}\Big|\sum_{l \neq k}(1-\frac{1}{2}\delta_{il})\big(|W_{il}|^2-\sigma_N^2\big)\esp[R(z_1)_{ll}]\Big|\Big|^2\Big]^{1/2}\\
  & \leq \frac{N^2\sigma_N^4}{|\Im z_1|^3|\Im z_2|^2}\big(\delta_N^2s_N^2-2\sigma_N^2s_N^2+\sigma_N^4+(N-2)(m_N-\sigma_N^4)\big)^{1/2}=O(N^{-1/2}).
 \end{align*}
 As a consequence, $\sigma_N^4 \sum_{k=1}^N  \tilde{R}(z_1)_{kk} \tilde{R}(z_2)_{kk}b_k^{(N)} \underset{N\to +\infty}{\longrightarrow} 0$ in probability, and
 \[\varepsilon_{k,5}^{(N)} \underset{N\to +\infty}{\longrightarrow} 0\] in probability.
\end{proof}

\begin{proof}[Proof of Lemma \ref{lem_2nd_term_simplification2}]
 Remark that
 \begin{align*}
  \varepsilon_{k,6}^{(N)} & = \varepsilon_{k,5}^{(N)} + \sum_{i,j<k}\tilde{R}(z_1)_{ii}\sum_{l<k}W_{il}\esp_{\leq k-1}[R^{(kil)}(z_1)_{lj}]\esp_{\leq k-1}[R^{(kil)}(z_2)_{ji}]\\
  & + \sum_{i,j<k}\tilde{R}(z_1)_{ii}\sum_{l<k}(1-\frac{1}{2}\delta_{il})W_{il}^2\esp_{\leq k-1}[R^{(kil)}(z_1)_{lj}]\esp_{\leq k-1}[R^{(kil)}(z_2)_{ji}R^{(k)}(z_2)_{li}]\\
  & + \sum_{i,j<k}\tilde{R}(z_1)_{ii}\sum_{l<k}\bigg((1-\frac{1}{2}\delta_{il})|W_{il}|^2\esp_{\leq k-1}[R^{(kil)}(z_1)_{lj}]\esp_{\leq k-1}[R^{(kil)}(z_2)_{jl}R^{(k)}(z_2)_{ii}]\\
  & \hspace{5cm} - \sigma_N^2\esp_{\leq k-1}[R^{(k)}(z_1)_{lj}]\esp_{\leq k-1}[R^{(k)}(z_2)_{jl}]\tilde{R}(z_2)_{ii}\bigg)\\
  & := \varepsilon_{k,5}^{(N)} +e_k^{(N)}+f_k^{(N)}+g_k^{(N)}.
 \end{align*}
We focus on $e_k^{(N)}$. Its $L^2$ norm is bounded by:
\[\frac{1}{|\mathcal{I}z_1|}\sum_{i<k}\|\sum_{l<k}W_{il}\sum_{j<k}\esp_{\leq k-1}[R^{(kil)}(z_1)_{lj}]\esp_{\leq k-1}[R^{(kil)}(z_2)_{ji}]\|_{L^2} .\]
Develop the square of the $l$-sum: the sum of squares is bounded by
\begin{align*}
 \sigma_N^2 \sum_{l<k, l\neq i}\mathbb{E}  \Big[\big|\sum_{j<k}\esp_{\leq k-1}[R^{(kil)}(z_1)_{lj}] & \esp_{\leq k-1}[R^{(kil)}(z_2)_{ji}]\big|^2\Big]\\
 & + \ s_N^2\esp\Big[\big|\sum_{j<k}\esp_{\leq k-1}[R^{(kii)}(z_1)_{ij}]\esp_{\leq k-1}[R^{(kii)}(z_2)_{ji}]\big|^2\Big].
\end{align*}
Recall that $\|R^{(k)}(z)-R^{(kil)}(z)\| \leq 2\delta_N|\Im z|^{-2}$. Replacing $R^{(kil)}$ by $R^{(k)}$ yields the following.
\begin{align*}
 \sum_{l<k}\mathbb{E}\Big[\big|\sum_{j<k} & \esp_{\leq k-1}[R^{(kil)}(z_1)_{lj}]  \esp_{\leq k-1}[R^{(kil)}(z_2)_{ji}]\big|^2\Big]\\
 & \leq 3\sum_{l<k}\mathbb{E}\Big[\big|\sum_{j<k}R^{(k)}(z_1)_{lj}\esp_{\leq k-1}[R^{(k)}(z_2)_{ji}]\big|^2\Big] + O(N\delta_N)=O(N\delta_N),
\end{align*}
uniformly in $k$, using Lemma \ref{technicalbound}. Then the sum of squares is $O(\delta_N)$, uniformly in $k$.

It remains to control the sum of the cross terms $\esp[W_{il}\overline{W_{il'}}\alpha_{il}\overline{\alpha_{il'}}]$, where \[\alpha_{il}=\sum_{j<k}\esp_{\leq k-1}[R^{(kil)}(z_1)_{lj}]\esp_{\leq k-1}[R^{(kil)}(z_2)_{ji}].\]
Following Bai and Silverstein \cite{BaiSilbook}, we replace $R^{(kil')}$ and $R^{(kil)}$ by $R^{(kil'il)}$:
\begin{eqnarray*}
 \alpha_{il}&=&\sum_{j<k}\esp_{\leq k-1}[R^{(kilil')}(z_1)_{lj}]\esp_{\leq k-1}[R^{(kilil')}(z_2)_{ji}]\\
		&& + \sum_{j<k}\esp_{\leq k-1}\big[R^{(kil)}(z_1)_{lj}-R^{(kilil')}(z_1)_{lj}\big]\esp_{\leq k-1}[R^{(kil)}(z_2)_{ji}]\\
		&& + \sum_{j<k}\esp_{\leq k-1}[R^{(kilil')}(z_1)_{lj}]\esp_{\leq k-1}\big[R^{(kil)}(z_2)_{ji}-R^{(kilil')}(z_2)_{ji}\big]\\
		&& := \alpha_{ill'}+\beta_{ill'}+\gamma_{ill'}. 
\end{eqnarray*}

Note that, due to independence properties, the sum vanishes when $\alpha_{il}$ is replaced by $\alpha_{ill'}$ or when $\alpha_{il'}$ is replaced by $\alpha_{ill'}$. It remains to control the four error terms: $\esp[W_{il}\overline{W_{il'}}\beta_{ill'}\overline{\beta_{il'l}}]$, $\esp[W_{il}\overline{W_{il'}}\beta_{ill'}\overline{\gamma_{il'l}}]$, $\esp[W_{il}\overline{W_{il'}}\gamma_{ill'}\overline{\beta_{il'l}}]$ and $\esp[W_{il}\overline{W_{il'}}\gamma_{ill'}\overline{\gamma_{il'l}}]$.
The resolvent identity yields (note that we can remove one $\esp_{\leq k-1}$):
\begin{align*}
   \esp[W_{il} & \overline{W_{il'}}\beta_{ill'}\overline{\beta_{il'l}}]\\
  & = (1-\frac{1}{2}\delta_{il})(1-\frac{1}{2}\delta_{il'})\bigg\{\esp\Big[|W_{il}|^2|W_{il'}|^2R^{(kilil')}(z_1)_{li}\sum_{j<k}R^{(kil)}(z_1)_{l'j}\esp_{\leq k-1}[R^{(kil)}(z_2)_{ji}]\\
  & \hspace{4.5cm} \times \overline{R^{(kil'il)}(z_1)_{l'i}\sum_{j'<k}R^{(kil')}(z_1)_{lj'}\esp_{\leq k-1}[R^{(kil')}(z_2)_{j'i}]}\Big]\\
 & \hspace{2cm} + \esp\Big[W_{il}^2|W_{il'}|^2R^{(kilil')}(z_1)_{li}\sum_{j<k}R^{(kil)}(z_1)_{l'j}\esp_{\leq k-1}[R^{(kil)}(z_2)_{ji}]\\
 & \hspace{4.5cm} \times \overline{R^{(kil'il)}(z_1)_{l'l}\sum_{j'<k}R^{(kil')}(z_1)_{ij'}\esp_{\leq k-1}[R^{(kil')}(z_2)_{j'i}]}\Big]\\
 & \hspace{2cm}  + \esp\Big[|W_{il}|^2\overline{W_{il'}}^2R^{(kilil')}(z_1)_{ll'}\sum_{j<k}R^{(kil)}(z_1)_{ij}\esp_{\leq k-1}[R^{(kil)}(z_2)_{ji}]\\
 & \hspace{4.5cm} \times \overline{R^{(kil'il)}(z_1)_{l'i}\sum_{j'<k}R^{(kil')}(z_1)_{lj'}\esp_{\leq k-1}[R^{(kil')}(z_2)_{j'i}]}\Big]\\
 & \hspace{2cm} + \esp\Big[W_{il}^2\overline{W_{il'}}^2R^{(kilil')}(z_1)_{ll'}\sum_{j<k}R^{(kil)}(z_1)_{ij}\esp_{\leq k-1}[R^{(kil)}(z_2)_{ji}]\\
& \hspace{4.5cm} \times \overline{R^{(kil'il)}(z_1)_{l'l}\sum_{j'<k}R^{(kil')}(z_1)_{ij'}\esp_{\leq k-1}[R^{(kil')}(z_2)_{j'i}]}\Big]\bigg\}.
\end{align*}

The first term can be bounded as follows:
\begin{align*}
\Big|\sum_{l\neq l'}(1-\frac{1}{2}\delta_{il})(1-\frac{1}{2}\delta_{il'}) & \esp\Big[|W_{il}|^2|W_{il'}|^2R^{(kilil')}(z_1)_{li}\sum_{j<k}R^{(kil)}(z_1)_{l'j}\esp_{\leq k-1}[R^{(kil)}(z_2)_{ji}] \\
& \hspace{1cm} \times \overline{R^{(kil'il)}(z_1)_{l'i}\sum_{j'<k}R^{(kil')}(z_1)_{lj'}\esp_{\leq k-1}[R^{(kil')}(z_2)_{j'i}]}\Big]\Big|\\
& \leq \frac{1}{|\mathcal{I}z_1|^2|\Im z_2|^2}\sum_{l\neq l'}\esp[|W_{il}|^2|W_{il'}|^2|R^{(kilil')}(z_1)_{li}||\overline{R^{(kil'il)}(z_1)_{l'i}}|]\\
& \leq \frac{\sigma_N^4}{|\mathcal{I}z_1|^2|\Im z_2|^2}\sum_{l\neq l'}\esp[|R^{(kilil')}(z_1)_{li}||\overline{R^{(kil'il)}(z_1)_{l'i}}|] \ \text{by independence}\\
& \leq \frac{\sigma_N^4}{|\mathcal{I}z_1|^2|\Im z_2|^2}\sum_{l\neq l'}(\esp[|R^{(k)}(z_1)_{li}||\overline{R^{(k)}(z_1)_{l'i}}|]+O(\delta_N))\\
& \leq \frac{\sigma_N^4}{|\mathcal{I}z_1|^2|\Im z_2|^2}\big(\esp\big[\big(\sum_{l}|R^{(k)}(z_1)_{li}|\big)^2\big]+O(\delta_NN^2)\big)\\
& \leq \frac{\sigma_N^4}{|\mathcal{I}z_1|^2|\Im z_2|^2}\big(N\esp\big[\sum_{l}|R^{(k)}(z_1)_{li}|^2\big]+O(\delta_NN^2)\big)\\
& \leq \frac{\sigma_N^4}{|\mathcal{I}z_1|^2|\Im z_2|^2}\big(\frac{N}{|\Im z_1|^2}+O(\delta_NN^2)\big)=O(\delta_N),
\end{align*}
uniformly in $k$.
The three other terms can be treated similarly: they are of order $O(\delta_N)$, uniformly in $k$.

Using again the resolvent identity and very similar computations, the sums over $l$ and $l'$ of the other three error terms $\esp[W_{il}\overline{W_{il'}}\beta_{ill'}\overline{\gamma_{il'l}}]$, $\esp[W_{il}\overline{W_{il'}}\gamma_{ill'}\overline{\beta_{il'l}}]$ and $\esp[W_{il}\overline{W_{il'}}\gamma_{ill'}\overline{\gamma_{il'l}}]$ 
are proved to be of order $O(\delta_N)$, uniformly in $k$.
As a consequence, $\|e_k^{(N)}\|_{L^2} = O(N\delta_N^{1/2})$, uniformly in $k$ and
\[\sigma_N^4\sum_{k=1}^N\tilde{R}(z_1)_{kk}\tilde{R}(z_2)_{kk}e_k^{(N)} \underset{N\to +\infty}{\longrightarrow} 0\]
in probability.

The $L^2$ norm of $\sigma_N^4\sum_{k=1}^N\tilde{R}(z_1)_{kk}\tilde{R}(z_2)_{kk}f_k^{(N)}$ is bounded by
\[\frac{\sigma_N^4}{|\Im z_1|^2|\Im z_2|}\sum_{k=1}^N\sum_{i<k}\mathbb{E}\Big[\Big|\sum_{l<k}(1-\frac{1}{2}\delta_{il})W_{il}^2\sum_{j<k}\mathbb{E}_{\leq k-1}\big[\mathbb{E}_{\leq k-1}[R^{(kil)}(z_1)_{lj}]R^{(kil)}(z_2)_{ji}R^{(k)}(z_2)_{li}\big]\Big|^2\Big]^{1/2},\]
and then, using Jensen inequality (with respect to $\mathbb{E}_{\leq k-1}$) and Cauchy-Schwarz inequality (with respect to the $l$-sum), by
\[\frac{\sigma_N^4}{|\Im z_1|^2|\Im z_2|}\sum_{k=1}^N\sum_{i<k}\mathbb{E}\Big[\sum_{l<k}|W_{il}|^4\Big|\sum_{j<k}\mathbb{E}_{\leq k-1}[R^{(kil)}(z_1)_{lj}]R^{(kil)}(z_2)_{ji}\Big|^2\sum_{l<k}|R^{(k)}(z_2)_{li}|^2\Big]^{1/2}.\]
From Lemma \ref{technicalbound}, one gets:
\[N^2\sigma_N^4\Big(\delta_N^2s_N^2+(N-2)m_N\Big)^{1/2}|\Im z_1|^{-3}|\Im z_2|^{-3}=O(N^{-1/2}).\]
The last term is also negligible. This follows by replacing successively $(1-\frac{1}{2}\delta_{il})|W_{il}|^2$, $R^{(k)}(z_2)_{ii}$, $R^{(kil)}(z_1)_{lj}$, $R^{(kil)}(z_2)_{jl}$ by $\sigma_N^2, \tilde{R}(z_2)_{ii}, R^{(k)}(z_1)_{lj}, R^{(k)}(z_2)_{jl}$.
Indeed, this last term can be written as follows:
\begin{align*}
  & \sigma_N^4\sum_{k=1}^N \tilde{R}(z_1)_{kk}\tilde{R}(z_2)_{kk}g_k^{(N)}\\
  & = \sigma_N^4\sum_{k=1}^N\tilde{R}(z_1)_{kk}\tilde{R}(z_2)_{kk}\sum_{i,j,l<k}\tilde{R}(z_1)_{ii}\\
  & \hspace{2cm} \times \bigg\{\big((1-\frac{1}{2}\delta_{il})|W_{il}|^2-\sigma_N^2\big)\esp_{\leq k-1}[R^{(kil)}(z_1)_{lj}]\esp_{\leq k-1}[R^{(kil)}(z_2)_{jl}R^{(k)}(z_2)_{ii}]\\
 & \hspace{3cm} +\sigma_N^2\esp_{\leq k-1}[R^{(kil)}(z_1)_{lj}]\esp_{\leq k-1}[R^{(kil)}(z_2)_{jl}(R^{(k)}(z_2)_{ii}-\tilde{R}(z_2)_{ii})]\\
 & \hspace{3cm} +\sigma_N^2\esp_{\leq k-1}[R^{(kil)}(z_1)_{lj}-R^{(k)}(z_1)_{lj}]\esp_{\leq k-1}[R^{(k)}(z_2)_{jl}]\tilde{R}(z_2)_{ii}\\
 & \hspace{3cm} +\sigma_N^2\esp_{\leq k-1}[R^{(k)}(z_1)_{lj}]\esp_{\leq k-1}[R^{(kil)}(z_2)_{jl}-R^{(k)}(z_2)_{jl}]\tilde{R}(z_2)_{ii}\bigg\}\\
& := h_k^{(N)}+p_k^{(N)}+q_k^{(N)}+r_k^{(N)}.
\end{align*}

By Minkowski inequality, the $L^2$ norm of term $h_k^{(N)}$ can be bounded as follows:
\begin{align*}
\|h_k^{(N)}\|_2 
   & \leq  \frac{\sigma_N^4}{|\Im z_1|^2|\Im z_2|^2}\sum_{k=1}^N\sum_{i<k}\Big\|\esp_{\leq k-1}\Big[\sum_{l<k}\big((1-\frac{1}{2}\delta_{il})|W_{il}|^2-\sigma_N^2\big)\sum_{j<k}\esp_{\leq k-1}[R^{(kil)}(z_1)_{lj}]R^{(kil)}(z_2)_{jl}\Big]\Big\|_2\\
 & \leq  \frac{\sigma_N^4}{|\Im z_1|^2|\Im z_2|^2}\sum_{k=1}^N\sum_{i<k}\esp\Big[\Big(\sum_{l<k}|(1-\frac{1}{2}\delta_{il})|W_{il}|^2-\sigma_N^2| \Big|\sum_{j<k}\esp_{\leq k-1}[R^{(kil)}(z_1)_{lj}]R^{(kil)}(z_2)_{jl}\Big|\Big)^2\Big]^{1/2}\\
 & \leq  \frac{N^2\sigma_N^4}{|\Im z_1|^3|\Im z_2|^3}\Big(\frac{1}{4}\delta_N^2s_N^2-\sigma_N^2s_N^2+\sigma_N^4+(N-2)(m_N-\sigma_N^4)\Big)^{1/2}\\
 & =O\big(N^{-1/2}\big).
\end{align*}

The term $p_k^{(N)}$ can be bounded as follows (using Lemma \ref{technicalbound}):
\begin{align*}
|p_k^{(N)}| & \leq \frac{\sigma_N^6}{|\Im z_1|^2|\Im z_2|}\sum_{k=1}^N \esp_{\leq k-1}\Big[\sum_{i,l<k}\Big|\sum_{j<k}\esp_{\leq k-1}[R^{(kil)}(z_1)_{lj}]R^{(kil)}(z_2)_{jl}\Big||R^{(k)}(z_2)_{ii}-\tilde{R}(z_2)_{ii}|\Big]\\
   & \leq  \frac{\sigma_N^6}{|\Im z_1|^3|\Im z_2|^2}\sum_{k=1}^N\sum_{i,l<k}\esp_{\leq k-1}\Big[|R^{(k)}(z_2)_{ii}-\tilde{R}(z_2)_{ii}|\Big].
\end{align*}
Since $\esp[|R^{(k)}(z_2)_{ii}-\tilde{R}(z_2)_{ii}|]=O(N^{-1/2})$ uniformly in $i,k$, 
the $L^1$ norm of $p_k^{(N)}$ satisfies:
\[\|p_k^{(N)}\|_{L^1}=O(\sigma_N^6N^3N^{-1/2})=O(N^{-1/2}),\] 
uniformly in $k$.

Using Lemma \ref{technicalbound} and \eqref{removebound}, 
\[\Big|\sum_{j<k}\esp_{\leq k-1}[R^{(kil)}(z_1)_{lj}-R^{(k)}(z_1)_{lj}]R^{(k)}(z_2)_{jl}\Big|\leq \| \esp_{\leq k-1}[R^{(kil)}(z_1)-R^{(k)}(z_1)]\| \| R^{(k)}(z_2)\| \leq \frac{2\delta_N}{|\Im z_1|^2|\Im z_2|}.\]
Hence, the term $q_k^{(N)}$ is bounded by:
\begin{align*}
|q_k^{(N)}|&\leq \frac{\sigma_N^6}{|\Im z_1 \Im z_2|^2}\sum_{k=1}^N\sum_{i,l<k}\esp_{\leq k-1}\Big[\Big|\sum_{j<k}\esp_{\leq k-1}[R^{(kil)}(z_1)_{lj}-R^{(k)}(z_1)_{lj}]R^{(k)}(z_2)_{jl}\Big|\Big]\\
    &\leq \frac{2\delta_N\sigma_N^6N^3}{|\Im z_1|^4|\Im z_2|^3}= O(\delta_N),
\end{align*}
uniformly in $k$.
Similarly, 
\[|r_k^{(N)}|\leq \frac{2\delta_N\sigma_N^6N^3}{|\Im z_1|^3|\Im z_2|^4}=O(\delta_N),\]
uniformly in $k$.
As $\delta_N \underset{N\to +\infty}{\longrightarrow}  0$, we conclude that
\[\varepsilon_{k,6}^{(N)} \underset{N\to +\infty}{\longrightarrow} 0\]
in probability.
\end{proof}

\subsubsection{Contribution of the third term} 
The third term 
 \[ \tau_N^2\sum_{i,j<k}\mathbb{E}_{\leq k-1}[R^{(k)}(z_1)_{ij}]\mathbb{E}_{\leq k-1}[R^{(k)}(z_2)_{ij}]\]
is very similar to the second one  
 \[ \sigma_N^4\sum_{i,j<k}\mathbb{E}_{\leq k-1}[R^{(k)}(z_1)_{ij}]\mathbb{E}_{\leq k-1}[R^{(k)}(z_2)_{ji}].\]
Recall that $\tau_N^2$ and $\sigma_N^4$ are of the same order. Therefore, the only difference lies in the entry of matrix $R^{(k)}(z_2)$ appearing on the right. As a consequence, all computations which were performed on $R^{(k)}(z_1)_{ij}$ are still valid here. Thus
\begin{align*}
 \mathbb{E}_{\leq k-1} [& R^{(k)}(z_1)_{ij}]\mathbb{E}_{\leq k-1}[R^{(k)}(z_2)_{ij}]\\
& =\delta_{ij}\tilde{R}(z_1)_{ii}\mathbb{E}_{\leq k-1}[R^{(k)}(z_2)_{ij}] +\sum_{l<k}\tilde{R}(z_1)_{ii}W_{il}\mathbb{E}_{\leq k-1}[R^{(kil)}(z_1)_{lj}]\mathbb{E}_{\leq k-1}[R^{(k)}(z_2)_{ij}]\\
& + \quad \sum_{l\neq k}(1-\frac{1}{2}\delta_{il})\tilde{R}(z_1)_{ii}\mathbb{E}_{\leq k-1}[W_{il}^2R^{(kil)}(z_1)_{li}R^{(k)}(z_1)_{lj}]\mathbb{E}_{\leq k-1}[R^{(k)}(z_2)_{ij}]\\
& + \quad \mathbb{E}_{\leq k-1}[\sum_{l\neq k}(1-\frac{1}{2}\delta_{il})\tilde{R}(z_1)_{ii}(|W_{il}|^2R^{(kil)}(z_1)_{ll}-\sigma_N^2\mathbb{E}[R(z_1)_{ll}])R^{(k)}(z_1)_{ij}]\mathbb{E}_{\leq k-1}[R^{(k)}(z_2)_{ij}]\\
& - \quad \frac{\sigma_N^2}{2}\tilde{R}(z_1)_{ii}\mathbb{E}[R(z_1)_{ii}]\mathbb{E}_{\leq k-1}[R^{(k)}(z_1)_{ij}]\mathbb{E}_{\leq k-1}[R^{(k)}(z_2)_{ij}]\\
& - \quad \sigma_N^2\tilde{R}(z_1)_{ii}\mathbb{E}[R(z_1)_{kk}]\mathbb{E}_{\leq k-1}[R^{(k)}(z_1)_{ij}]\mathbb{E}_{\leq k-1}[R^{(k)}(z_2)_{ij}].
\end{align*}

Similarly to what was done previously, the last four terms will not asymptotically contribute.
\begin{lem}\label{lem_3rd_term_simplification1}
 \begin{align*}
  \sum_{i,j<k}&\esp_{\leq k-1}[R^{(k)}(z_1)_{ij}]\esp_{\leq k-1}[R^{(k)}(z_2)_{ij}]\\
  & = \sum_{i<k}\tilde{R}(z_1)_{ii}\mathbb{E}_{\leq k-1}[R^{(k)}(z_2)_{ii}] + \sum_{i,j<k}\tilde{R}(z_1)_{ii}\sum_{l<k}W_{il}\mathbb{E}_{\leq k-1}[R^{(kil)}(z_1)_{lj}]\mathbb{E}_{\leq k-1}[R^{(k)}(z_2)_{ij}] + \tilde{\varepsilon}_{k,5}^{(N)},
 \end{align*}
 with $\tau_N^2\sum_{k=1}^N\tilde{R}(z_1)_{kk}\tilde{R}(z_2)_{kk}\tilde{\varepsilon}_{k,5}^{(N)} \underset{N\to +\infty}{\longrightarrow} 0$ in probability.
\end{lem}
The proof of this Lemma is exactly the same as the one of Lemma \ref{lem_2nd_term_simplification1}, by considering that $R^{(k)}(z_2)_{ij}={}^tR^{(k)}(z_2)_{ji}$. We use again the resolvent identity (Lemma \ref{lem_resolvent_identity}) to remove the dependence between $W_{il}$ and $R^{(k)}(z_2)$. Note however that, compared to what was done in Section \ref{sec_2nd_term}, we reverse the roles played by $R^{(kil)}(z_2)$ and $R^{(k)}(z_2)$ to get the following.
\begin{align*}
 \sum_{i,j<k}\tilde{R}(z_1)_{ii} \sum_{l<k} & W_{il}\mathbb{E}_{\leq k-1}[R^{(kil)}(z_1)_{lj}]\mathbb{E}_{\leq k-1}[R^{(k)}(z_2)_{ij}]\\
 & = \sum_{i,j<k}\tilde{R}(z_1)_{ii}\sum_{l<k}W_{il}\mathbb{E}_{\leq k-1}[R^{(kil)}(z_1)_{lj}]\mathbb{E}_{\leq k-1}[R^{(kil)}(z_2)_{ij}]\\
 & + \sum_{i,j<k}\tilde{R}(z_1)_{ii}\sum_{l<k}(1-\frac{1}{2}\delta_{il})W_{il}^2\mathbb{E}_{\leq k-1}[R^{(kil)}(z_1)_{lj}]\mathbb{E}_{\leq k-1}[R^{(k)}(z_2)_{ii}R^{(kil)}(z_2)_{lj}]\\
 & + \sum_{i,j<k}\tilde{R}(z_1)_{ii}\sum_{l<k}(1-\frac{1}{2}\delta_{il})|W_{il}|^2\mathbb{E}_{\leq k-1}[R^{(kil)}(z_1)_{lj}]\mathbb{E}_{\leq k-1}[R^{(k)}(z_2)_{il}R^{(kil)}(z_2)_{ij}].
\end{align*}
Heuristically, consider that $W_{il}^2$ is close to its expectation $\tau_N$, $R^{(k)}(z_2)_{ii}\approx \tilde{R}(z_2)_{ii}$ and replace $R^{(kil)}(z)$ by $R^{(k)}(z)$. Then the second term of the right-hand side becomes 
\[\tau_N\Big(\sum_{i<k}\tilde{R}(z_1)_{ii}\tilde{R}(z_2)_{ii}\Big)\Big(\sum_{l,j<k}\mathbb{E}_{\leq k-1}[R^{(k)}(z_1)_{lj}]\mathbb{E}_{\leq k-1}[R^{(k)}(z_2)_{lj}]\Big).\]

Subtract this quantity to both sides of equation in Lemma \ref{lem_3rd_term_simplification1}. This yields the following Lemma, whose proof is postponed to the end of the paragraph.
\begin{lem}\label{lem_3rd_term_simplification2}
\begin{align*}
 (1-\tau_N\sum_{i<k}\tilde{R}(z_1)_{ii}\tilde{R}(z_2)_{ii})\sum_{i,j<k}\mathbb{E}_{\leq k-1}[R^{(k)}(z_1)_{ij}] & \mathbb{E}_{\leq k-1}[R^{(k)}(z_2)_{ij}]\\
 & = \sum_{i<k}\tilde{R}(z_1)_{ii}\mathbb{E}_{\leq k-1}[R^{(k)}(z_2)_{ii}]+\tilde{\varepsilon}_{k,6}^{(N)},
\end{align*}
with $\tau_N^2\sum_{k=1}^N\tilde{R}(z_1)_{kk}\tilde{R}(z_2)_{kk}\tilde{\varepsilon}_{k,6}^{(N)} \underset{N\to +\infty}{\longrightarrow} 0$ in probability.
\end{lem}

We are ready now to derive the contribution of the whole third term.
\begin{prop}
\begin{align*}
& \tau_N^2\sum_{k=1}^N\tilde{R}(z_1)_{kk}\tilde{R}(z_2)_{kk}\sum_{i,j<k}\mathbb{E}_{\leq k-1}[R^{(k)}(z_1)_{ij}]\mathbb{E}_{\leq k-1}[R^{(k)}(z_2)_{ij}]\\
 & \underset{N\to +\infty}{\longrightarrow} -\log\Big(1-\tau\int_{\mathbb{R}}\frac{\nu_{\infty}(dx)}{(\omega(z_1)-x)(\omega(z_2)-x)}\Big)-\tau\int_{\mathbb{R}}\frac{\nu_{\infty}(dx)}{(\omega(z_1)-x)(\omega(z_2)-x)}.
\end{align*}
\end{prop}
The proof is obtained from the proof of Proposition \ref{prop_2nd_term} by changing $\sigma_N^2$ into $\tau_N$ and $\sigma^2$ into $\tau$ (recall that $\tau_N$ and $\tau$ are supposed to be real numbers). It remains to prove Lemma \ref{lem_3rd_term_simplification2}.
\begin{proof}[Proof of Lemma \ref{lem_3rd_term_simplification2}]
 Remark that
 \begin{align*}
  \tilde{\varepsilon}_{k,6}^{(N)} & = \tilde{\varepsilon}_{k,5}^{(N)} + \sum_{i,j<k}\tilde{R}(z_1)_{ii}\sum_{l<k}W_{il}\esp_{\leq k-1}[R^{(kil)}(z_1)_{lj}]\esp_{\leq k-1}[R^{(kil)}(z_2)_{ij}]\\
  & + \sum_{i,j<k}\tilde{R}(z_1)_{ii}\sum_{l<k}(1-\frac{1}{2}\delta_{il})|W_{il}|^2\esp_{\leq k-1}[R^{(kil)}(z_1)_{lj}]\esp_{\leq k-1}[R^{(k)}(z_2)_{il}R^{(kil)}(z_2)_{ij}]\\
  & + \sum_{i,j<k}\tilde{R}(z_1)_{ii}\sum_{l<k}\bigg((1-\frac{1}{2}\delta_{il})W_{il}^2\esp_{\leq k-1}[R^{(kil)}(z_1)_{lj}]\esp_{\leq k-1}[R^{(k)}(z_2)_{ii}R^{(kil)}(z_2)_{lj}]\\
  & \hspace{5cm} - \tau_N\esp_{\leq k-1}[R^{(k)}(z_1)_{lj}]\esp_{\leq k-1}[R^{(k)}(z_2)_{lj}]\tilde{R}(z_2)_{ii}\bigg)\\
  & := \tilde{\varepsilon}_{k,5}^{(N)} +\tilde{e}_k^{(N)}+\tilde{f}_k^{(N)}+\tilde{g}_k^{(N)}.
 \end{align*}
We proceed as in the proof of Lemma \ref{lem_2nd_term_simplification2}.

We focus on $\tilde{e}_k^{(N)}$. Its $L^2$ norm is bounded by:
\[\frac{1}{|\mathcal{I}z_1|}\sum_{i<k}\|\sum_{l<k}W_{il}\sum_{j<k}\esp_{\leq k-1}[R^{(kil)}(z_1)_{lj}]\esp_{\leq k-1}[R^{(kil)}(z_2)_{ij}]\|_{L^2} .\]
Similarly to what was done for $e_k^{(N)}$ in lemma \ref{lem_2nd_term_simplification2}, we develop the square of the $l$-sum. The sum of squares is bounded by $O(N\delta_N)$.

In order to control the sum of the cross terms $\esp[W_{il}\overline{W_{il'}}\tilde{\alpha}_{il}\overline{\tilde{\alpha}_{il'}}]$, where \[\tilde{\alpha}_{il}=\sum_{j<k}\esp_{\leq k-1}[R^{(kil)}(z_1)_{lj}]\esp_{\leq k-1}[R^{(kil)}(z_2)_{ij}],\] 
we replace $R^{(kil')}$ and $R^{(kil)}$ by $R^{(kil'il)}$:

\begin{align*}
 \tilde{\alpha}_{il} & = \sum_{j<k}\esp_{\leq k-1}[R^{(kilil')}(z_1)_{lj}]\esp_{\leq k-1}[R^{(kilil')}(z_2)_{ij}]\\
 & + \sum_{j<k}\esp_{\leq k-1}\big[R^{(kil)}(z_1)_{lj}-R^{(kilil')}(z_1)_{lj}\big]\esp_{\leq k-1}[R^{(kil)}(z_2)_{ij}]\\
 & + \sum_{j<k}\esp_{\leq k-1}[R^{(kilil')}(z_1)_{lj}]\esp_{\leq k-1}\big[R^{(kil)}(z_2)_{ij}-R^{(kilil')}(z_2)_{ij}\big]\\
 & := \tilde{\alpha}_{ill'}+\tilde{\beta}_{ill'}+\tilde{\gamma}_{ill'}. 
\end{align*}

Note that, due to independence properties, the sum vanishes when $\tilde{\alpha}_{il}$ is replaced by $\tilde{\alpha}_{ill'}$ or when $\tilde{\alpha}_{il'}$ is replaced by $\tilde{\alpha}_{ill'}$. 

Using again the resolvent identity and very similar computations, the sums over $l$ and $l'$ of the four error terms $\esp[W_{il}\overline{W_{il'}}\tilde{\beta}_{ill'}\overline{\tilde{\beta}_{il'l}}]$, $\esp[W_{il}\overline{W_{il'}}\tilde{\beta}_{ill'}\overline{\tilde{\gamma}_{il'l}}]$, $\esp[W_{il}\overline{W_{il'}}\tilde{\gamma}_{ill'}\overline{\tilde{\beta}_{il'l}}]$ and $\esp[W_{il}\overline{W_{il'}}\tilde{\gamma}_{ill'}\overline{\tilde{\gamma}_{il'l}}]$ 
are proved to be of order $O(\delta_N)$, uniformly in $k$.
As a consequence, $\|\tilde{e}_k^{(N)}\|_{L^2} = O(N\delta_N^{1/2})$, uniformly in $k$ and
\[\tau_N^2\sum_{k=1}^N\tilde{R}(z_1)_{kk}\tilde{R}(z_2)_{kk}\tilde{e}_k^{(N)} \underset{N\to +\infty}{\longrightarrow} 0\]
in probability.

As in the proof of Lemma \ref{lem_2nd_term_simplification2}, the $L^2$ norm of $\tau_N^2\sum_{k=1}^N\tilde{R}(z_1)_{kk}\tilde{R}(z_2)_{kk}\tilde{f}_k^{(N)}$ is $O(N^{-1/2})$.

The last term is also negligible. This follows by replacing successively $(1-\frac{1}{2}\delta_{il})W_{il}^2$, $R^{(k)}(z_2)_{ii}$, $R^{(kil)}(z_1)_{lj}$, $R^{(kil)}(z_2)_{lj}$ by $\tau_N$, $\tilde{R}(z_2)_{ii}$, $R^{(k)}(z_1)_{lj}$ and $R^{(k)}(z_2)_{lj}$.
Indeed, this last term can be written as follows:
\begin{align*}
& \tau_N^2\sum_{k=1}^N\tilde{R}(z_1)_{kk}\tilde{R}(z_2)_{kk}\tilde{g}_k^{(N)}\\
& = \tau_N^2\sum_{k=1}^N\tilde{R}(z_1)_{kk}\tilde{R}(z_2)_{kk}\sum_{i,j,l<k}\tilde{R}(z_1)_{ii}\\
& \hspace{3cm} \times \bigg\{\big[\big((1-\frac{1}{2}\delta_{il}\big)W_{il}^2-\tau_N)\esp_{\leq k-1}[R^{(kil)}(z_1)_{lj}]\esp_{\leq k-1}\big[R^{(k)}(z_2)_{ii}R^{(kil)}(z_2)_{lj}\big]\\
& \hspace{4cm} +\tau_N\esp_{\leq k-1}[R^{(kil)}(z_1)_{lj}]\esp_{\leq k-1}\big[(R^{(k)}(z_2)_{ii}-\tilde{R}(z_2)_{ii})R^{(kil)}(z_2)_{lj}\big]\\
& \hspace{4cm} +\tau_N\esp_{\leq k-1}[R^{(kil)}(z_1)_{lj}-R^{(k)}(z_1)_{lj}]\tilde{R}(z_2)_{ii}\esp_{\leq k-1}[R^{(kil)}(z_2)_{lj}]\bigg\}\\
& \hspace{4cm} +\tau_N\esp_{\leq k-1}[R^{(k)}(z_1)_{lj}]\tilde{R}(z_2)_{ii}\esp_{\leq k-1}[R^{(kil)}(z_2)_{lj}-R^{(k)}(z_2)_{lj}]\bigg\}\\
& := \tilde{h}_k^{(N)}+ \tilde{p}_k^{(N)} + \tilde{q}_k^{(N)} +\tilde{r}_k^{(N)}.
\end{align*}
These four terms are proved to be negligible, as in Lemma \ref{lem_2nd_term_simplification2}. 

As a consequence, $\tau_N^2\sum_{k=1}^N\tilde{R}(z_1)_{kk}\tilde{R}(z_2)_{kk}\tilde{\varepsilon}_{k,6}^{(N)} \underset{N\to +\infty}{\longrightarrow} 0$ in probability.
\end{proof}

\subsubsection{Contribution of the fourth term} 
To handle the fourth term, we use that, because of Jensen inequality, 
\[\esp[|\mathbb{E}_{\leq k-1}[R^{(k)}(z)_{ii}]-\tilde{R}(z)_{ii}|^2]=o(1),\]
uniformly in $i,k$, so that 
\begin{align*}
\kappa_N\sum_{1\leq i<k\leq N}\tilde{R}(z_1)_{kk}& \tilde{R}(z_2)_{kk} \mathbb{E}_{\leq k-1}[R^{(k)}(z_1)_{ii}]\mathbb{E}_{\leq k-1}[R^{(k)}(z_2)_{ii}]\\
& =\kappa_N\sum_{1\leq i<k\leq N}\tilde{R}(z_1)_{kk}\tilde{R}(z_2)_{kk}\tilde{R}(z_1)_{ii}\tilde{R}(z_2)_{ii}+o(1)\\
& =\frac{\kappa_N}{2}\Big(\sum_{1\leq i\leq N}\tilde{R}(z_1)_{ii}\tilde{R}(z_2)_{ii}\Big)^2+o(1)\\
& \underset{N\to +\infty}{\longrightarrow} \frac{\kappa}{2}\Big(\int_{\mathbb{R}}\frac{\nu_{\infty}(dx)}{(\omega(z_1)-x)(\omega(z_2)-x)}\Big)^2.
\end{align*}

\subsubsection{Conclusion}\label{sec:CV_hook_ccl}
From previous computations, we know that 
\begin{equation*}
\gamma_N(z_1,z_2)=\sum_{k=1}^N\mathbb{E}_{\leq k-1}\Big[\mathbb{E}_{\leq k}\big[\phi_k^{(N)}(z_1)\big]\mathbb{E}_{\leq k}\big[\phi_k^{(N)}(z_2)\big]\Big] \underset{N\to +\infty}{\longrightarrow} \gamma(z_1,z_2)
\end{equation*}
in probability, where
\begin{align*}
\gamma(z_1,z_2) & := (s^2-\sigma^2-\tau)\int_{\mathbb{R}}\frac{\nu_{\infty}(dx)}{(\omega(z_1)-x)(\omega(z_2)-x)}+\frac{\kappa}{2}\Big(\int_{\mathbb{R}}\frac{\nu_{\infty}(dx)}{(\omega(z_1)-x)(\omega(z_2)-x)}\Big)^2\\
& -\log\Big(1-\sigma^2\int_{\mathbb{R}}\frac{\nu_{\infty}(dx)}{(\omega(z_1)-x)(\omega(z_2)-x)}\Big)-\log\Big(1-\tau\int_{\mathbb{R}}\frac{\nu_{\infty}(dx)}{(\omega(z_1)-x)(\omega(z_2)-x)}\Big).
\end{align*}
Recall that a sequence $(X_N)_{N\geq 1}$ of random variables converges in probability to a random variable $X$ 
if and only if, from any subsequence extracted from $(X_N)_{N\geq 1}$, 
one can further extract a subsubsequence almost surely converging to $X$. 
We will use this criterion twice in the following argument. 
First, we deduce by diagonal extraction from the convergence in probability above 
that, given a countable subset of $(\C\setminus\R)^2$, 
one can extract from any subsequence of $\gamma_N$ a subsubsequence almost surely
converging to $\gamma$ pointwise on this subset. 
Second, we will use it to reduce the proof of convergence in probability of the sequence  $(\Gamma_N(z_1,z_2))_{N\geq 1}$ 
to $\Gamma(z_1,z_2)$ for a fixed $(z_1,z_2) \in (\C\setminus\R)^2$ to the proof of the almost sure convergence of some subsequence, which is achieved by a normal family argument and analytic continuation principle.

\section{Fluctuations of linear spectral statistics of deformed Wigner matrices}\label{extension}

In the preceding section, a CLT has been established for $\mathcal{N}_N[\varphi]$, 
when $\varphi \in \mathcal{L}_1$. Recall that $\mathcal{L}_1$ is dense in $\mathcal{H}_s$ for any $s>0$. 
For $\varphi \in \mathcal{H}_s$, following \cite{Shcherbina11}, 
$V_N[\varphi]:=\Var[\mathcal{N}_N(\varphi)]$ satisfies:
\begin{align*}
V_N[\varphi] &\leq C(s)\|\varphi \|_{\mathcal{H}_s}^2\int_0^{+\infty}y^{2s-1}e^{-y}\int_{\mathbb{R}}\Var[\Tr(R_N(x+iy)]dxdy\\
		&\leq C(s)\|\varphi \|_{\mathcal{H}_s}^2\int_0^{+\infty}y^{2s-1}e^{-y}\int_{\mathbb{R}}2\bigg(s_N^2y^{-3+\delta}+2y^{-3-\delta}\frac{m_N(N\sigma_N^2)^{\delta}}{\sigma_N^2}\bigg)\sum_{k=1}^N\mathbb{E}[|R_N(x+iy)_{kk}|^{1+\delta}]dxdy\\
		&\leq C\|\varphi \|_{\mathcal{H}_s}^2,
\end{align*}
where $C < +\infty$ as soon as $\int_0^{+\infty}(y^{2s-4}+y^{2s-4-2\delta})e^{-y}dy < +\infty$. This holds when $s>3/2+\delta$. Therefore, a CLT holds for $\mathcal{N}_N[\varphi]$ when $\varphi $ belongs to the set $\mathcal{L}$ of real-valued functions in $\mathcal{H}_{s},\, s>3/2$, as a consequence of the following extension Lemma due to Shcherbina:
\begin{thm}\cite{Shcherbina11}
Let $(\mathcal{L},\| \, \|)$ be a normed vector space. Assume that:
\begin{itemize}
\item there exists $C>0$ such that for any $\varphi \in \mathcal{L}$ and large enough $N\geq 1$, one has 
\[ V_N[\varphi] \leqslant C\|\varphi\|^2,\]
\item there exists a dense linear subspace $\mathcal{L}_1 \subset \mathcal{L}$ such that a CLT is valid for $\mathcal{N}_N(\varphi)$ for all $\varphi \in \mathcal{L}_1$: there exists a continuous quadratic function $V: \mathcal{L}_1 \to \R$ such that, for all $\varphi \in \mathcal{L}_1$,
\[ \mathcal{N}_N(\varphi)-\esp[\mathcal{N}_N(\varphi)] \Rightarrow \mathcal{N}(0,V[\varphi]).\]
\end{itemize}
Then $V$ admits a unique continuous extension to $\mathcal{L}$ and 
\[ \mathcal{N}_N(\varphi)-\esp[\mathcal{N}_N(\varphi)] \Rightarrow \mathcal{N}(0,V[\varphi])\]
holds for all $\varphi \in \mathcal{L}$.
\end{thm}

\appendix

\section{Truncation and centering}\label{Truncation}
In Sections \ref{subsectionconvb}, \ref{subsectionreduction} and \ref{subsectionhook}, the convergence of bias and the fluctuations of linear spectral statistics for test functions in $\mathcal{L}_1$ were studied under the hypothesis that the entries of $W_N$ are bounded by a sequence $\delta_N$ slowly converging to $0$. Note that this assumption was not needed to extend these results to more general functions. Therefore, this appendix deals only with smooth enough functions.

For a bounded Lipschitz continuous function $\varphi: \R \to \C$, let 
$$\mathcal{N}_N(\varphi):=\Tr(\varphi(X_N))=\sum_{\lambda \in \text{sp}(X_N)}\varphi(\lambda)=N\int_{\mathbb{R}}\varphi(x)\mu_N(dx).$$
In this section, we truncate and center the entries of $W_N$, in order to show that it is sufficient to study the fluctuations of $\mathcal{N}_N(\varphi)$ for matrices $W_N$ with entries bounded by $\delta_N$, where $(\delta_N)_{N\geq 1}$ is a sequence of positive real numbers such that $\delta_N \underset{N\to +\infty}{\longrightarrow} 0$ at rate less than $N^{-\eta}$ for any $\eta >0$.

Define $\hat{X}_N=\hat{W}_N+D_N$ and $\mathring{X}_N=\mathring{W}_N+D_N$ by 
$$\hat{W}_{ij}:=W_{ij}\mathbf{1}_{|W_{ij}|\leqslant \frac{\delta_N}{2}},\quad 1\leqslant i, j\leqslant N,$$
\[\mathring{W}_{ij}:=\hat{W}_{ij}-\esp[\hat{W}_{ij}],\quad 1\leqslant i \neq j\leqslant N,\]
and accordingly 
$$\mathring{\mathcal{N}}_N(\varphi):=\Tr(\varphi(\mathring{X}_N)).$$
Note that the entries of $\mathring{W}$ are centred and bounded by $\delta_N$. Furthermore, the off-diagonal entries are independent and identically distributed, as well as entries on the diagonal.
Observe that, for $i\neq j$, 
$$|\mathbb{E}[\hat{W}_{ij}]|=|\mathbb{E}[W_{ij}\mathbf{1}_{|W_{ij}|\leq \frac{\delta_N}{2}}]|=|\mathbb{E}[W_{ij}\mathbf{1}_{|W_{ij}|>\frac{\delta_N}{2}}]|=O(\delta_N^{-3-4\varepsilon}N^{-2(1+\varepsilon)}).$$
Furthermore, $\Var \mathring{W}_{ij}=\sigma_N^2-\esp[|W_{ij}|^2\mathbbm{1}_{|W_{ij}|>\frac{\delta_N}{2}}]-|\esp[\hat{W}_{ij}]|^2$, and $\esp[|W_{ij}|^2\mathbbm{1}_{|W_{ij}|>\frac{\delta_N}{2}}]=O(N^{-2(1+\varepsilon)}\delta_N^{-2-4\varepsilon})$. Therefore, $\Var \mathring{W}_{ij}=\mathring{\sigma}_N^2$ with $\mathring{\sigma}^2_N=\sigma_N^2+O(\delta_N^{-3-4\varepsilon}N^{-2(1+\varepsilon)})$. Note that $N\mathring{\sigma}_N^2 \underset{N\to +\infty}{\longrightarrow} \sigma^2$.
Similar bounds may be proved for $\esp[|\mathring{W}_{ij}|^4]$, $\esp[|\mathring{W}_{ij}|^{4(1+\varepsilon)}]$, $\esp[\mathring{W}_{ii}^2]$ and $\esp[|\mathring{W}_{ii}|^{2(1+\varepsilon)}]$ from which it may be shown that the entries of $\mathring{W}$ satisfy the same properties as the ones of $W_N$. In particular, one has $N\esp[\mathring{W}_{ii}^2]=N\mathring{s}^2_N \underset{N\to +\infty}{\longrightarrow} s^2$, and for $i \neq j$ $N\esp[\mathring{W}_{ij}^2]=N\mathring{\tau}_N \underset{N\to +\infty}{\longrightarrow} \tau$, $N^2(\esp[|\mathring{W}_{ij}|^4]-2\mathring{\sigma}_N^4-\mathring{\tau}_N^2)=N^2\mathring{\kappa}_N\underset{N\to +\infty}{\longrightarrow}\kappa $.

By Lemma \ref{lipschitz}, 
 \begin{align*}
\esp\big[|\mathring{\mathcal{N}}_N(\varphi)-\mathcal{N}_N(\varphi)|\big]&\leqslant \esp\big[|\Tr(\varphi(\mathring{X}_N))-\Tr(\varphi(X_N))|\big]\\
 & \leqslant \|\varphi\|_{\text{Lip}}\sum_{i,j=1}^N \esp[|W_{ij}-\mathring{W}_{ij}|]\\
  &\leqslant 2\|\varphi\|_{\text{Lip}}\sum_{i,j=1}^N \esp[|W_{ij}|\mathbf{1}_{|W_{ij}|\geqslant \frac{\delta_N}{2}}]\\
 &\leqslant \|\varphi\|_{\text{Lip}}\big(\delta_N^{-3-4\varepsilon}\sum_{i\neq j} \esp[|W_{ij}|^{4(1+\varepsilon)}] +\delta_N^{-1-2\varepsilon}\sum_{i=1}^N \esp[|W_{ii}|^{2(1+\varepsilon)}] \big)\\
 & \leqslant \|\varphi\|_{\text{Lip}}\big(O(\delta_N^{-3-4\varepsilon}N^{-2\varepsilon})+O(\delta_N^{-1-2\varepsilon}N^{-2\varepsilon})\big)=o(1).
\end{align*}

Therefore, assuming that $\mathring{\mathcal{N}}_N(\varphi)-\esp[\mathring{\mathcal{N}}_N(\varphi)]$ converges to a Gaussian variable yields that $\mathcal{N}_N(\varphi)-\esp[\mathcal{N}_N(\varphi)]$ converges to the same Gaussian variable. Furthermore, if $\esp[\mathring{\mathcal{N}}_N(\varphi)]-N\int_{\R}\varphi d\rho_N$ converges to some limit $b(\varphi)$, the same will hold for $\esp[\mathcal{N}_N(\varphi)]-N\int_{\R}\varphi d\rho_N$.

Therefore, for our purposes, we may suppose that the entries of $W_N$ are bounded almost surely by $\delta_N$.
 
\section{Asymptotic infinitesimal freeness of GUE and deterministic matrices}

Observe that the bias $b(\varphi)$ in Theorem \ref{main} vanishes whenever $s^2=\sigma^2, \tau=0$ and $\kappa=0$ (this is for instance the case for the deformed GUE). In that case, the mean empirical spectral measure $\mathbb{E}[\mu_N]$ of $W_N+D_N$ is particularly well approximated by the free additive convolution $\rho_N=\mu_{N\sigma_N^2}\boxplus \nu_N$ of a semicircular distribution and $\nu_N$. It is then natural to ask whether a multivariate generalization of this result holds: in the noncommutative probability space $(\mathcal{M}_N(\mathbb{C})\otimes L^{\infty-},\mathbb{E}N^{-1}\Tr )$, let $(W_1, \ldots , W_k)$ be a $k$-tuple of independent copies of $W_N$ with entries (having finite moments of any order) satisfying \ref{hyp:indep}, \ref{hyp:offdiagonal} and \ref{hyp:diagonal} with $s^2=\sigma^2, \tau=0$ and $\kappa=0$, and $(A_1, \ldots , A_l)$ be a $l$-tuple of bounded deterministic matrices with $*$-distribution $\nu_N$; is the $*$-distribution $\xi_N$ of the $(k+l)$-tuple $(W_1, \ldots , W_k, A_1, \ldots , A_l)$ well approximated by the free product $\mu_{N\sigma_N^2}^{\star k}\star \nu_N$ of a semicircular family and $\nu_N$? This was proved for independent standard GUE matrices in the absence of deterministic matrices in \cite{Thorbjornsen00}, and, formulated as an asymptotic infinitesimal freeness result, with finite rank deterministic matrices in \cite{Shlyakhtenko18}. This holds in general; for simplicity, we restrict ourselves to GUE matrices. In other words, we prove asymptotic infinitesimal freeness for independent GUE matrices and a tuple of bounded deterministic matrices converging in $*$-distribution.

\begin{thm} \label{asympinffree}
Let $(W_1, \ldots , W_k)$ be a $k$-tuple of independent GUE$(N,\sigma_N^2)$ matrices (with the assumption $\sigma_N^2=O(N^{-1})$) and $(A_1, \ldots , A_l)$ be a $l$-tuple of deterministic matrices with $*$-distribution $\nu_N$ in the noncommutative probability space $(\mathcal{M}_N(\mathbb{C})\otimes L^{\infty-},\mathbb{E}N^{-1}\Tr )$. Assume that $\sup_{N\geq 1}\|A_i\|<+\infty,\, 1\leq i\leq l$. Then the $*$-distribution $\xi_N$ of the $(k+l)$-tuple $(W_1, \ldots , W_k, A_1, \ldots , A_l)$ satisfies: for $P\in \mathbb{C}\langle X_1,\ldots ,X_{k+l}\rangle$,
$$\xi_N(P)=(\mu_{N\sigma_N^2}^{\star k}\star \nu_N)(P)+O(N^{-2}).$$
\end{thm}

\begin{proof}
By linearity, it is sufficient to prove the statement for monomials $P$ and for Hermitian $A_1, \ldots , A_l$. Our proof relies on the combinatorics of free probability, as exposed in Lecture 22 of \cite{NicSpebook}; we follow their notations. 
In particular, $NC(n)$ denotes the lattice of non-crossing partitions, $(\kappa_n)_{n\geq 1}$ the sequence of free cumulant functionals, $K$ the Kreweras complementation map, $\gamma$ the cyclic permutation $(1, \ldots ,n)$. 
Let $(\mathcal{A},\varphi)$ be a noncommutative probability space, $w_1, \ldots , w_k\in \mathcal{A}$ free semicircular elements with variance $N\sigma_N^2$ and $(a_1, \ldots , a_l)\in \mathcal{A}^l$ a $l$-tuple of selfadjoint noncommutative random variables with noncommutative distribution $\nu_N$, free from $\{w_1, \ldots , w_k\}$. It is sufficient to consider non-constant monomial $P$ in $w_1, \ldots , w_k, a_1, \ldots , a_l$ (resp. $W_1, \ldots , W_k, A_1, \ldots , A_l$) of the form $x^1a^1\cdots x^na^n$ (resp. $X^1A^1\cdots X^nA^n$) with $n\geq1$, $x^1=w_{j_1}, \ldots , x^n=w_{j_n}$ (resp. $X^1=W_{j_1}, \ldots , X^n=W_{j_n}$) and $a^1,\ldots ,a^n$ in the multiplicative semigroup generated by $\{a_1, \ldots , a_l\}$ (resp. $A^1,\ldots ,A^n$ in the multiplicative semigroup generated by $\{A_1, \ldots , A_l\}$). For such a monomial $P$, on the one hand, 
\begin{eqnarray*}
(\mu_{N\sigma_N^2}^{\star k}\star \nu_N)(P)&=&\varphi(x^1a^1\cdots x^na^n)\\
					   &=&\sum_{\pi \in NC(n)}\kappa_{\pi}(x^1, \ldots ,x^n)\varphi_{K(\pi)}(a^1, \ldots , a^n)\\
					   &=&\sum_{\pi \in NC_2^{(j)}(n)}(N\sigma_N^2)^{n/2}\varphi_{K(\pi)}(a^1, \ldots , a^n),
\end{eqnarray*}
where $NC_2^{(j)}(n)$ is defined page 376 of \cite{NicSpebook}.
On the other hand, still with the notations of Lecture 22 of \cite{NicSpebook},
\begin{eqnarray*}
\xi_N(P)&=&\mathbb{E}[N^{-1}\Tr (X^1A^1\cdots X^nA^n)]\\
	  &=&N^{-1}\sum_{i_1,\ldots ,i_{2n}=1}^N\mathbb{E}[X_{i_1i_2}^1\cdots X_{i_{2n-1}i_{2n}}^n]A_{i_2i_3}^1\cdots A_{i_{2n}i_1}^n\\
	  &=&N^{-1}\sum_{i_1,\ldots ,i_{2n}=1}^N\sum_{\pi \in \mathcal{P}_2(n)}\mathbb{E}_{\pi}[X_{i_1i_2}^1,\ldots , X_{i_{2n-1}i_{2n}}^n]A_{i_2i_3}^1\cdots A_{i_{2n}i_1}^n,
\end{eqnarray*}
using Wick formula for the centred complex Gaussian process $(X_{i_1i_2}^1,\ldots ,X_{i_{2n-1}i_{2n}}^n)$. Denote $e_t:=(i_{2t-1},i_{2t}), e_t^*:=(i_{2t},i_{2t-1})$ and observe that 
$$\mathbb{E}_{\pi}[X_{i_1i_2}^1,\ldots , X_{i_{2n-1}i_{2n}}^n]=\left\{
	\begin{array}{ll}
	\sigma_N^n &\text{ if } n \text{ is even, } \pi \in \mathcal{P}_2^{(j)}(n) \text{ and } e_s=e_t^*,\quad \forall \{s,t\}\in \pi,\\
	0 &  \textrm{ otherwise. }
	\end{array}\right.$$
Hence, exchanging sums,
\begin{eqnarray*}
\xi_N(P)&=&N^{-1}\sigma_N^n\sum_{\pi \in \mathcal{P}_2^{(j)}(n)}\Tr_{\pi \gamma}(A^1, \ldots , A^n)\\
	  &=&\sum_{\pi \in \mathcal{P}_2^{(j)}(n)}N^{-n/2-1+|\pi \gamma|}(N\sigma_N^2)^{n/2}(N^{-1}\Tr)_{\pi \gamma}(A^1, \ldots , A^n)\\
	  &=&\sum_{\pi \in NC_2^{(j)}(n)}(N\sigma_N^2)^{n/2}(N^{-1}\Tr)_{K(\pi)}(A^1, \ldots , A^n)+O(N^{-2})\\
	  &=&(\mu_{N\sigma_N^2}^{\star k}\star \nu_N)(P)+O(N^{-2}),
\end{eqnarray*}
where we have used the fact that $n/2+1-|\pi \gamma|$ is an even nonnegative integer vanishing if and only if the pairing $\pi $ is non-crossing (see \cite{Thorbjornsen00}).
\end{proof}

\subsection*{Acknowledgments}
We are glad to thank the GDR MEGA for partial support.

\def\cprime{$'$}

\end{document}